\documentclass[11pt]{article}

\usepackage[utf8]{inputenc}
\usepackage{fullpage,graphicx,dsfont,amsfonts,amssymb,amsmath,psfrag,enumerate,float,amsthm,enumitem,mathtools,setspace,natbib,bbding,subcaption,tcolorbox,comment,ifthen,xstring,comment,xcolor}

\usepackage{breakcites}

\allowdisplaybreaks
\mathtoolsset{showonlyrefs}

\usepackage[colorlinks=true,bookmarks=true,linkcolor=blue,urlcolor=blue,citecolor=blue,breaklinks=true]{hyperref}

\usepackage{sectsty}
\makeatletter\def\@seccntformat#1{\protect\makebox[0pt][r]{\csname the#1\endcsname\hspace{12pt}}}\makeatother

\usepackage{algorithm,algpseudocode}
\usepackage{etoolbox}
\AtBeginEnvironment{algorithmic}{\setlength{\baselineskip}{1.5em}}

\newboolean{show_proofs}
\setboolean{show_proofs}{true}

\ifthenelse{\boolean{show_proofs}}
{
}
{
	\excludecomment{proof}
}
\definecolor{ddarkbrown}{rgb}{0.5,0.2,0.05} \definecolor{bbluegray}{rgb}{0.05,0,0.5}

\newtheorem{theorem}{Theorem}[section]
\newtheorem{proposition}[theorem]{Proposition}
\newtheorem{definition}[theorem]{Definition}
\newtheorem{lemma}[theorem]{Lemma}

\newtheorem{assumption}{Assumption}[section]
\newtheorem{example}[theorem]{Example}

\newtheorem{remark}[theorem]{Remark}

\newcommand{\xbar}{\overline{x}}

\let\la\langle
\let\ra\rangle
\newcommand{\inner}[2]{\left\langle{#1},{#2}\right\rangle}
\newcommand{\BSP}{\begin{equation*}\begin{split}}
      \newcommand{\ESP}{\end{split}\end{equation*}}

\newcommand{\BEAS}{\begin{eqnarray*}}
    \newcommand{\EEAS}{\end{eqnarray*}}
\newcommand{\BEA}{\begin{align*}}
    \newcommand{\EEA}{\end{align*}}
\newcommand{\BEQ}{\begin{equation}}
    \newcommand{\EEQ}{\end{equation}}

\newcommand{\BEQS}{\begin{equation*}}
    \newcommand{\EEQS}{\end{equation*}}

\newcommand{\BIT}{\begin{itemize}}
    \newcommand{\EIT}{\end{itemize}}
\newcommand{\BNUM}{\begin{enumerate}}
    \newcommand{\ENUM}{\end{enumerate}}

\newcommand{\BA}{\begin{array}}
    \newcommand{\EA}{\end{array}}

\newcommand{\reals}{{\mathbb R}}
\newcommand{\naturals}{{\mathbb N}}

\newcommand{\symm}{\mathbb{S}}  % symmetric matrices
\newcommand{\Rd}{\reals^d}
\newcommand{\Rdd}{\reals^{d \times d}}

\newcommand{\diag}{\mathop{\bf diag}}

\newcommand{\Prob}{\mathbb{P}}

\newcommand{\stopcrit}{\mathtt{stop}}
\newcommand{\stopk}[1]{\mathtt{stop}_{#1}}
\newcommand{\stopoob}{\mathtt{boundary}}

\newcommand{\stopres}{\mathtt{residual}}

\newcommand{\tCGbg}{\mathtt{tCG\text{-}bg}}
\newcommand{\boundarygradientstep}{\mathtt{boundary\_gradient\_step}}

\newcommand{\PRP}[1]{{\color{purple} #1}} % modifications in algorithms

\newcommand{\flow}{f_{\mathrm{low}}}
\newcommand{\qmin}{q_{\mathrm{min}}}
\newcommand{\Tin}{T_{\mathrm{in}}}

\newcommand{\Flg}{ \bar{\mathcal{F}} }

\newcommand{\bKrs}{\overline{K}_{r,s}}
\newcommand{\bKru}{\overline{K}_{r,u}}

\newcommand{\Deltacrit}{\Delta_{\rm crit}}
\newcommand{\Deltainf}{\Delta_{\rm inf}}

\newcommand{\opnorm}[1]{\|{#1}\|_{\mathrm{op}}}
\newcommand{\lambdamin}{\lambda_{\mathrm{min}}}

\newcommand{\dt}{\mathrm{d}t}

\title{A practical randomized trust-region method\\to escape saddle points in high dimension}

\author{Radu-Alexandru Dragomir, Xiaowen Jiang, Bonan Sun, Nicolas Boumal}

\date{Compiled \today}

\begin{document}

\maketitle

\begin{abstract}
Without randomization, escaping the saddle points of $f \colon \reals^d \to \reals$ requires at least $\Omega(d)$ pieces of information about $f$ (values, gradients, Hessian-vector products).
With randomization, this can be reduced to a polylogarithmic dependence in $d$.
The prototypical algorithm to that effect is perturbed gradient descent (PGD): through sustained jitter, it reliably escapes strict saddle points.
However, it also never settles: there is no convergence.
What is more, PGD requires precise tuning based on Lipschitz constants and a preset target accuracy.

To improve on this, we modify the time-tested trust-region method with truncated conjugate gradients (TR-tCG).
Specifically, we randomize the initialization of tCG (the subproblem solver),
and we prove that tCG automatically amplifies the randomization near saddles (to escape) and absorbs it near local minimizers (to converge).
Saddle escape happens over several iterations.
Accordingly, our analysis is multi-step, with several novelties.

The proposed algorithm is practical: it essentially tracks the good behavior of TR-tCG, with three minute modifications and a single new hyperparameter (the noise scale $\sigma$).
We provide an implementation and numerical experiments.
\end{abstract}

% Display TOC with section titles only
% APPENDICES ARE EXCLUDED WITH AN EXTRA LINE OF CODE JUST BEFORE \appendix -- check that out if you want to list the appendices too.
\setcounter{tocdepth}{1}
\tableofcontents

\clearpage
\section{Introduction}

We propose\footnote{A Matlab \href{https://github.com/NicolasBoumal/manopt/tree/master/manopt/solvers/trustregions}{implementation} is available as part of the Manopt toolbox~\citep{manopt}, extended to optimization on Riemannian manifolds.} a practical algorithm for unconstrained minimization problems,
\begin{align}
    \min_{x \in \Rd} f(x),
    \label{eq:minf}
\end{align}
where the cost function $f \colon \Rd \to \reals$ is twice continuously differentiable and possibly non-convex.

For \emph{theory}, we further assume that the gradient $\nabla f$ and Hessian $\nabla^2 f$ are Lipschitz continuous, that the critical points of $f$ are uniformly non-degenerate, and that the gradient at $x$ is large when there are no critical points near $x$. (The algorithm can be run without these assumptions.)
% Note: We also assume $f$ is lower bounded, but this is rather tame and unobjectionable, so let's skip it in the intro so we can get on with it.

Our aim is that, under these favorable conditions, the algorithm should offer the following \emph{theoretical guarantees in high dimensions}:
\begin{enumerate}
    \item The sequence of iterates $x_0, x_1, x_2, \ldots$ converges to a local minimizer;
    \item For every small enough $\epsilon > 0$, there is an iterate $x_K$ within distance $\epsilon$ from a local minimizer and such that the total number of oracle calls (to $f$, $\nabla f$ and Hessian-vector products) expended to reach $x_K$ grows:
    \begin{enumerate}
      \item not worse than $\log(1/\epsilon)$ with respect to the accuracy $\epsilon$,\footnote{Even if $f$ were a strongly convex quadratic it would not be possible to do better in terms of $\epsilon$: see~\citep[\S2.1.4]{Nesterov20018}, \citep{Nemirovski1983}.} and
      \item not worse than $\log(d)$ with respect to the dimension $d$. % Note: About the lower-bound: it matters that the bad function is quadratic, because this means you can get Hessian-vector products too, simply because Hess f(x)[v] = grad f(v) - grad f(0).
    \end{enumerate}
\end{enumerate}
Before proposing something new, let us review some contenders.

For starters, gradient descent may get stuck in a saddle point, thus failing even the first requirement (albeit not generically so~\citep{Lee2016}). % Note: We hesitated to cite \citep{Du2017} here to point out that even with generic initialization GD can get really slow, but one would have to check if the bad functions built in that ref satisfy the assumptions we make here. We don't need to go there though, so let's not take a detour so early in the paper.

Under our assumptions,
second-order algorithms such as the trust-region method (TR) with an exact subproblem solver do meet requirements 1 and 2a: see \citep[\S3]{Goyens2024}. %Yes, they have a remark about actual convergence; infinite loop: no need to pre specify epsilon.
However, to solve the subproblem exactly, as many as $d$ Hessian-vector products are necessary in the worst case.
Thus, requirement 2b is not met.

In fact, such linear growth with $d$ is unavoidable for \emph{deterministic} algorithms.
This is folklore: see Section~\ref{ss:justification_randomization} for a proof.

Accordingly, we turn to \emph{randomized} algorithms.
The aim now becomes to satisfy all of the above with probability at least $1-\delta$ in each run of the algorithm, allowing the number of oracle calls to grow with $\log\!\left(\frac{d}{\delta\epsilon}\right)$.

As a prominent example, \citet{Jin2017,Jin2021} showed that adding sustained and well-calibrated random perturbations to the iterates of gradient descent is sufficient to escape saddle points with polylogarithmic scaling in dimension.
Their \emph{perturbed gradient descent} algorithm does not require Hessian-vector products.
It was later accelerated~\citep{Jin2018} and extended to Riemannian optimization~\citep{sun2019prgd,criscitiello2019escapingsaddles,criscitiello2020accelerated}.

These perturbed gradient algorithms, however, do not converge.
Rather, they produce a sequence of iterates which never settle, as they are pushed around by random perturbations.
If sufficiently many iterates are produced, then a constant fraction of them are near second-order critical (which, under our assumptions, implies they are near a local minimizer).
Moreover, one must decide on a target accuracy $\epsilon$ ahead of time, then carefully tune the step-sizes, the magnitude of the perturbations and the total number of iterations as a function of $\epsilon$ and the various other problem parameters (such as Lipschitz constants).
This is arguably impractical.

Meanwhile, practical deterministic algorithms tend to perform well in most cases.
This is true in particular of the \emph{trust-region method} with the \emph{truncated conjugate gradients} subproblem solver (TR-tCG, see below).
The latter does not solve the TR subproblem exactly, relying instead on a heuristic modification of the classical conjugate gradients (CG) algorithm to solve the subproblem approximately.
This typically entails a modest number of Hessian-vector product evaluations.
However, in its deterministic form, TR-tCG may get stuck at strict saddle points.
Even if not, its worst-case behavior scales at least polynomially in $d$ (because it is deterministic).

Therefore, we propose to \emph{lightly randomize TR-tCG}: enough to secure the desired theoretical guarantees, while being careful not overly to interfere with its habitual behavior.
This involves only three minor modifications to the standard algorithm (described next), and results in a method that meets the announced requirements.

\subsection{Reminders about the standard TR-tCG algorithm}

Given $x_0 \in \Rd$ and $\Delta_0 > 0$,
a trust-region method~\citep{Conn2000} generates a sequence of points $x_0, x_1, x_2, \ldots$ and radii $\Delta_0, \Delta_1, \Delta_2, \ldots$ as follows (see Algorithm~\ref{algo:pTR} with $\sigma = 0$ for now).
At iteration $k$, we compute an (approximate) minimizer $u_k$ for the \emph{trust-region subproblem} (TRS):
\begin{align}
    \min_{v \in \Rd} \; m_{x_k}(v) && \textrm{ subject to } && \|v\| \leq \Delta_k, && \textrm{ with } && m_{x_k}(v) = \inner{\nabla f(x_k)}{v} + \frac{1}{2} \inner{\nabla^2 f(x_k)v}{v}.
    \label{eq:subproblem_intro}
\end{align}
By design, $f(x_k + v) \approx f(x_k) + m_{x_k}(v)$ for $v$ small enough.
Thus, $x_k + u_k$ (approximately) minimizes a quadratic approximation of $f$ within the ball of radius $\Delta_k$ around $x_k$.
As we consider moving to that new point, the ratio
\begin{align}
    \frac{f(x_k) - f(x_k + u_k)}{m_{x_k}(0) - m_{x_k}(u_k)}
    \label{eq:rhointroclassic}
\end{align}
compares the actual decrease in function value to the decrease predicted by the model.
Depending on this ratio, the next iterate is either $x_{k+1} = x_k + u_k$ (accept) or $x_{k+1} = x_k$ (reject), and the next radius $\Delta_{k+1}$ is either set to $\Delta_k$, or reduced, or increased.

The eventual efficiency and practicality depend on the subproblem solver.
The time-tested approach we build on is the \emph{truncated conjugate gradients} method (tCG), also known as the Steihaug--Toint method---see~\citep{toint1981note,Steihaug1983tCG} and~\citep[\S7.5]{Conn2000}, and Algorithm~\ref{algo:tCG} with $\xi = 0$ and no boundary gradient step for now.

In a nutshell:
tCG consists in running CG on the quadratic model $m_{x_k}$ (at the cost of one Hessian-vector product per iteration), while
(a)~hoping to stop early with a sufficiently good approximate solution, and
(b) monitoring for signs that $\nabla^2 f(x_k)$ is not positive definite or that the iterates are leaving the ball of radius $\Delta_k$ (in which case we abort with a heuristic).
In all cases, some progress is made.
In the happy scenario where the Hessian is positive definite and the (then unique) minimizer of $m_{x_k}$ has norm less then $\Delta_k$, tCG (initialized at zero) is identical to CG: its iterates converge exponentially fast to the minimizer of the TRS.

\subsection{Proposed modifications of TR-tCG}

Our modifications to TR-tCG, reflected in Algorithms~\ref{algo:pTR},~\ref{algo:tCG} and~\ref{algo:bg} and Figure~\ref{fig:tCG-bg}, are as follows:
\begin{enumerate}
    \item Initialize tCG not with zero but with a small random vector $\xi_k$ of norm $\sigma$ (typically).

    This builds on a heuristic proposed by \citet[\S2.3.3]{baker2008riemannian} and connects with theory for randomized Krylov spaces developed by \citet{LanczosBound1992}, \citet{Carmon2018a} and \citet[Thm.~1]{Royer2020}.

    The noise scale $\sigma$ is the only new hyperparameter.
    For our theory to apply, it should be small enough compared to (often unknown) problem parameters, but it is safe to take $\sigma$ ``very small'' because undershooting only incurs a logarithmic penalty.

    \item Shift the ratio $\rho_k$ usually computed in TR methods.

    This accounts for the fact that tCG with a nonzero initialization may return an update vector $u_k$ which performs worse than the zero vector.
    Specifically, we add $\theta_k = m_{x_k}(\xi_k) - m_{x_k}(0)$ to both the numerator and denominator of~\eqref{eq:rhointroclassic}, resulting in~\eqref{eq:rhok}.
    This shift allows TR occasionally to accept a step leading to an increase in function value ($f(x_{k+1}) > f(x_k)$).
    We found that this helps to escape saddle points.

    \item Augment tCG with a ``boundary gradient step'': $\tCGbg$.

    Classically, tCG runs CG until it either obtains an acceptable solution (small residue) or it runs into the boundary of the trust region (radius $\Delta_k$).
    Here, we run tCG with radius $\Delta_k / 2$.
    If it runs into \emph{that} boundary, then we execute one more gradient step from that point, limited to the larger region of radius $\Delta_k$ (see Figure~\ref{fig:tCG-bg}).

    This unlocks a key decrease inequality which would otherwise be unavailable when the model is non-convex and the gradient is small (that is, near saddles): see Section~\ref{sec:tcgbgbenefits}.
\end{enumerate}
These modifications are simple to implement, and do not overly perturb the normal functioning of TR-tCG when unnecessary.
They do, however, enable new theoretical guarantees.

\begin{figure}[t]
    \centering
    \begin{subfigure}{0.45\textwidth}
        \includegraphics[width=0.9\linewidth]{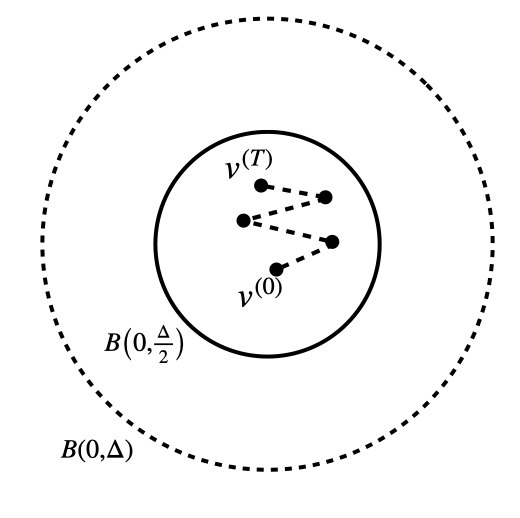}
        \caption{$\stopcrit = \stopres$}
    \end{subfigure}
    \begin{subfigure}{0.45\textwidth}
        \includegraphics[width=0.9\linewidth]{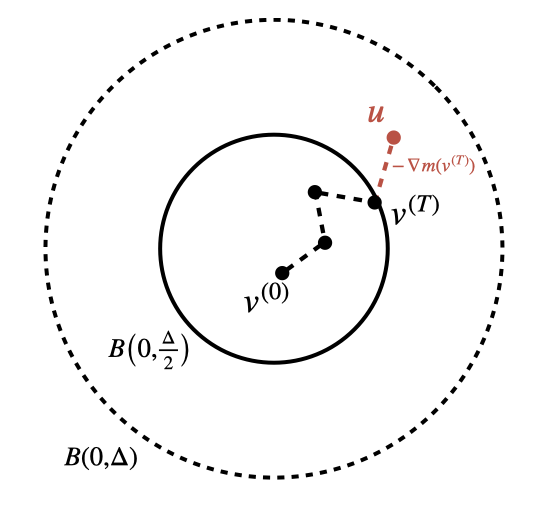}
        \caption{$\stopcrit = \stopoob$}
    \end{subfigure}
    \caption{Demonstrating the two scenarios in $\tCGbg$ (Algorithm~\ref{algo:tCG}). If the $\stopres$ criterion triggers while the iterates are still inside the ball of radius $\Delta/2$, we stop there (and the method is equivalent to standard CG). If the iterates reach the sphere of radius $\Delta/2$, we perform a last gradient step from $v^{(T)}$. This last step unlocks a descent guarantee in the non-convex case.}
    \label{fig:tCG-bg}
\end{figure}

\subsection{How these modifications help to reach the stated goal}

One of the key difficulties in the algorithm design is that it must both be able to escape saddle points (which requires ``jitter'') and to converge to local minimizers (which requires steadiness)---yet it is not straightforward to determine whether we are near one or the other.

An important feature of our algorithm is that it does not try to make such a determination, and therefore it also does not need to know problem constants such as the ones that appear in our assumptions below, nor does it require one to preset a target tolerance $\epsilon$.
This is in contrast to several theoretical algorithms which switch between modes (e.g., ``large gradient steps'' or ``saddle point detection''), often based on unknowable thresholds.

Our analysis relies on the fact that \emph{CG itself behaves differently near saddles and minimizers}.

\begin{itemize}
  \item \textbf{Near local minima}, the quadratic model is strongly convex, and the CG iterates are ``captured'' inside the trust region, provided the radius is large enough.
  It follows that the TR iterates converge quadratically to the nearby minimizer.
  This property \emph{still holds in the presence of randomization} (unlike methods such as perturbed gradient descent).
  This is because the noise is added to the \emph{initialization} of the subproblem solver rather than to its output: CG has an opportunity to \emph{absorb the noise}.

  \item \textbf{Near saddle points}, the CG iterates are ``pushed outwards'' by the negative curvature: this result is stated in Proposition~\ref{prop:vt_qi_growth}, which is of independent interest (see Figure~\ref{fig:tcg_nonconvex} for an illustration).
  Therefore, they cannot converge to the saddle, unless the initialization is in a small unfavorable region.
  The random perturbation avoids this region with high probability.

  To obtain complexity guarantees, it is not sufficient to prove that the iterates do not converge to the saddle: we also need to show that they achieve \emph{sufficient decrease in function value}.
  Unfortunately, we cannot secure this with standard tCG. Indeed, even though CG eventually escapes, it might do so through a ``flat'' direction of the landscape which achieves little to no decrease.

  We solve this issue by performing an additional gradient step from the boundary of the trust region (see Section~\ref{s:proof_sketch} for a proof sketch).
  The idea is that, when the Hessian is non-degenerate, the gradient norm increases as we move away from the saddle.
\end{itemize}

% The analysis rests on several new ideas.
% In particular, it relies on controlling the behavior of the algorithm over the span of several (inner and outer) iterations.

Our theoretical analysis crucially requires that the Hessian eigenvalues at critical points are bounded away from zero by a global constant (the $\mu$-Morse property).
In this setting, we show that the algorithm converges with high probability to a local minimizer.
The complexity to reach an $\epsilon$-minimizer depends on several problem constants, including a lower bound on the gradient norm away from critical points, but only logarithmically on dimension $d$ and precision~$\epsilon$.

Moreover, our algorithm is \emph{practical}: we only introduce one extra parameter (the noise scale $\sigma$), and it does not require additional costly subroutines, such as eigenvalue decompositions.
Its implementation and performance is mostly comparable to standard TR-tCG

\subsection{Limitations}

Let us briefly highlight some limitations.

While it is straightforward to run the proposed algorithm on any $C^2$ cost function, our theoretical analysis requires isolated critical points (among other things).
Relaxing this assumption may be possible, but it seems non-trivial.

Moreover, the algorithm is not entirely adaptive since we require the noise scale $\sigma$ to be smaller than some threshold $\bar \sigma_{\rm g}$ which depends on unknown problem constants.
Fortunately, our bounds depend only logarithmically on $\sigma$.
In practice, we set $\sigma$ to a small positive value (e.g., $10^{-6}$).
If this is smaller than necessary, the penalty is mild.
If it is larger than necessary, the algorithm may still work in practice, and theoretical guarantees similar to standard TR-tCG still apply.

Finally, although our complexity estimate has good dependence in $d$ and $\epsilon$, it has non-ideal dependence in the Hessian condition number (near saddle points), and in the maximal trust-region radius $\bar \Delta$.
% \nb{We do get the right scaling wrt condition number near local minimizers though (Lemma 7.1) -- let's make sure this is clear, too.}
Some competing algorithms (see related work below) have better complexity, at the expense of being less practical.
% \radu{the paragraph "approximate second-order methods" in the related work section. Our closest direct competitors are the Newton-CG method from \cite{Royer2020}, and trust-region "capped-CG" from \cite{Curtis2021}}
This may be due to improvable bounds in our analysis.
% We believe that our theoretical analysis is overly conservative on some aspects, and that the bounds could be improved in future work.

\subsection{Outline of the paper}

Section~\ref{s:algo_details} gives additional details on the algorithm, Section~\ref{ss:assumptions} states the assumptions, Section~\ref{sss:informal_thm} announces the main theorem, and Section~\ref{s:proof_sketch} sketches the proof.
The rest of the paper provides a detailed proof (Sections~\ref{s:preliminaries}--\ref{s:complexity}) as well as numerical experiments (Section~\ref{s:numerical_experiments}), followed by a discussion.

\subsection{Related work}

\paragraph{Approximate second-order methods.}

Several works study inexact second-order methods for non-convex problems, with convergence guarantees to a local minimizer. Closest to our work are the trust-region CG  methods of \citep{Curtis2021,Goyens2024}. Other approaches include inexact cubic-regularized Newton \citep{Agarwal2017}, and
Newton-CG variants analyzed in \citep{Royer2018,Royer2020,Yao2022}. More recently, \cite{Liu2023} suggest replacing CG with MINRES as a subproblem solver in Newton schemes. We highlight that, unlike our approach, all these methods require the use of an additional subroutine, such as eigenvector search, to detect negative curvature around saddles.

\paragraph{Other saddle-avoiding algorithms.}

A line of work uses tools from dynamical systems, such as the \emph{center-stable manifold theorem}, to show that gradient descent avoids saddle points from almost all initializations \citep{Lee2016,Lee2019}.
These results do not address the rate of convergence, which can be slow \citep{Du2017}.
See also~\citep{dixit2023exittime} for escape rates of GD near saddles.

Several other works design algorithms to find a local minimizer (or more generally, a second-order stationary point), approximately, with complexity guarantees.
\begin{itemize}
  \item \textbf{Perturbed gradient descent:} the idea of adding well-calibrated noise to gradient descent in order to avoid saddles appeared in \citep{Ge2015}, and the complexity was later improved by \citet{Jin2017}.
  \item \textbf{Accelerated gradient methods:} \citet{Carmon2017convex_guilty} propose an efficient algorithm based on Nesterov acceleration for non-convex problems in order to compute an approximate critical point. \citet{Carmon2018} show that approximate second-order stationarity guarantees can be obtained for accelerated gradient descent, at the price of using an \emph{eigenvector search} subroutine to detect negative curvature.
  \cite{Jin2018} design a perturbed accelerated gradient method that avoids eigenvector search, instead relying on the addition of noise and a simpler ``negative curvature exploitation'' subroutine.
  \item \textbf{Exact second-order methods:} algorithms that perform exact minimization of the second-order model provably avoid saddle points. This is the case for cubic-regularized Newton \citep{Nesterov2006,Cartis2009}, as well as trust-region methods \citep{Cartis2012,Curtis2016}. The main caveat is that solving the subproblem exactly is costly in high dimension.
\end{itemize}

Other papers on saddle avoidance algorithms include \citep{AllenZhu2018,cao2025efficientlyescapingsaddlepoints,li2024randomizedalgorithmnonconvexminimization}.
Our work is also related to the theory of randomized Krylov subspaces \citep{LanczosBound1992,Carmon2018a}, which we use to analyze CG with perturbed initialization.

\paragraph{Global guarantees in non-convex optimization.}

Our work contributes to research on algorithms and theoretical guarantees for non-convex optimization.
Several such problems exhibit what is known as \emph{benign non-convexity}, where the objective function may have saddle points but it has no spurious local minima.
This observation motivates efforts to build optimization algorithms that avoid saddles and converge to local minima.
Notable examples include phase retrieval \citep{Sun2017}, matrix completion \citep{Ge2016}, Burer-Monteiro factorizations \citep{Boumal2019} and linear neural networks \citep{Kawaguchi2016}.
See also \citep{Chi2019} for more.

\subsection{Notation}

At iteration $k$ of Algorithm~\ref{algo:pTR}, we let $g_k = \nabla f(x_k)$ and $H_k = \nabla^2 f(x_k)$.
The iterates of Algorithm~\ref{algo:pTR} are the \emph{outer} iterates, indexed with the subscript $k$ (e.g., $x_k$), while the iterates of Algorithm~\ref{algo:tCG} ($\tCGbg$, the \emph{inner} iterates) are indexed with the superscript ${(t)}$ (e.g., $v^{(t)}$).
We write $v_k^{(t)}$ to denote the $t$th inner iterate of $\tCGbg$ at outer iteration $k$ of Algorithm~\ref{algo:pTR} (and likewise for $r_k^{(t)}$ and $p_k^{(t)}$).

The positive part of $s \in \reals$ is $[s]_+ = \max(s, 0)$.
Also, $\lfloor s \rfloor$ is $s$ rounded down to the nearest integer.
The set of symmetric matrices of size $d$ is $\symm_d$.
For $A \in \symm_d$, we let $\lambdamin(A)$
%and $\lambdamax(A)$
denote the smallest
%and largest eigenvalues of $A$.
eigenvalue of $A$.
For $v \in \Rd$, the Euclidean norm is $\|v\| = \sqrt{v^\top v}$.
For $M \in \Rdd$, the operator (or spectral) norm is $\opnorm{M}$ (the largest singular value of $M$).
%
%The unit sphere in $\Rd$ for the Euclidean norm is $S^{d-1} = \{ x \in \Rd : \|x\| = 1 \}$.
%, while $B(x, R)$ is the closed Euclidean ball of radius $R$ centered at $x$.

\section{The algorithm} \label{s:algo_details}

We described the proposed randomized trust-region method at a high level in the introduction.
The specifics are given in Algorithm~\ref{algo:pTR} (the trust-region outer scheme), Algorithm~\ref{algo:tCG} (the subproblem solver, $\tCGbg$) and Algorithm~\ref{algo:bg} (the boundary gradient step).

Before diving into the theoretical analysis, we give additional details on the subproblem solver $\tCGbg(H, g, \Delta, \xi)$.
This method generates a sequence of iterates $v^{(0)}, \ldots, v^{(T)}$ and outputs a vector $u$ as well as a termination reason $\stopcrit \in \{ \stopoob, \stopres \}$, in an attempt to approximately minimize the model $m(v) = \inner{g}{v}+ \frac{1}{2}\inner{v}{Hv}$ in the ball of radius $\Delta$.
Internally, the method also generates the \emph{conjugate directions} $p^{(t)}$ and the \emph{residual} vectors $r^{(t)} = -\nabla m(v^{(t)})$.

This algorithm starts from the initialization $v^{(0)} = \xi$, and it generates the iterates $v^{(t)}$ using (classical) CG for iterations $t = 0, 1, \ldots$ until one of the following happens:
\begin{itemize}
    \item If the residual $r^{(t)}$ is small with respect to $g$, then it terminates with output $u = v^{(t)}$ and $\stopcrit = \stopres$.
    \item If the normal CG step would \emph{leave the ball} of radius $\Delta / 2$ or if \emph{nonpositive curvature} is detected in the Hessian, then
    \begin{itemize}
        \item $v^{(t)}$ is computed instead by minimizing $s \mapsto m(v^{(t-1)} + s p^{(t-1)})$ constrained to the ball of radius $\Delta/2$ (this is a simple quadratic on an interval), and
        \item $\tCGbg$ \textbf{performs a last gradient step on
        $m$ from $v^{(t)}$} constrained to the ball of radius $\Delta$, to get $u = v^{(t)} - s' \nabla m(v^{(t)})$ with the optimal $s'$. The details are specified in Algorithm~\ref{algo:bg}. It then terminates with output $u$ and $\stopcrit = \stopoob$.
    \end{itemize}
\end{itemize}

Notice that the subproblem solver (Algorithms~\ref{algo:tCG} and~\ref{algo:bg}) is deterministic.
The only source of randomness is the vectors $\xi_k$ generated in Algorithm~\ref{algo:pTR} (and passed to the subproblem solver).

\begin{algorithm}[t]
    \caption{\PRP{Randomized} trust-region method (\PRP{modifications} are highlighted in color)}
    \label{algo:pTR}
    \begin{algorithmic}[1]
        \Statex \textbf{Parameters:} maximum radius $\bar\Delta > 0$ and acceptance thresholds $0 < \rho' < \rho''< 1$
        \Statex \textbf{Input:} initial iterate $x_0 \in \reals^d$, initial radius $\Delta_0$ and \PRP{noise scale $\sigma$} with $0 < \PRP{4\sigma} \le \Delta_0 \le \bar \Delta$ \vspace{.5em}
        \For{$k = 0, 1, 2, \ldots$}
            \State Let $H_k = \nabla^2 f(x_k)$ and $g_k = \nabla f(x_k)$
            \State \PRP{Sample $\bar\xi_k$ uniformly at random on the unit sphere in $\Rd$}
            \State \PRP{Set $\xi_k = \pm \min(\sigma, \Delta_k/4) \bar\xi_k$ with the sign chosen such that $\langle H_k \xi_k, g_k \rangle \ge 0$} \label{line:xik}
            \State Call $\PRP{\tCGbg}(H_k, g_k, \Delta_k, \PRP{\xi_k})$ to get step $u_k$ and termination reason $\stopk{k}$
            \State Compute the \PRP{shifted} ratio of improvement in $f$ vs the model $m_{x_k}(v) = \inner{g_k}{v} + \frac{1}{2} \inner{v}{H_kv}$:
              \begin{align}
                \rho_k & = \frac{f(x_k) - f(x_k + u_k) + \PRP{\theta_k}}{m_{x_k}(0) - m_{x_k}(u_k) + \PRP{\theta_k}}, &&
                \textrm{ with } && \PRP{\theta_k = m_{x_k}(\xi_k) - m_{x_k}(0)}.
                \label{eq:rhok}
              \end{align}
            \State Accept or reject the tentative next iterate:
            \begin{align}
              x_{k+1} =
                \begin{cases}
                    x_k + u_k & \textrm{if } \rho_k \ge \rho' \textrm{ (accept)}, \\
                    x_k & \textrm{otherwise (reject).}
                \end{cases}
              \label{eq:acceptreject}
            \end{align}
            \State Update the trust-region radius: % Note: first condition is more commonly $\rho_k < 1/4$ with $0 < \rho' < 1/4 < \rho'' < 1$, but this has few ramifications.
            \begin{align*}
              \Delta_{k+1} =
                \begin{cases}
                    \frac{1}{4} \Delta_k & \textrm{if } \rho_k < \rho', \\
                    \min(2 \Delta_k, \bar\Delta) & \textrm{if } \rho_k > \rho'' \textrm{ and } \stopk{k} = \stopoob, \\
                    \Delta_k & \textrm{otherwise.}
                \end{cases}
            \end{align*}
        \EndFor
    \end{algorithmic}
\end{algorithm}

\begin{algorithm}[ph]
    \caption{$\tCGbg(H, g, \Delta, \xi)$: truncated conjugate gradients with boundary gradient step}
    \label{algo:tCG}
    \begin{algorithmic}[1]
        \Statex \textbf{Parameters:} residual thresholds $\omega_1 \in (0, 1)$ and $\omega_2 > 0$
        \Statex \textbf{Input:} $H \in \symm_d$, $g \in \reals^d$, radius $\Delta > 0$, initial point $\xi \in \reals^d$ with $\|\xi\| \le \Delta / 4$
        \Statex \textbf{Output:} approximate minimizer $u$ of $m(v) = \inner{g}{v} + \frac{1}{2}\inner{v}{Hv}$ subject to $\|v\| \le \Delta$,
        \Statex \phantom{\textbf{Output:}} termination reason $\stopcrit \in \{ \stopoob, \stopres \}$, and total number of iterations $T$ \vspace{.5em}
        \State Initialize $v^{(0)} = \xi$
        \State Set $r^{(0)} = -(H v^{(0)} + g)$ and $p^{(0)} = r^{(0)}$
        \State \textbf{if} $r^{(0)} = 0$ \textbf{then} \textbf{output} $u = v^{(0)}$, $\stopcrit = \stopres$, and $T = 0$ % Note: We can get rid of this by moving the residual stopping criterion to the beginning of the for loop, but it's also fine as is. Keeping the residual stopping criterion at the end makes it possible to exit before computing the last beta and p, which is how one would program it too, so let's keep it that way.
        \For{$t = 1, 2, \ldots$}
            \State $\alpha^{(t)} = \frac{\|r^{(t-1)}\|^2}{\inner{p^{(t-1)}}{H p^{(t-1)}}}$
            \State $v^{(t-1)}_+ = v^{(t-1)} + \alpha^{(t)} p^{(t-1)}$
            \If{$\inner{p^{(t-1)}}{H p^{(t-1)}} \le 0$ \textbf{or} $\|v^{(t-1)}_+\| \ge \Delta / \PRP{2}$} \label{line:iftoboundaryintcgbg}
                \State Set $v^{(t)} = v^{(t-1)} + s p^{(t-1)}$ with $s \ge 0$ such that $\|v^{(t)}\| = \Delta / \PRP{2}$
                \State \PRP{$u = \boundarygradientstep(H, g, \Delta, v^{(t)})$}
                \State \textbf{output} $u$, $\stopcrit = \stopoob$, and $T = t$
            \Else
                \State $v^{(t)} = v^{(t-1)}_+$
            \EndIf
            \State $r^{(t)} = r^{(t-1)} - \alpha^{(t)} H p^{(t-1)}$
            \If{$\|r^{(t)}\| \le \min(\omega_1 \|g\|, \omega_2 \|g\|^2)$}
                \State \textbf{output} $u = v^{(t)}$, $\stopcrit = \stopres$, and $T = t$
            \EndIf
            \State $\beta^{(t)} = \frac{\|r^{(t)}\|^2}{\|r^{(t-1)}\|^2}$
            \State $p^{(t)} = r^{(t)} + \beta^{(t)} p^{(t-1)}$
        \EndFor
    \end{algorithmic}
\end{algorithm}

\begin{algorithm}[ph]
    \caption{$\boundarygradientstep(H,g,\Delta,v)$}
    \label{algo:bg}
    \begin{algorithmic}[1]
        \Statex \textbf{Input:} $H \in \symm_d$, $g \in \reals^d$, radius $\Delta > 0$, initial point $v \in \reals^d$ with $\|v\| \le \Delta$.
        % Note: this performs a Cauchy-type gradient step starting from $v$, constrained to the ball $B(0, \Delta)$
        \Statex Minimizes $m(u) = \inner{g}{u} + \frac{1}{2} \inner{u}{Hu}$ along $\{ v - s \nabla m(v) : s \in \reals \}$ constrained to $\|u\| \leq \Delta$. \vspace{.5em}
        \State $r = -(Hv + g)$
        \State $\bar{\alpha} = \frac{\|r\|^2}{\langle r, H r \rangle}$
        \State $v_+ = v + \bar{\alpha} r$
        \If{$\langle r, H r \rangle \le 0$ \textbf{or} $\|v_+\| \ge \Delta$}
            \State \textbf{output} $u = v + s' r$ with $s' \ge 0$ such that $\|u\| = \Delta$
        \Else
            \State \textbf{output} $u = v_+$
        \EndIf
    \end{algorithmic}
\end{algorithm}

\section{Saddle escape and fast convergence in high dimension}
\label{s:assumptions_main_thm}

In this section, we first state and comment on our assumptions, then provide a simplified statement of the main theorem, and finally sketch the proof of that theorem to serve as a roadmap.

\subsection{Assumptions} \label{ss:assumptions}

We consider the minimization problem~\eqref{eq:minf} under the following standing assumptions on $f$.
\begin{assumption}\label{ass:f}
  We work in dimension\footnote{We assume $d \geq 3$ for convenience because the probability concentration bounds derived in Appendix~\ref{app:technical} would require separate treatment for $d \in \{1, 2\}$. In any case, for those low dimensions, it is more practical to solve the trust-region subproblem exactly.} $d \geq 3$.
  The cost function $f \colon \Rd \to \reals$ is twice continuously differentiable and lower-bounded by $\flow \in \reals$, that is, $f(x) \geq \flow$ for all $x \in \Rd$.
  Moreover,
    \begin{enumerate}[label=(\roman*)]
        \item \label{subass:lh} The Hessian of $f$ is Lipschitz continuous with constant $L_H$, that is,
        \begin{align*}
          \opnorm{\nabla^2 f(x) - \nabla^2 f(y)} \leq L_H \|x-y\| && \forall x, y \in \Rd,
        \end{align*}
        and
        \item \label{subass:lg} The gradient of $f$ is Lipschitz continuous with constant $L_G$, that is,
        \begin{align*}
          \| \nabla f(x) - \nabla f(y) \| \leq L_G \|x-y\| && \forall x, y \in \Rd.
        \end{align*}
        Equivalently, the Hessian is bounded: $\opnorm{\nabla^2 f(x)} \leq L_G$ for all $x \in \Rd$.
    \end{enumerate}
\end{assumption}
These assumptions are common to streamline theoretical analysis~\citep{Conn2000,Cartis2012,Jin2017}.
In practice, the Lipschitz conditions may not hold globally on $\Rd$.
However, since the algorithm is (approximately) monotone, the iterates belong to a sublevel set of the form ${\{x \in \Rd \,:\, f(x) \leq f(x_0) + \varepsilon \}}$; see Lemma~\ref{lemma:obj_decrease}.
Thus, it would be sufficient to assume that the gradient and Hessian are Lipschitz on that set.
This is the case if, for instance, $f$ is $C^3$ and coercive (as in Example~\ref{ex:mb} below).
At any rate, the constants $\flow, L_G, L_H$ are not required to run the algorithm.

We impose two additional conditions: the $\mu$-Morse property (to ensure critical points are uniformly non-degenerate) and the strong gradient property (to ensure the gradient norm is lower bounded away from critical points).

\begin{definition}\label{def:mu_morse}
    Let $\mu > 0$.
    We say that $f$ satisfies the \textbf{$\mu$-Morse property} if, for every critical point $\xbar$ of $f$, the singular values of $\nabla^2 f(\xbar)$ are at least $\mu$, that is,
    \begin{align*}
      \| \nabla^2 f(\xbar)u \| \geq \mu \|u\| && \forall u \in \Rd.
    \end{align*}
    Equivalently, no eigenvalue of $\nabla^2 f(\xbar)$ lies in the interval $(-\mu, \mu)$.
\end{definition}
If $f$ has the $\mu$-Morse property, then its critical points are isolated. Let $\xbar$ be a critical point: if $\nabla^2 f(\xbar)$ is positive definite, then $\xbar$ is a (strict) local minimizer, otherwise, we say that $\xbar$ is a (strict) saddle point.\footnote{For convenience, our definition of saddle points includes local maxima; this will not be an issue.}
Under Assumption~\ref{ass:f}, it further follows that the Hessian's eigenvalues at every critical point lie in the set $[-L_G,-\mu] \cup [\mu,L_G]$, where necessarily $\mu \leq L_G$.
\begin{definition}\label{def:strong_saddle}
    Let $R,\gamma > 0$.
    We say that $f$ satisfies the \textbf{$(R,\gamma)$-strong gradient property} if, for every $x \in \reals^d$, at least one of the following holds:
    \begin{enumerate}[label=(\roman*)]
        \item $\|\nabla f(x)\| \geq \gamma$;
        \item there exists a critical point $\xbar$ of $f$ such that $\|x - \xbar\| \leq R$.
    \end{enumerate}
\end{definition}
One way to secure the strong gradient property is via Proposition~\ref{prop:strong_gradient} in the appendix, which simply requires the gradient to be bounded away from zero for all $x$ outside a compact set.

The next example from~\citep[\S4]{Jin2017} illustrates these properties.

\begin{example}[Rank-one matrix factorization]\label{ex:mb}
    Let $M \in \symm_d$ be a positive definite matrix with eigenvalues $\lambda_1 > \dots > \lambda_d > 0$.
    Let
    \[
    f(x) = \frac{1}{4}\|xx^\top - M\|_{\mathrm{F}}^2,
    \]
    where $\|\cdot\|_{\mathrm{F}}$ denotes the Frobenius norm.
    Then $f$ satisfies the $\mu$-Morse property with constant
    \[
    \mu = \min\!\big( \lambda_1 - \lambda_2, \ldots, \lambda_{d-1} - \lambda_d, \lambda_d \big) > 0.
    \]
    (See~\citep[\S4]{Jin2017} or Lemma~\ref{lemma:mu_morse_example} in the appendix for a more general proof.)
    This function also satisfies the $(R, \gamma)$-strong gradient property for \emph{any} $R>0$ with \emph{some} $\gamma>0$, because $\|\nabla f(x)\|$ is lower bounded by a positive constant for $x$ large enough (see Proposition~\ref{prop:strong_gradient} and Lemma~\ref{lemma:mu_morse_example}).

    The Lipschitz conditions in Assumption~\ref{ass:f} do not hold globally; however, as commented above, this can be resolved by noting that the sublevel sets are compact.
\end{example}

Various versions of the assumptions above appear in the related literature, including in~\citep[\S3.1]{Ge2015}, \citep[\S3.1]{Jin2017} and~\citep[\S2.3]{Goyens2024}.
See also more literature pointers in~\citep[Rem.~2.2]{Goyens2024}.
The main way in which our assumptions depart from those is that we require all critical points to be isolated.
We believe this could be relaxed, but the extension seems non trivial.

Their commonality is that they allow to split the search space $\Rd$ in \emph{three regions}.
To make this explicit, let us state our formal assumption.
It requires the strong gradient radius $R$ to be small enough: Proposition~\ref{prop:strong_gradient} can be helpful to that effect, as in the example above.

\begin{assumption} \label{ass:mumorse_and_stronggrad}
  The cost function $f$ satisfies the $\mu$-Morse property and the $(R_s, \gamma_s)$-strong gradient property with
  $
    R_s \leq \frac{\mu^2 }{4L_H L_G}.
  $
\end{assumption}

\noindent
Under Assumptions~\ref{ass:f} and~\ref{ass:mumorse_and_stronggrad}, it follows that for any $R\leq \mu / (2L_H)$ we can partition $\Rd$ in three regions.
This will follow from Lemma~\ref{lemma:growth_mu} with $c = 1/2$:
\begin{enumerate}[label=(\roman*)]
  \item The \textbf{saddle point neighborhoods}: around each saddle point $\xbar$, the ball of radius $R$ contains no other critical point, and $\nabla^2 f(x)$ has an eigenvalue less than $-\mu/2$ for all $x$ in that ball.
  \item The \textbf{local minimizer neighborhoods}: around each local minimizer $x^*$, the function $f$ is $\mu/2$-strongly convex in the ball of radius $R$.
  \item The \textbf{large gradient region}: for all $x$ in the complement of those balls, $\|\nabla f(x)\|$ is lower bounded by $\min\!\left( \mu R/2, \gamma_s \right)$.
  % Note: The logic goes: if $R \geq R_s$, then we get $\gamma_s$ as a lower-bound directly; and if $R < R_s$, then use Lemma~\ref{lemma:growth_mu}(i) with $c = 1/2$ to cover the points that are at distance between $R$ and $R_s$ of a critical point.
\end{enumerate}
In particular, the balls around critical points are disjoint.
This is formalized in Section~\ref{s:global_comp}.

\subsection{Main theorem}\label{sss:informal_thm}

% Defined some macros so if we want to change the notation back later it is easier
\newcommand{\gradlb}{\bar G} % in previous version it was \gamma(\bar R), I use \bar G to be consistent with later notation
\newcommand{\tildeO}{\tilde{\mathcal{O}}}
\newcommand{\sigmag}{\bar \sigma_{\rm global}} % in previous version it was \bar \sigma_{\rm g} but I use \bar \sigma_{\rm global} to be consistent with later notation

Consider a function $f$ satisfying Assumptions~\ref{ass:f} and~\ref{ass:mumorse_and_stronggrad}.
We define a radius $\bar R$ (Eq.~\eqref{eq:def_constants_DeltacritRbar} in Section~\ref{s:global_comp}) and a maximal noise value $\sigmag$ (Eq.~\eqref{eq:def_constants} and \eqref{eq:cond_sigma_global}).
They both depend polynomially on problem constants $(\mu,L_H,L_G,\gamma_s)$ and algorithm parameters $(\rho',\omega_2,\Delta_0)$, with $\sigmag$ also depending \textbf{logarithmically} on the confidence $\delta$ and dimension~$d$ (see Remark \ref{remark:sigma_global}).

The constant $\bar R$ scales roughly as $\min(\gamma_s / L_G, \mu^2 / (L_HL_G))$ and represents the radius of the balls around critical points that we consider for the local analysis.
Following the discussion above, let
\[\gradlb :=\frac{\mu}{2} \bar R\]
 denote a lower bound for the gradient norm outside these neighborhoods.
Here is a simplified statement of our main result (in Appendix~\ref{app:informal_justif}, we justify that this is indeed a valid simplification of the formal theorems).
\begin{theorem}[Simplification of Theorems~\ref{thm:outer_complexity} and~\ref{thm:global_inner_its}]
    Assume that the noise scale satisfies $0 < \sigma \leq \sigmag$.
    Then, the radius is lower bounded as $\Delta_k \geq 8 \bar  R$ for every~$k$ (Lemma~\ref{lemma:lowbound_deltak}) and, for any target accuracy $\epsilon \leq \bar R$, with probability at least $1-\delta$, the method finds a point $x_K$ that is $\epsilon$-close to a local minimizer $x^*$ in at most
    \begin{equation}\label{eq:bound_outer_informal}
    K \leq \tildeO\!\left(  \frac{ L_G\left(f(x_0)-\flow\right)}{\gradlb^2} \log\log \!\left(\frac{d}{\delta \sigma}\right)  + \log_2 \! \left(1 + \log_2 \frac{\bar R}{\epsilon}\right)\right)
    \end{equation}
    outer iterations.
    These require at most $K+2$
    evaluations of $f$ and $\nabla f$, and at most
    \begin{equation}\label{eq:bound_inner_informal}
      K \cdot  \tildeO\!\left( \sqrt{\frac{L_G}{\mu}} \log\!\left( \frac{d}{\delta \sigma \epsilon} \right) + \frac{L_G^2 \bar \Delta^2}{ \min( \omega_1 \gradlb, \omega_2 \gradlb^2)^2 } \right)
    \end{equation}
    Hessian-vector products (inner iterations) in total.
    Furthermore, once the point $x_K$ is found, for all $k > K$, the iterates $x_k$ converge to $x^*$ deterministically with a quadratic convergence rate.
    Here, the notation $\tilde{\mathcal{O}}$ hides additional logarithmic dependence on problem constants $(\flow, L_G, L_H, \mu, \gamma_s)$, algorithm parameters $(\rho', \omega_1, \omega_2)$ and inputs $(x_0, \Delta_0)$.
\end{theorem}

The right-hand side of \eqref{eq:bound_outer_informal} reveals a two-phase convergence behavior.
The first term bounds the number of iterations performed in both the \emph{large gradient} region and \emph{saddle point} neighborhoods.
Up to logarithmic factors, it matches the classical $\mathcal{\tilde{O}}\!\left( L_G \left( f(x_0)-\flow\right)/\bar G^2 \right) $ bound that describes the number of outer iterations needed for standard TR to find a point with gradient norm less than $\bar G$ \citep{Conn2000}.
After the first phase, the iterates enter the $\bar R$-neighborhood of a local minimizer, where the method enjoys deterministic quadratic convergence: this corresponds to the second term in~\eqref{eq:bound_outer_informal}.
See Section~\ref{s:discussion} for additional discussion.

Note also that, when $f$ fails to satisfy the $\mu$-Morse and strong gradient property, Algorithm \ref{algo:pTR} is still capable of finding $\epsilon$-critical points.
We show in Proposition~\ref{prop:epsilon_complexity} that, when $f$ only satisfies Assumption~\ref{ass:f}, and provided that $\sigma$ is small enough, the method finds an $\epsilon$-critical point in $\mathcal{\tilde{O}}(L_G \left( f(x_0)-\flow\right)/ \epsilon^2)$ outer iterations (again, the standard rate).

\subsection{Proof sketch} \label{s:proof_sketch}
Let $g_k = \nabla f(x_k)$ and $H_k = \nabla^2 f(x_k)$ for every $k \geq 0$.
Recalling the partition of $\Rd$ described in Section~\ref{ss:assumptions}, we classify the iterations in three types: \textit{large gradient}, \textit{saddle point} and \textit{local minimizer} iterations.
Importantly, the algorithm does not explicitly distinguish between these types: it is a theoretical construct.

For \textbf{local minimizer} iterations, we establish a local convergence result (Proposition~\ref{prop:min_capture}): if $\|x_k - x^*\| \leq \bar R$ where $x^*$ is a local minimizer, then
\[
  \frac{\|x_{k+1}-x^*\|}{2\bar R} \leq \left( \frac{\|x_{k}-x^*\|}{2\bar R}\right)^2,
\]
and therefore $x_k, x_{k+1}, \ldots$ converges (at least) quadratically to $x^*$.
This is deterministic.
In particular, the random perturbations $\xi_k, \xi_{k+1}, \ldots$ are ``absorbed''.

To prove that $x_k$ eventually enters the neighborhood of a minimizer, we show that both {large gradient} and {saddle point} iterations decrease the objective by a sufficient amount $\Flg \approx \frac{\gradlb^2}{L_G}$.
This provides control since the total progress in objective value is at most $f(x_0) - \flow$.

For successful \textbf{large gradient} iterations (Proposition~\ref{prop:bound_large_gradient}), we use the fact that
\BEQ\label{eq:sketch_lg}
  f(x_k)-f(x_{k+1}) \geq \rho' \cdot\min\!\left( \frac{1}{2L_G}\|g_k\|^2, \frac{1}{8} \Delta_k \|g_k\| \right) - \mathcal{O}(\sigma).
\EEQ
This is similar to the analysis of standard TR-tCG, up to a perturbation due to randomization involving the noise scale $\sigma$.
Remember that, in this region, $\|g_k\|$ is lower bounded by $\gradlb$.
This leads to a decrease guarantee if the noise is small enough (which we assume) and if the radius $\Delta_k$ is lower bounded (which we prove in Lemma~\ref{lemma:lowbound_deltak}).
This too is deterministic.

For \textbf{saddle point} iterations, $\|g_k\|$ is small and we cannot use the bound \eqref{eq:sketch_lg} to establish sufficient decrease. This is where we need the following \emph{saddle escape} result.
Ultimately, this is what drives the modifications we brought to the algorithm. Notice how (a) the escape happens over several iterations (single iteration analysis is not enough), and (b) this is where randomization comes in.

\begin{theorem}[Saddle escape, simplification of Theorem~\ref{thm:saddle_escape}]
  Assume there exists an iteration $k_0$ of Algorithm~\ref{algo:pTR} such that $\|x_{k_0} - \xbar\| \leq \bar R$ where $\xbar$ is a saddle point of $f$.
  Then, for any $\delta \in (0,1)$, if the noise scale satisfies $0 < \sigma \leq \sigmag$, with probability at least $1-\delta$, there exists an index $k_r \geq k_0$ such that
  \BEQ\label{eq:sketch_saddle}
    f(x_{k_0}) - f(x_{k_r +1}) \geq \frac{\rho' \mu^2}{256 L_G} \Delta_{k_r}^2,
  \EEQ
  and the escape time is bounded as
  \begin{align*}
    k_r - k_0 \leq K_{\rm esc} && \textrm{ with } &&
  K_{\rm esc} = \tilde{\mathcal{O}}\!\left( \log\log\!\left( \frac{d}{\delta \sigma}\right) \right),
  \end{align*}
  where $\tildeO$ hides logarithmic dependence on function constants $(L_G, L_H, \mu, \gamma_s)$ and algorithm parameters $(\rho', \omega_2, \Delta_0)$.
\end{theorem}

% We note that, because of randomness, the function value might increase slightly for individual iterations between $k_0$ and $k_r$. However, overall decrease is guaranteed at iteration $k_r+1$.

Here too, we get a decrease guarantee upon lower-bounding $\Delta_k$ as done in Lemma~\ref{lemma:lowbound_deltak}.
It follows that the right-hand sides of both~\eqref{eq:sketch_lg} and~\eqref{eq:sketch_saddle} are lower bounded by some constant $\Flg >0$.

From \eqref{eq:sketch_saddle} we deduce that, \emph{on average}, each ``saddle point'' iteration decreases the objective by at least $\Flg/K_{\rm esc}$, while each successful ``large gradient'' iteration does so by $\Flg$.
After also arguing that not too many of those can be unsuccessful (Lemma~\ref{lem:mostlysuccessful}), we conclude that, with high probability, the method enters the neighborhood of a minimizer after at most ${K_{\rm esc}}(f(x_0)-\flow)/\Flg$ iterations.

Finally, we need to bound the number of Hessian-vector products required by $\tCGbg$: this essentially matches the number of inner iterations.
We do so in Section~\ref{s:complexity}.
For \emph{local minimizer} iterations, the model is strongly convex and the complexity of CG is well known.
In \emph{saddle point} neighborhoods, we rely on a result by \citet{Carmon2018a} which quantifies how fast Krylov-based methods are able to detect negative curvature.
As for \emph{large gradient} iterations, we establish a new result, noticing that the maximal model progress $\tCGbg$ can make is bounded due to the trust-region constraint (Lemma~\ref{lemma:cg_complexity_delta}).

The saddle escape theorem does much of the work, with the most new ingredients.
Let us sketch its proof next.

\paragraph{Proof sketch of saddle escape theorem.}
Let $x_{k_0}$ be close to a saddle point $\xbar$. We first show that the iterates \emph{do not converge} to $\xbar$.
Then we prove that there exists an escape iteration $k_r\geq k_0$ for which sufficient decrease is achieved.

Say $x := x_k$ is one of the iterates with $k \geq k_0$ that is still close to $\xbar$.
The quadratic model $m$ for $f$ around $x$ is non-convex as the Hessian at $x$ has at least one eigenvalue smaller than $-\mu/2$.
Consider running $T$ iterations of CG on $m$ initialized at $v^{(0)}$.
We show in Proposition~\ref{prop:vt_qi_growth} that the resulting iterates $v^{(t)}$ are ``pushed away'' by the negative curvature.
More precisely, for every $0 \leq t \leq T$,
\begin{align}
  |\la v^{(t)} - \bar v, q \ra |\geq \left(1 + \frac{\mu}{2L_G}\right)^{t} |\la v^{(0)} - \bar v, q \ra |,
  \label{eq:res_lowbound_sketch}
\end{align}
where $q$ is an eigenvector corresponding to a negative eigenvalue of the Hessian at $x$ and $\bar v$ is the critical point of the quadratic model $m$ around $x$.
To prove this, we show that CG makes at least as much progress in direction $q$ as gradient descent with step size $1/L_G$ (whose iterates can be computed in closed form).
Geometrically, this means that there is a region around the saddle point which the iterates never enter:
see Figure~\ref{fig:tcg_nonconvex} for an illustration.

\begin{figure}
  \centering
  \includegraphics[width=0.55\linewidth]{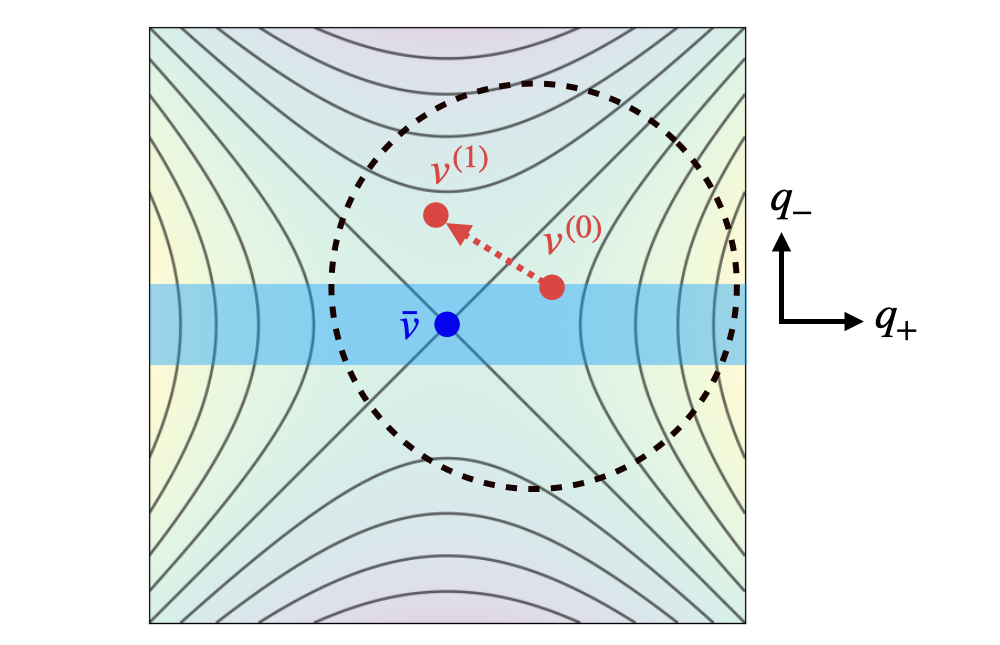}
  \caption{Illustration of the CG dynamics on a non-convex model $m$ (Proposition~\ref{prop:vt_qi_growth}). The vectors $q_+$ and $q_-$ are direction of respectively positive and negative curvature, and $\bar v$ is the saddle point. Since the projection of $v^{(t)}-\bar v$ on $q_-$ increases along the iterations, the iterates never enter the blue area. They eventually reach the boundary of the trust region (dotted line).}
  \label{fig:tcg_nonconvex}
\end{figure}

Then, randomness comes into play: because $v^{(0)} = \xi$ is distributed on a sphere of (typical) radius $\sigma$ (uniformly, up to sign), the right-hand side of \eqref{eq:res_lowbound_sketch} is lower bounded by a positive value with high probability.
Locally, the model residual is commensurate with the distance to the saddle point: we conclude that $\|x_k-\xbar\|$ is lower bounded, and that the iterates do not converge to the saddle.

To show the escape result, we examine the stopping criterion in $\tCGbg$.
If the $\stopres$ criterion triggers, the distance to the saddle is \emph{reduced} quadratically.
By our previous point, this cannot happen more than a few outer iterations in a row; eventually, the $\stopoob$ criterion must trigger and the iterates reach the boundary of the trust region.
Denote $k_r$ the first moment this happens:
\[
  \quad k_r = \inf\{ k \geq k_0 \,:\,k \textrm{ is successful and }\stopk{k} = \stopoob \}.
\]
We show that, at iteration $k_r$, the function value decreases sufficiently.
By the $\mu$-Morse property, the model gradient around saddle points satisfies the growth inequality
\[
  {\|\nabla m(v^{(T)})-g\| =\|\nabla m(v^{(T)})-\nabla m(0)\|  \geq \frac{\mu}{2} \|v^{(T)}\|},
\]
and $\|v^{(T)}\| = \Delta/2$ when the $\stopoob$ criterion triggers.
Therefore, either $\|g\|$ or $\| \nabla m(v^{(T)})\|$ is large, when $\Delta$ is large.
If $\|g\|$ is large, the first iteration of $\tCGbg$ (the Cauchy step) provides sufficient decrease.
Otherwise, $\|\nabla m(v^{(T)})\|$ is large and the decrease is achieved by performing a gradient step from $v^{(T)}$: \emph{this is why the new boundary gradient step helps escape saddles},
see Section~\ref{sec:tcgbgbenefits} for details.

\section{Deterministic supporting lemmas, old and new} \label{s:preliminaries}

In this section, we establish useful supporting lemmas and propositions for our analysis.
Section~\ref{sec:ineqsf} provides general inequalities regarding $f$ and its derivatives under Lipschitz and Morse assumptions.
Section~\ref{s:CGfacts} relates $\tCGbg$ to CG and tCG, and states mostly standard facts about the latter to shine a light on the former.
One possibly novel result there is the precise control of CG's behavior on non-convex quadratics as shown in Proposition~\ref{prop:vt_qi_growth}.
Section~\ref{sec:tcgbgbenefits} introduces mostly new results that quantify the exact benefits of the new boundary gradient step that was introduced to turn tCG into $\tCGbg$.
Finally, Section~\ref{s:progressTRbasics} provides basic results about Algorithm~\ref{algo:pTR} in terms of objective value decrease and behavior near critical points.

All statements in this section are deterministic and independent of the strong gradient property.
The section is designed in part so that we can discuss the saddle escape theorem next (Section~\ref{s:saddle}).

\subsection{Preliminary inequalities} \label{sec:ineqsf}

The trust-region algorithm relies on approximate minimization of second-order approximations of $f$ in the form $f(x + u) \approx f(x) + m_x(u)$, where the model $m_x \colon \Rd \to \reals$ around $x$ is defined by
\begin{align}
  m_x(u) & = \la \nabla f(x),u \ra + \frac{1}{2}\la \nabla^2 f(x)u,u \ra.
  \label{eq:mk}
\end{align}
Assumption~\ref{ass:f} ensures this approximation is accurate, in the sense that for all $x,u \in \reals^d$ we have
\begin{align}
  |f(x+u) - f(x) - m_x(u)| & \leq \frac{L_H}{6} \|u\|^3, \textrm{ and } \label{eq:lh_bound} \\
  |f(x+u) - f(x) - m_x(u)| & \leq \frac{1}{2}\Big( L_G \|u\|^2 + \big|\la \nabla^2 f(x)u,u \ra \big| \Big) \leq L_G \|u\|^2.
  \label{eq:lg_bound}
\end{align}
See for example~\citep[Lem.~1]{Nesterov2006}, where it is also shown that
\begin{align}
  \|\nabla f(x+u) - \nabla f(x) - \nabla^2 f(x)[u]\| \leq \frac{L_H}{2}\|u\|^2.
  \label{eq:LHnablabound}
\end{align}
Adding the $\mu$-Morse property (Definition~\ref{def:mu_morse}) provides further control on the gradient and Hessian of $f$ near critical points.

\begin{lemma}\label{lemma:growth_mu}
  Let $f \colon \reals^{d} \to \reals$ have an $L_H$-Lipschitz continuous Hessian. % (e.g., as included in Assumption~\ref{ass:f}).
  Assume $\xbar$ is a critical point of $f$ such that the singular values of $\nabla^2 f(\xbar)$ are all at least $\mu > 0$. % (e.g., as provided by the $\mu$-Morse property).
  Let $c \in [0, 1]$ and $x \in \reals^d$ be such that
  \[
    \|x - \xbar\| \leq \frac{\mu}{L_H} c.
  \]
  Then,
  \begin{enumerate}[label=(\roman*)]
    \item\label{item:growth_mu1} We have $\|\nabla f(x)\| \geq \mu(1 - \frac{c}{2})\|x-\xbar\|$ and $\|\nabla^2 f(x)u\| \geq \mu(1-c)\|u\|$ for all $u \in \reals^d$.
    % and therefore the singular values of $\nabla^2 f(x)$ belong to $[\mu(1-c),L_G]$.
    \item\label{item:growth_mu2} If $\xbar$ is a saddle point of $f$, then
    $
      \lambda_{\min}(\nabla^2 f(x)) \leq -\mu(1-c).
    $
    \item\label{item:growth_mu3} If $\xbar$ is a local minimizer of $f$, then
    $
      \lambda_{\min}(\nabla^2 f(x)) \geq \mu(1-c).
    $
  \end{enumerate}
\end{lemma}
\begin{proof}
  The proof is standard, see Appendix~\ref{app:additional}.
\end{proof}

\subsection{Facts about the conjugate gradients method} \label{s:CGfacts}

We begin by analyzing the first phase of $\tCGbg$ (Algorithm~\ref{algo:tCG}), for which the iterates are identical to those of standard CG~\citep{Hestenes1952CG}.
This phase lasts as long as the iterates remain in the interior of the ball of radius $\Delta/2$ (in particular, the method has not detected negative curvature).

We argue momentarily that $\tCGbg$ terminates in finite time.
Let $T$ denote the total number of iterations (as stated in the algorithm's output).
Further define
\begin{align}
  \Tin & = \max\!\left\{ 0 \leq t \leq T \,:\, \|v^{(0)}\|, \ldots, \|v^{(t)}\| < \frac{\Delta}{2} \right\}.
  \label{eq:def_tin}
\end{align}
By design, $\Tin$ equals either $T-1$ or $T$.
We highlight the following important fact:
\begin{quote}
  \begin{center}
    \textbf{For } $0 \leq t \leq \Tin$,\,\,\, \textbf{$\tCGbg$ is equivalent to standard CG}.
  \end{center}
\end{quote}
Thus, the sequences $v^{(t)}, r^{(t)}, p^{(t)}$ are identical to those of CG for $t$ up to $\Tin$ included.\footnote{In Algorithm~\ref{algo:tCG}, we initialize $v^{(0)}$ to a nonzero vector~$\xi$.
Several results on CG and tCG from the literature assume initialization at zero.
To apply them to our setting, we use the standard fact that CG on some quadratic with nonzero initialization is equivalent to CG on a shifted quadratic with zero initialization. See Lemma~\ref{lemma:cg_equiv}.}

We first recall classical properties of CG.
For $t \geq 1$ and $b \in \reals^d$, we consider Krylov subspaces
\begin{align}
  \mathcal{K}_t(H,b) = \mathrm{span}(b, Hb,\dots, H^{t-1} b).
  \label{eq:KrylovspaceKtHb}
\end{align}
The following properties are used throughout.

\begin{lemma}[Classical properties of CG]\label{lemma:well_known_props}
  Let $H \in \symm_d$ satisfy $\opnorm{H} \leq L_G$.
  Fix vectors $g, \xi \in \reals^d$ and a radius $\Delta > 0$ such that $\|\xi\| \leq \Delta / 4$.
  Run $\tCGbg(H, g, \Delta, \xi)$ (Algorithm~\ref{algo:tCG}) and let $\Tin$ be as defined in~\eqref{eq:def_tin}.
  Then for every $1 \leq t \leq \Tin$ we have
  % Note: if $r^{(0)} = 0$, then $T_in = T = 0$ so the lemma is correct because the interval 1 <= t <= T_in is empty. The claim about $r^{(0)}$ is true as stated too.
  \begin{enumerate}[label=(\roman*)]
    \item\label{item:pt_basis} $p^{(0)}, \ldots, p^{(t-1)}$ form a basis of $\mathcal{K}_t(H, r^{(0)})$,
    \item\label{item:res} $r^{(t)} = -\nabla m(v^{(t)}) = -(Hv^{(t)} + g)$ with $m(v) = \inner{g}{v} + \frac{1}{2} \inner{v}{Hv}$ (also true for $t = 0$),
    \item\label{item:kt} $v^{(t)}$ is the solution of the following minimization problem over a shifted Krylov space~\eqref{eq:KrylovspaceKtHb}:
    \begin{align*}
      \min_{v \in \reals^d} m(v) && \textrm{ subject to } && v \,\in\,
      v^{(0)} + \mathcal{K}_t(H,r^{(0)}),
    \end{align*}
    \item\label{item:grad_decrease} the model value decreases with each iteration as
    \begin{align*}
      m(v^{(t)}) - m(v^{(t-1)}) \leq-\frac{1}{2L_G} \|r^{(t-1)}\|^2,
    \end{align*}
    \item\label{item:psd_kt} the restriction of $H$ to $\mathcal{K}_t(H, r^{(0)})$ is positive definite, i.e.,
    \begin{align*}
      \la Hu,u \ra > 0 && \forall u \in \mathcal{K}_t(H, r^{(0)}), u \neq 0,
    \end{align*}
    \item\label{item:steplength} $v^{(t)} = v^{(t-1)} + \alpha^{(t)} p^{(t-1)}$ with $\alpha^{(t)} \geq \frac{1}{L_G}$.
  \end{enumerate}
  Moreover, the algorithm terminates in finite time with $\Tin \leq T \leq d$.
\end{lemma}
\begin{proof}
  We defer details to Appendix~\ref{app:additional}.
  Items~\ref{item:pt_basis}--\ref{item:kt} are classical.
  Item~\ref{item:grad_decrease} comes from the fact that a CG iteration is at least as good as a gradient descent step.
  Item~\ref{item:psd_kt} holds because up to $\Tin$ the method did not stop, hence it did not detect negative curvature.
  Item~\ref{item:steplength} follows from~\citep[Lem.~7]{Royer2020}.
  Finally, $T \leq d$ because CG is guaranteed to reach zero residual (for convex models) or to detect negative curvature (for non-convex models) in at most $d$ iterations.
\end{proof}

\paragraph{Boundedness of CG for strongly convex quadratics.}
At the heart of our analysis is the fact that CG behaves differently depending on whether the model is convex or not. In the strongly convex case, iterates stay bounded.
To prove this, we rely on a classical result by \citet{Steihaug1983tCG}.

\begin{lemma}\label{lemma:tcg_boundedness}
  Let $H \in \symm_d$ satisfy $H \succ 0$.
  Consider vectors $g, \xi \in \reals^d$ and a radius $\Delta > 0$ such that $\|\xi\| \leq \Delta/4$.
  Run $\tCGbg(H, g, \Delta, \xi)$ with output $T$.
  Then, the tentative iterates $v_+^{(t)}$ satisfy
  \begin{align*}
    \|v^{(t)}_+\| \leq 2\|\xi\| + \|H^{-1}g\| && \textrm{ for every } 0 \leq t \leq T-1.
  \end{align*}
\end{lemma}
\begin{proof}
  Since the model is strictly convex, for every $0 \leq t \leq T-1$ we have $\alpha^{(t+1)} > 0$ and $v_+^{(0)}, \ldots, v^{(T-1)}_+$ are the same iterates as those of standard CG.
  We now invoke \citet[Thm.~2.1]{Steihaug1983tCG} to state that the distance of the CG iterates to the initial point is increasing along the iterations.\footnote{\citet{Steihaug1983tCG} assumes $v^{(0)} = 0$; we apply it to our setting using Lemma~\ref{lemma:cg_equiv}.}
  Again by strict convexity of the model, the CG iterates converge to the unique minimizer $v^* = -H^{-1}g$.
  As a consequence,
  \[
    \|v_+^{(0)} - v^{(0)} \|\leq \dots \leq \|v_+^{(T-1)} - v^{(0)}\| \leq \|v^* - v^{(0)}\|.
  \]
  It follows that for every $0 \leq t \leq T-1$, $\|v^{(t)}_+\| \leq \|v^{(0)}\| + \|v^* - v^{(0)}\| \leq 2\|v^{(0)}\|+ \|v^*\|$.
\end{proof}

\paragraph{Divergence of CG for non-convex quadratics.}
If the model is not convex, the CG iterates (and its residuals) diverge exponentially along the negative curvature directions.
Though intuitive, this precise quantification seems novel.
It is an essential piece of the saddle escape theorem.

\begin{proposition}\label{prop:vt_qi_growth}
  Let $H \in \symm_d$ satisfy $\opnorm{H} \leq L_G$.
  Consider vectors $g, \xi \in \reals^d$ and a radius $\Delta > 0$ such that $\|\xi\| \leq \Delta/4$.
  Define ${\bar v = - H^{+}g}$, with $H^+$ the pseudoinverse of $H$.
  Let $q$ denote an eigenvector of $H$ associated to a negative eigenvalue $\lambda < 0$.

  Then the iterates of $\tCGbg(H, g, \Delta, \xi)$ satisfy, for all $0 \leq t \leq \Tin$ with $\Tin$ as in~\eqref{eq:def_tin},
  \begin{align}
    |\la v^{(t)} - \bar v, q \ra |\geq \left(1 + \frac{|\lambda|}{L_G}\right)^{t} |\la v^{(0)} - \bar v, q \ra |.
    \label{eq:vt_qi_growth}
  \end{align}
  This further implies some growth in the residuals since $\|r^{(t)}\| \geq |\lambda| |\la v^{(t)} - \bar v, q \ra|$ if $\|q\| = 1$.
\end{proposition}
\begin{proof}
 All statements in this proof are understood to hold for every $1 \leq t \leq \Tin$ (some also hold for $t = 0$, trivially).
 Owing to the update rule $p^{(t)} = r^{(t)} + \beta^{(t)} p^{(t-1)}$ with $\beta^{(t)} \geq 0$, we have
 \begin{align*}
  p^{(t)} = r^{(t)} + \sum_{s = 0}^{t-1} \zeta_{t,s} r^{(s)},
  \label{eq:decomp_rt}
 \end{align*}
 where $\zeta_{t,s}$ are some nonnegative coefficients.
 Since $t \leq \Tin$, we have $v^{(t-1)}_+ = v^{(t)}$, and therefore
 \begin{align*}
    v^{(t)} -\bar v & = v^{(t-1)}-\bar v + \alpha^{(t)}p^{(t-1)} \\
      & = v^{(t-1)}-\bar v + \alpha^{(t)}r^{(t-1)} + \alpha^{(t)}\sum_{s=0}^{t-2} \zeta_{t-1,s} r^{(s)} \\
      & = v^{(t-1)}-\bar v - \alpha^{(t)}(Hv^{(t-1)}+g) - \alpha^{(t)}\sum_{s=0}^{t-2} \zeta_{t-1,s} (Hv^{(s)}+g).
 \end{align*}
 Since $q$ is in the image of $H$, notice that $\la g, q \ra = \la HH^+ g, q \ra = \la -H\bar v, q \ra = -\la \bar v, Hq \ra$.
 Take inner products with $q$ on both sides of the above.
 We find for every $s$ that $\la Hv^{(s)}+g,q \ra = \la v^{(s)}-\bar v,Hq \ra = \lambda \la v^{(s)}-\bar v,q \ra $.
 Assume without loss of generality that $\la v^{(0)}-\bar v,q \ra \geq 0$ (otherwise, flip the sign of $q$).
 Since $\lambda < 0$, we get
 \begin{align*}
  \la v^{(t)} - \bar v, q \ra & = \left(1+\alpha^{(t)}|\lambda| \right)  \la v^{(t-1)} - \bar v,q \ra + \alpha^{(t)} |\lambda|  \sum_{s=0}^{t-2} \zeta_{t-1,s} \la v^{(s)} - \bar v, q \ra.
 \end{align*}
 Remembering that $\alpha^{(t)}$ and $\zeta_{t-1,s}$ are nonnegative, it follows by induction that $\la v^{(t)} - \bar v, q \ra \geq 0$, and therefore
 \begin{align*}
  \la v^{(t)} - \bar v, q \ra \geq \left(1 + \alpha^{(t)} |\lambda| \right)  \la v^{(t-1)} - \bar v, q \ra.
 \end{align*}
 To conclude, we use the fact that $\alpha^{(t)} \geq 1/L_G$ by Lemma~\ref{lemma:well_known_props}\ref{item:steplength}.

 The last statement about the residuals $r^{(t)} = -(Hv^{(t)} + g)$ holds since, if $\|q\| = 1$, then
 \begin{align*}
  \|r^{(t)}\| \geq |\la r^{(t)}, q \ra| = |\la H v^{(t)} + g, q \ra| = |\la v^{(t)} - \bar v, Hq \ra| = |\lambda| |\la v^{(t)} - \bar v, q \ra|.
 \end{align*}
 (Here, we see that the residuals, too, grow along the direction $q$).
\end{proof}

\subsection{Decrease guarantees for $\tCGbg$ and benefits of the boundary gradient step} \label{sec:tcgbgbenefits}

We now establish a lower bound on the objective decrease achieved by Algorithm~\ref{algo:tCG}.
To do so, we build on the CG results above and incorporate an analysis specific to $\tCGbg$, which takes into account the ball constraint and the new boundary gradient step.

We rely on the following lemma, which will be used to analyze the boundary gradient step. It is similar to the analysis of the \emph{Cauchy step} in standard tCG~\citep{Steihaug1983tCG}.
\begin{lemma}\label{lemma:cauchy_step}
  Let $H \in \symm_d$ satisfy $\opnorm{H} \leq L_G$.
  Consider vectors $g, v \in \reals^d$ and a radius $\Delta > 0$.
  Assume $\|v\| \leq \Delta / 2$.
  Run $\boundarygradientstep(H,g,\Delta,v)$ with output $u$ (Algorithm~\ref{algo:bg}).
  Then,
  \[
   m(v) - m(u) \geq \begin{cases}
    \frac{1}{2L_G}\|\nabla m(v)\|^2 & \textrm{ if } \|u\| < \Delta, \\
    \frac{1}{4}\Delta \|\nabla m(v)\| & \textrm{ if } \|u\| = \Delta,
  \end{cases}
  \]
  where $m$ denotes the quadratic $m(w) = \la g, w \ra + \frac{1}{2} \la w, Hw \ra$.
\end{lemma}
\begin{proof}
 Let $r = -\nabla m(v) = -(g+Hv)$.
 Let $s' \geq 0$ be the scalar such that $u = v + s' r$.
 Observe that
 \begin{equation}
  m(u) - m(v) = -s'\|r\|^2 + \frac{(s')^2}{2}\la Hr,r \ra.
  \label{eq:exp_v1}
 \end{equation}
 We distinguish three cases, according to the conditional statement in Algorithm~\ref{algo:bg}.
 From that algorithm, recall the definition of $\bar{\alpha} = \|r\|^2 / \langle r, H r \rangle$.

 \noindent
 \textbf{Case 1:} $\|u\| < \Delta$.
 Then, we did not enter the ``if'' statement.
 Thus, $s' = \bar \alpha$, $\la r,Hr \ra > 0$ and
 \begin{equation*}
  m(u) - m(v) = -\frac{\|r\|^4}{2 \la r,Hr \ra} \leq  - \frac{1}{2L_G}\|r\|^2.
 \end{equation*}

 \noindent
 \textbf{Case 2:} $\|u\| = \Delta$ and $\la r, Hr \ra > 0$.
 This implies that $\|v_+\| \geq \Delta$ where $v_+ = v + \bar \alpha r$, and $s'$ is the unique positive real such that $\|v+s'r\| = \Delta$ (uniqueness follows from the fact that $\|v\| < \Delta$ by assumption).
 We deduce that $0 < s' \leq \bar \alpha$ (since $\bar \alpha$ is positive) and so
 \begin{equation*}
  m(u) - m(v) = s' \left(
  -\|r\|^2 + \frac{s'}{2} \la r, Hr \ra
  \right)
  \leq s' \left(
  -\|r\|^2 + \frac{ \bar \alpha}{2} \la r, Hr \ra
  \right)
  = -\frac{1}{2}s'\|r\|^2.
 \end{equation*}

 \noindent
 \textbf{Case 3:} $\|u\| = \Delta$ and $\la r, Hr \ra \leq 0$.
 From \eqref{eq:exp_v1} we deduce that
 \begin{equation*}
  m(u)-m(v) \leq -s'\|r\|^2.
 \end{equation*}
 It remains to lower bound $s'$ in Cases 2 and 3.
 Because $\|v\|\leq \Delta/2$ and $\|v+s'r\| = \Delta$, the triangle inequality gives
 \begin{equation*}
  \Delta \leq \|v\| + s'\|r\| \leq \frac{\Delta}{2} + s'\|r\| \quad \implies \quad s' \geq \frac{\Delta}{2\|r\|},
 \end{equation*}
 which leads to the desired inequality.
\end{proof}

Using this lemma, we provide a lower bound on the model decrease achieved by $\tCGbg$.

\begin{lemma}[Lower bound on model decrease] \label{lemma:cauchy_decrease}
  Let $H \in \symm_d$ satisfy $\opnorm{H} \leq L_G$.
  Consider vectors $g, \xi \in \reals^d$ and a radius $\Delta > 0$.
  Assume $\|\xi\| \leq \Delta / 4$.
  Run $\tCGbg(H,g,\Delta,\xi)$ with output $(u, \stopcrit, T)$.
  Then,
  \[
    m(v^{(0)} ) - m(u) \geq W_1 + W_2 \geq 0,
  \]
  where $m(v) = \la g,v \ra + \frac{1}{2} \la v, Hv \ra$ and
  \[
  \setlength{\jot}{12pt}
  \begin{split}
    W_1 &= \begin{cases}
      \frac{1}{2L_G}\|r^{(0)}\|^2 & \textrm{ if } \|v^{(1)}\| < \Delta / 2, \\
      \frac{1}{8}\Delta \|r^{(0)}\| & \textrm{ if } \|v^{(1)}\| = \Delta / 2,
    \end{cases}\\
         %
     % \min\!\left( \frac{1}{2L_G}\|r^{(0)}\|^2, \frac{1}{8}\Delta \|r^{(0)}\| \right),\\
    W_2 &= \begin{cases}
      \min\!\left( \frac{1}{2L_G}\|r^{(T)}\|^2, \frac{1}{4}\Delta \|r^{(T)}\| \right) & \textrm{ if } \stopcrit = \stopoob, \\
      0 & \textrm{ otherwise}.
    \end{cases}
  \end{split}
  \]
  In particular, the denominator in the definition of $\rho_k$~\eqref{eq:rhok} is nonnegative.
\end{lemma}
\begin{proof}
  We decompose the decrease into three parts:
  \begin{align}
    m(v^{(0)}) - m(u) =
      \Big( m(v^{(0)}) - m(v^{(1)}) \Big) +
      \Big( m(v^{(1)}) - m(v^{(T)}) \Big) +
      \Big( m(v^{(T)}) - m(u) \Big).
    \label{eq:decomp}
  \end{align}
  \paragraph{Term 1.}
  Let us bound $m(v^{(0)}) - m(v^{(1)})$.
  By examining the first iteration of Algorithm~\ref{algo:tCG}, we notice that, formally, $v^{(1)}$ is equal to the output of
  \[
    \boundarygradientstep\!\left(H,g, \frac{\Delta}{2}, v^{(0)}\right)
  \]
  with $\|v^{(0)}\| \leq \Delta/4$. Lemma~\ref{lemma:cauchy_step} then yields $m(v^{(0)})-m(v^{(1)}) \geq W_1$, since $r^{(0)} = - \nabla m(v^{(0)})$.

  \paragraph{Term 2.}
  Note that $v^{(T)}$ is equivalent to the output of the standard tCG algorithm. Therefore, we use the monotonicity result of \citet[Thm.~2.1]{Steihaug1983tCG} to deduce that $m(v^{(1)}) - m(v^{(T)}) \geq 0$.

  \paragraph{Term 3.} Finally, we bound $m(v^{(T)}) - m(u)$.
  If $\stopcrit = \stopres$, then $u = v^{(T)}$ and the term is~$0$.
  Otherwise, $u$ is the output of
  \[\boundarygradientstep\!\left(H, g, \Delta, v^{(T)}\right)\]
  with $\|v^{(T)}\| = \Delta / 2$. We then apply Lemma~\ref{lemma:cauchy_step} to get the lower bound by $W_2$.
\end{proof}

\paragraph{Benefit of boundary step: decrease guarantee for small gradient models.}

Using the previous result, we show that if the Hessian is nondegenerate and $\tCGbg$ terminates with $\stopcrit = \stopoob$, then a sufficient decrease inequality is satisfied.
Crucially, this inequality \textbf{does not depend on the norm of $g$}, as shown next in Proposition~\ref{prop:decrease_tcg_nonconvex}.
This will be essential for analyzing Algorithm~\ref{algo:pTR} around saddle points, where the gradient norm is small.

The key is as follows: although the terms $W_1$ and $W_2$ in Lemma~\ref{lemma:cauchy_decrease} can both be small, they cannot both be small for a same instance.

\begin{proposition}\label{prop:decrease_tcg_nonconvex}
  Let $H \in \symm_d$ have all its singular values in $[\nu, L_G]$ for some $0 < \nu \leq L_G$.
  Consider vectors $g, \xi \in \reals^d$ and a radius $\Delta > 0$.
  Assume $\|\xi\| \leq \Delta / 4$ and $\la H\xi, g \ra \geq 0$.
  Run $\tCGbg(H,g,\Delta,\xi)$ with output $(u, \stopcrit, T)$.
  If $\stopcrit = \stopoob$, then
  \[
    m(v^{(0)}) - m(u) \geq \frac{\nu^2 \Delta^2}{32L_G}.
  \]
\end{proposition}
\begin{proof}
  % Note: $r^{(0)} \neq 0$ because otherwise we would have $\stop = \stopres$.
  First note that $\|r^{(0)}\| = \|H\xi + g\| \geq \|g\|$ owing to the assumption $\la H\xi, g \ra \geq 0$.
  From Lemma~\ref{lemma:cauchy_decrease}, we obtain
  \begin{align}
    m(v^{(0)}) - m(u)
      & \geq \min\!\left( \frac{1}{2L_G}\|r^{(0)}\|^2, \frac{\Delta}{8} \|r^{(0)}\| \right)
          +
        \min\!\left( \frac{1}{2L_G}\|r^{(T)}\|^2, \frac{\Delta}{4} \|r^{(T)}\| \right) \nonumber\\
      & \geq \min\!\left( \frac{1}{2L_G}\|g\|^2, \frac{\Delta}{8} \|g\| \right)
          +
        \min\!\left( \frac{1}{2L_G}\|Hv^{(T)}+g\|^2, \frac{\Delta}{4} \|Hv^{(T)}+g\| \right).
        \label{eq:bound_gHv}
  \end{align}
  There are two cases to consider.

  \noindent
  \textbf{Case 1:} $\|g\| \geq \nu \Delta / 4$.
  Then we keep only the first term in the right-hand side of~\eqref{eq:bound_gHv}:
  \[
    m(v^{(0)}) - m(u) \geq \min\!\left( \frac{\nu^2 \Delta^2}{32L_G}, \frac{\nu\Delta^2}{32} \right) = \frac{\nu^2 \Delta^2}{32L_G}.
  \]

  \noindent
  \textbf{Case 2:} $\|g\| \leq \nu\Delta/4$.
  We now only keep the second term in the right-hand side of~\eqref{eq:bound_gHv}.
  As the singular values of $H$ belong to $[\nu,L_G]$, we have
  \[
    \|Hv^{(T)}+g\| \geq \|Hv^{(T)}\|-\|g\| \geq \nu \|v^{(T)}\| - \|g\|.
  \]
  Since $\stopcrit = \stopoob$, $\|v^{(T)}\| = \Delta/2$.
  Thus, $\|Hv^{(T)}+g\| \geq \frac{1}{2}\nu\Delta - \frac{1}{4}\nu \Delta = \frac{1}{4} \nu \Delta$, so that
  \[
    m(v^{(0)}) - m(u) \geq \min\!\left(
      \frac{\nu^2 \Delta^2}{32L_G}, \frac{\nu\Delta^2}{16}
    \right) = \frac{\nu^2 \Delta^2}{32L_G},
  \]
  concluding the proof.
\end{proof}

\begin{remark}
  Proposition~\ref{prop:decrease_tcg_nonconvex} leads to a decrease guarantee of $\mathcal{O}(\frac{\mu^2}{L_G} \Delta^2)$ for objectives that satisfy the $\mu$-Morse property.
  In comparison, methods that use a \textbf{minimum eigenvalue oracle}, such as that of \cite{Goyens2024}, have a $\mathcal{O}(\mu \Delta^2)$ decrease guarantee (by computing $u$ as the direction of minimal eigenvalue of $H$).
  Our bound is worse by a factor of $\mu / L_G$, which is the theoretical price we pay for having an arguably more practical method.
  Indeed, using a minimum eigenvalue oracle can be costly and is unnecessary in practice.
\end{remark}

\subsection{Progress guarantees for the trust-region method} \label{s:progressTRbasics}

As the last part of this section regarding preliminary guarantees, we provide some control over the outer iterates, that is, the iterates $x_k$ of Algorithm~\ref{algo:pTR}.
That algorithm is randomized (through the perturbations $\xi_k$), but the results stated next are still deterministic.
The random nature of these perturbations will be exploited later, in Section~\ref{s:saddle}.

First, we estimate the progress in objective value.
Recall $g_k = \nabla f(x_k)$, $H_k = \nabla^2 f(x_k)$ and $m_{x_k}(v) = \inner{g_k}{v} + \frac{1}{2} \inner{v}{H_kv}$.

\begin{lemma}[Objective decrease]\label{lemma:obj_decrease}
  Consider iteration $k$ in Algorithm~\ref{algo:pTR}.
  Assume $\opnorm{H_k} \leq L_G$.
  If iteration $k$ is successful (that is, $\rho_k \geq \rho'$), then
  \begin{align*}
    f(x_{k}) - f(x_{k+1}) \geq \rho'\left[ m_{x_k}(\xi_k) - m_{x_k}(u_k) \right] -\|g_k\| \|\xi_k\| - \frac{L_G}{2} \|\xi_k\|^2.
  \end{align*}
\end{lemma}
\begin{proof}
  Recall the definitions of $\rho_k$ and $\theta_k = m_{x_k}(\xi_k) - m_{x_k}(0)$~\eqref{eq:rhok}.
  For a successful step, $\rho_k \geq \rho'$ so that $x_{k+1} = x_k + u_k$ and
  \begin{align*}
    f(x_{k}) - f(x_{k+1}) & = f(x_{k}) - f(x_{k} + u_k) + \theta_k - \theta_k \\
    & = \rho_k \left[ m_{x_k}(\xi_k) - m_{x_k}(u_k)  \right] - \theta_k, \\
    & \geq \rho' \left[ m_{x_k}(\xi_k) - m_{x_k}(u_k)  \right] - \theta_k,
  \end{align*}
  where we used $m_{x_k}(\xi_k) - m_{x_k}(u_k) \geq 0$ by Lemma~\ref{lemma:cauchy_decrease}.
  Conclude with $\theta_k = \la g_k, \xi_k \ra + \frac{1}{2}\la H_k \xi_k,\xi_k \ra \leq \|g_k\| \|\xi_k\| + \frac{L_G}{2} \|\xi_k\|^2$.
\end{proof}

The next guarantee is related to the behavior of Algorithm~\ref{algo:pTR} near a critical point, under the $\mu$-Morse property (Definition~\ref{def:mu_morse}).
We show that if $\tCGbg$ terminates with $\stopcrit = \stopres$, the next outer iterate gets quadratically closer to that critical point (be it a local minimizer or a saddle point).
The proof parallels standard arguments (more often written for $\xbar$ a local minimizer).

\begin{lemma}[Quadratic convergence near critical points]\label{lemma:quadratic_conv_crit}
  Let $f \colon \reals^{d} \to \reals$ satisfy Assumption~\ref{ass:f}.
  Assume $\xbar$ is a critical point of $f$ such that the singular values of $\nabla^2 f(\xbar)$ are all at least $\mu > 0$.
  Let $R > 0$ satisfy
  \[
    R \leq \min\!\left( \frac{\mu^2}{4L_HL_G}, \frac{\mu}{8\omega_2 L_G^2}\right).
  \]
  If at iteration $k$ of Algorithm~\ref{algo:pTR} we have $\|x_k - \xbar\| \leq R$ and iteration $k$ is successful with $\stopk{k} = \stopres$,
  then the next iterate satisfies
  \begin{equation} \label{eq:bound_quadconv}
    \|x_{k+1} - \xbar\| \leq \frac{\|x_k -\xbar\|^2}{2R} \leq \frac{\|x_k-\xbar\|}{2} \leq \frac{R}{2}.
  \end{equation}
\end{lemma}
\begin{proof}
  Since the step is successful, $x_{k+1} = x_k + u_k$.
  According to the $\stopres$ criterion, $u_k$ satisfies $\|H_k u_k + g_k\| \leq \min(\omega_1\|g_k\|,\omega_2 \|g_k\|^2)$.
  This implies
  \begin{align*}
    \|H_k(x_{k+1}-\xbar)\| &= \| H_k(x_{k} + u_k-\xbar) - g_k + g_k\|\\
      & \leq \| H_k(x_{k}-\xbar) - g_k \| + \|H_k u_k + g_k\|\\
      & \stackrel{\mathrm{(a)}}{\leq} \frac{L_H}{2}\|x_k - \xbar\|^2 + \min(\omega_1\|g_k\|,\omega_2 \|g_k\|^2)\\
      & \leq \frac{L_H}{2}\|x_k - \xbar\|^2 + \omega_2 \|g_k\|^2\\
      & \stackrel{\mathrm{(b)}}{\leq} \left(\frac{L_H}{2} + \omega_2 L_G^2\right) \|x_k - \xbar\|^2,
  \end{align*}
  where (a) holds because $0 = \nabla f(\xbar) = \nabla f(x_k) + \nabla^2 f(x_k)[\xbar - x_k] + w_k$ with $\|w_k\| \leq \frac{L_H}{2} \|x_k - \xbar\|^2$ due to Lipschitz continuity of the Hessian~\eqref{eq:LHnablabound}, and inequality (b) holds because $\|g_k\| = \|\nabla f(x_k) - \nabla f(\xbar)\| \leq L_G \|x_k - \xbar\|$ due to Lipschitz continuity of the gradient.
  Since we have ${\|x_k - \xbar\| \leq R \leq \mu / (2L_H)}$ (because the assumptions imply $\mu \leq L_G$), we can also lower bound the left-hand side using Lemma~\ref{lemma:growth_mu} with $c = 1/2$ which yields $\frac{\mu}{{2}}\|x_{k+1} - \xbar\|\leq \|H_k(x_{k+1} - \xbar)\|$ and thus
  \BEQ\label{eq:quad_conv_former_label}
  \begin{split}
    \frac{\mu}{2}\|x_{k+1} - \xbar\| &\leq \left( \frac{L_H}{2} + \omega_2 L_G^2 \right) \|x_k - \xbar\|^2
    \leq \frac{\mu}{4R} \|x_k-\xbar\|^2,
  \end{split}
  \EEQ
   where the last bound holds due to $\frac{L_H}{2} \leq \frac{\mu}{8R}$ and $\omega_2 L_G^2 \leq \frac{\mu}{8R}$.
   %$R \leq \mu^2 / (4L_HL_G)\leq \mu / (4L_H)$ and $R \leq \mu/ (8\omega_2 L_G^2)$.
   %Inequality~\eqref{eq:bound_quadconv} follows.
\end{proof}

\section{Saddle point escape by randomization}\label{s:saddle}

We are ready to study Algorithm~\ref{algo:pTR} in the neighborhood of a strict saddle point $\xbar$.
The goal is to prove the saddle escape theorem (Theorem~\ref{thm:saddle_escape}), which is made possible by the random perturbations $\xi_k$.

The explicit requirements on $f$ in this section hold if $f$ satisfies Assumption~\ref{ass:f} and the $\mu$-Morse property at $\xbar$, so that the singular values of $\nabla^2 f(\xbar)$ lie in $[\mu, L_G]$, and at least one eigenvalue is negative.
We also assume the trust-region radii remain bounded away from zero: that is guaranteed later (Lemma~\ref{lemma:lowbound_deltak}) with additional assumptions on $f$.

Our strategy is to prove that, with high probability, after a small number of outer iterations which may remain close to the saddle point, $\tCGbg$ must eventually trigger the $\stopoob$ criterion, thereby causing sufficient decrease in $f$ thanks to the new boundary gradient step.
Intuitively, once $f(x_k)$ drops sufficiently below $f(\xbar)$, we have ``escaped'' $\xbar$ because $f(x_k)$ is (approximately) monotonically decreasing in $k$.

We partition outer iterations into \emph{successful} and \emph{unsuccessful} ones:
\begin{align}
  \mathcal{S} = \{ k \geq 0 \,:\, \rho_k \geq \rho' \} && \textrm{ and } &&
  \mathcal{U} = \{ k \geq 0 \,:\, \rho_k < \rho'\}.
  \label{eq:calScalU}
\end{align}
As Algorithm~\ref{algo:pTR} has randomness coming (only) from $\bar\xi_k$, the iterates and radii are random variables.
It can be seen that the sequence $\{(x_k, \Delta_k)\}_{k \geq 0}$ is a Markov process.

\paragraph{Residual lower bound.}
The first step is to prove that the method \emph{is unlikely to converge to the saddle point $\xbar$}.
To this end, we first combine Proposition~\ref{prop:vt_qi_growth}---which describes the divergence of CG for non-convex models---with a concentration inequality on the random perturbations $\xi_k$, in order to show that the CG residuals are likely to stay bounded away from zero if $x_k$ is near $\xbar$.

\begin{lemma}\label{lemma:reslowerbound}
  Let $f \colon \reals^{d} \to \reals$ satisfy Assumption~\ref{ass:f}.
  Assume $\xbar$ is a saddle point of $f$ such that the singular values of $\nabla^2 f(\xbar)$ are all at least $\mu > 0$.
  Let $\delta \in (0, 1)$ be a confidence parameter.

  If at iteration $k$ of Algorithm~\ref{algo:pTR} we have $\|x_k-\xbar\| \leq \mu/(2L_H)$,
  then, with probability at least $1-\delta$, the residuals of $\tCGbg$ are bounded as
  \begin{align*}
    \|r_k^{(t)}\| \geq \delta \mu \|\xi_k\| \sqrt{\frac{\pi}{32d}} && \forall t \in \{ 0, \ldots, \Tin^{(k)} \},
  \end{align*}
  where $\Tin^{(k)}$ is as in~\eqref{eq:def_tin}. % so that $\|v_k^{(t)}\| < \frac{\Delta_k}{2}$ for all $t \leq \Tin^{(k)}$.
\end{lemma}
\begin{proof}
  Since $\|x_k-\xbar\|\leq \mu/(2L_H)$, by Lemma~\ref{lemma:growth_mu}\ref{item:growth_mu1}--\ref{item:growth_mu2} the Hessian $H_k = \nabla^2 f(x_k)$ is nonsingular and has at least one eigenvalue $\lambda$ smaller than~$-\mu/2$.
  Let $q$ be a unit eigenvector of $H_k$ associated to $\lambda$.
  Also let $\bar v_k = -H_k^{-1}g_k$.
  % We have for every $1 \leq t \leq \Tin^{(k)}$:
  % \begin{align*}
  %   \|r_k^{(t)}\| &= \|H_k v_k^{(t)} + g_k \|\\
  %   &\geq |\la H_k v_k^{(t)} + g_k, q \ra|\\
  %   &= |\la v_k^{(t)} - \bar v_k, H_k q \ra| \\
  %   &= |\lambda\la v_k^{(t)} -\bar v_k, q \ra|.
  % \end{align*}
  Proposition~\ref{prop:vt_qi_growth} then yields
  \begin{align*}
    \|r_k^{(t)}\|\geq |\lambda|\left( 1+ \frac{|\lambda|}{L_G} \right)^t |\la v_k^{(0)} - \bar v_k , q \ra| \geq \frac{\mu}{2} |\la v_k^{(0)} - \bar v_k , q \ra|  = \frac{\mu}{2}|\la \xi_k, q \ra - \la \bar v_k, q \ra|.
  \end{align*}
  Apply the concentration inequality of Lemma~\ref{lemma:concentration_xi} with $c = -\la \bar v_k, q\ra$ to deduce that
  \begin{align*}
    | \la \xi_k,q \ra - \la \bar v_k,q \ra | \geq \delta \|\xi_k\| \sqrt{\frac{\pi}{8d}}
  \end{align*}
  holds with probability at least $1-\delta$.
  Combine to conclude.
\end{proof}

\paragraph{Escape time bound.}
Next, we show that if the residuals stay bounded away from zero (as provided by Lemma~\ref{lemma:reslowerbound}), then there must be an outer iteration for which the $\stopoob$ criterion triggers in $\tCGbg$.

Let us assume for now that the radius is lower bounded as $\Delta_k \geq \Deltainf$ for every iteration $k$.
An explicit value for $\Deltainf$ (depending on problem constants) is given in Section~\ref{ss:radius_lowbound}.

\begin{proposition}[Escape time bound]\label{prop:escape}
  Let $f \colon \reals^{d} \to \reals$ satisfy Assumption~\ref{ass:f}.
  Assume $\xbar$ is a saddle point of $f$ such that the singular values of $\nabla^2 f(\xbar)$ are all at least $\mu > 0$.
  Let $R > 0$ satisfy
  \begin{equation} \label{eq:R_cond}
    R \leq \min\!\left( \frac{\mu^2}{4L_HL_G}, \frac{\mu}{8\omega_2 L_G^2} \right).
  \end{equation}
  Running Algorithm~\ref{algo:pTR}, assume $\Delta_k \geq \Deltainf$ for all $k \geq 0$.

  Condition on the existence of an iteration $k_0$ such that $\|x_{k_0} - \xbar\| \leq R$.
  Let $k_r$ be the index (possibly infinite) defined by
  \begin{align}
    k_r = \inf\!\Big\{ k \geq k_0 \,:\, k \in \mathcal{S} \textup{ and } \stopk{k} = \stopoob \Big\},
    \label{eq:krpropescapetime}
  \end{align}
  that is, $k_r$ is the first successful iteration after $k_0$ for which $\tCGbg$ hits the trust-region boundary.
  Then,
  \begin{enumerate}[label=(\roman*)]
    \item\label{item:saddlecapture} for all $k \in \{k_0, \ldots, k_r\}$ we have $\|x_k - \xbar\| \leq R$, and
    \item\label{item:escaptime} for all $\delta \in (0, 1)$, we have with probability at least $1-\delta$ that
    \[
      k_r \leq k_0 + \bKrs + \bKru,
    \]
    where $\bKrs$ and $\bKru$ are bounds on the number of successful and unsuccessful iterations made between $k_0$ and $k_r$ respectively, with
    \begin{align}
      \bKrs = \log_2 \log_2\!\left(2+\frac{\sqrt{d}}{ \omega_2 \mu \delta \min(\sigma, \Deltainf/4)} \right) && \textrm{ and } &&
      \bKru = \log_4\!\left(\frac{\bar \Delta}{\Deltainf}\right).
      \label{eq:def_krs_kru}
    \end{align}
  \end{enumerate}
\end{proposition}
\begin{proof}
  We first show item~\ref{item:saddlecapture} by induction on $k$.
  The claim holds for $k = k_0$.
  Assume now that it holds for some $k \in \{k_0, \ldots, k_r-1\}$.
  If iteration $k$ is unsuccessful, then $x_{k+1} = x_k$ so $\|x_{k+1} - \xbar\| \leq R$.
  Otherwise, by definition of $k_r$, we know $k\in \mathcal{S}$ and $\stopk{k}=\stopres$.
  Lemma~\ref{lemma:quadratic_conv_crit} yields
  \begin{equation} \label{eq:quad_conv}
    \| x_{k+1} - \xbar \| \leq \frac{\|x_k-\xbar\|^2}{2R} \leq \frac{\|x_k-\xbar\|}{2} \leq \frac{R}{2}.
  \end{equation}
  In particular, we also have $\|x_{k+1}-\xbar\|\leq R$, as desired.

  Let us prove item~\ref{item:escaptime}.
  To bound the number of unsuccessful iterations, note that $\Delta_{k_0} \geq \cdots \geq \Delta_{k_r}$ because for $k_0 \leq k < k_r$ we have $k \in \mathcal{U}$ or $\stopk{k} = \stopres$ (or both).
  For each unsuccessful step, the radius is divided by $4$.
  Thus, the number of unsuccessful steps is $\log_4\!\left(\frac{\Delta_{k_0}}{\Delta_{k_r}}\right) \leq \log_4\!\left( \frac{\bar \Delta}{\Deltainf}\right) =: \bKru$ because $\Delta_{k_r} \geq \Deltainf$ by assumption and $\Delta_{k_0} \leq \bar\Delta$ by design.
  For the number of successful iterations, randomness comes into play.

  To bound $k_r$ with high probability, assume that $k_r >  k_0 + \lfloor \bKru +\bKrs \rfloor$.
  We show that the probability of this event is smaller than $\delta$.
  Let $K_{r,u}$ and $K_{r,s}$ be defined as
  \begin{align}
    K_{r,u} & = \left| \, \mathcal{U} \cap \left\{ k_0,\dots, k_0 + \lfloor \bKru + \bKrs \rfloor  \right\} \right|, & K_{r,s} & = \left| \, \mathcal{S} \cap \left\{ k_0,\dots, k_0 + \lfloor \bKru + \bKrs \rfloor \right\} \right|.
  \end{align}
  Notice that $K_{r,u} + K_{r,s} = \lfloor \bKru + \bKrs \rfloor +1$.
  By the previous paragraph, we have $K_{r,u} \leq \lfloor \bKru \rfloor$, hence $K_{r,s} \geq \lfloor \bKrs \rfloor + 1$.
  Therefore there exists at least one successful step; let us denote by $\bar k = \max\!\left\{ k \in \mathcal{S} \cap \left\{ k_0,\dots, k_0 + \lfloor \bKru + \bKrs \rfloor  \right\} \right\} $ the last such step in this interval.
  From~\eqref{eq:quad_conv} we find that $\{x_k\}$ approaches $\xbar$ quadratically over successful iterations.
  Unrolling~\eqref{eq:quad_conv}, we get:
  \[
    \frac{\|x_{\bar k}-\xbar\|}{2R} \leq \left( \frac{\|x_{k_0}-\xbar\|}{2R} \right)^{2^{ K_{r,s} }} \leq  \left( \frac{1}{2} \right)^{2^{ K_{r,s} }}.
  \]
  By the definition of $\bar k$ and our assumption on $k_r$, we have $\stopk{\bar k} = \stopres$.
  Therefore, letting $T_{\bar k}$ be the index of the last inner step of $\tCGbg$ at outer iteration $\bar k$, we have
  \[
  \|r_{\bar k}^{(T_{\bar k})}\| \leq \min\!\left( \omega_1\|g_{\bar k}\|, \omega_2 \|g_{\bar k}\|^2 \right)
  \leq \omega_2 \|g_{\bar k}\|^2.
  \]
  Since $\nabla f(\xbar) = 0$, we have $\|g_{\bar k}\| = \|\nabla f(x_{\bar k}) - \nabla f(\xbar)\| \leq L_G \| x_{\bar k} - \xbar\|$, so that
  \begin{align}
    \|r_{\bar k}^{(T_{\bar k})}\|
      \leq \omega_2L_G^2 \|x_{\bar k} - \xbar\|^2
      \leq \omega_2L_G^2 R^2 \cdot2^{2-2^{ K_{r,s} +1 }}
      < \frac{1}{64\omega_2}\cdot 2^{2-2^{ \bKrs }},
    \label{eq:bound_rtk}
  \end{align}
where we used that $RL_G \leq \mu/(8\omega_2L_G) \leq 1 / (8\omega_2)$ and $K_{r,s}\geq \lfloor \bKrs \rfloor +  1 > \bKrs $.

We proved that if $k_r>k_0 + \lfloor \bKru + \bKrs \rfloor$, then \eqref{eq:bound_rtk} holds.
% , and also, $x_{\bar k} \in \ballR$ and $4\sigma \leq \Delta_{\bar k}$.
Therefore,
\[
  \Prob\!\left[k_r > k_0 + \lfloor \bKru + \bKrs \rfloor \right] \leq \Prob\!\left[\|r_{\bar k}^{(T_{\bar k})}\| < \frac{1}{16\omega_2}\cdot 2^{-2^{\bKrs}}\right].
\]
On the other hand, since $\|x_{\bar k} - \xbar\| \leq R \leq \mu / (2L_H)$, Lemma~\ref{lemma:reslowerbound} provides that
\[
  \Prob\!\left[\|r_{\bar k}^{(T_{\bar k})}\| <  \delta \mu \|\xi_k\| \sqrt{\frac{\pi}{32d}} \right] \leq \delta.
\]
Using the expression~\eqref{eq:def_krs_kru} for $\bKrs$, we note that
\[
  \frac{1}{16\omega_2}\cdot 2^{-2^{\bKrs}} = \frac{1}{16\omega_2} \cdot \frac{1}{ \left(2 + \frac{\sqrt{d}}{ \omega_2 \mu \delta \min(\sigma, \Deltainf/4) } \right)}
   \leq \frac{\mu\delta \min(\sigma, \Deltainf/4)}{16 \sqrt{d}}
   < \delta \mu \min(\sigma, \Deltainf/4) \sqrt{\frac{\pi}{32d}}.
\]
Combining the three previous expressions with $\|\xi_k\| = \min(\sigma, \Delta_k/4) \geq \min(\sigma, \Deltainf/4)$, we find
\begin{align*}
  \Prob\big[k_r > k_0 + \lfloor \bKru + \bKrs \rfloor  \big] \leq \Prob\!\left[\|r_{\bar k}^{(T_{\bar k})}\| <  \delta \mu \|\xi_k\| \sqrt{\frac{\pi}{32d}} \right]  \leq \delta,
\end{align*}
as announced.
% To cover the case where $K_{r,s}$ is $0$, we conclude\footnote{Indeed, when $K_{r,s} = 0$ the reasoning leading to the bound above is not valid, and the bound could be negative. If we add $2$ to the argument of the double logarithm, we solve this issue and get a bound that holds in both cases.} that
% $ K_{r,s}
% \leq \log_2 \log_2\!\left(2+\frac{\sqrt{d}}{ \omega_2 \mu \delta \sigma}\right)$.
% It follows that $k_r \leq \bar k = k_0 + \lfloor \bKru +\bKrs \rfloor$ with probability at least $1-\delta$, as announced.
\end{proof}

Combining the above with Proposition~\ref{prop:decrease_tcg_nonconvex}, we can further guarantee that the objective decreases by a substantial amount upon escaping a saddle.

\begin{theorem}[Saddle escape]\label{thm:saddle_escape}
  Let $f \colon \reals^{d} \to \reals$ satisfy Assumption~\ref{ass:f}.
  Assume $\xbar$ is a saddle point of $f$ such that the singular values of $\nabla^2 f(\xbar)$ are all at least $\mu > 0$.
  Running Algorithm~\ref{algo:pTR}, assume $\Delta_k \geq \Deltainf > 0$ for all $k \geq 0$.

  Condition on the existence of an iteration $k_0$ such that
  \begin{align*}
    \|x_{k_0} - \xbar\| \leq R && \textrm{ with some } && R \leq \min\!\left( \frac{\mu^2}{4L_HL_G}, \frac{\mu}{8\omega_2L_G^2} \right).
  \end{align*}
  Set a confidence parameter $\delta \in (0, 1)$.
  Assume the noise scale $\sigma$ satisfies these two conditions:
  \begin{align}
  \label{eq:cond_sigma1}
    &0<\sigma \leq  \min\!\left( \frac{\Deltainf}{4},
    R
      \right), \textrm{ and}\\
  \label{eq:cond_sigma2}
    & \left(  1+  \log_2 \log_2\!\left(2+\frac{\sqrt{d}}{ \omega_2 \mu \delta\sigma} \right)\right) \sigma \leq  \frac{\rho' \mu^2\Deltainf^2}{512 R L_G^2}.
  \end{align}
  (See Remark~\ref{remark:sigma_bar1} for a comment on~\eqref{eq:cond_sigma2}.)
  Let $k_r$ be defined as in~\eqref{eq:krpropescapetime}:
  \[
    k_r = \inf\!\Big\{ k \geq k_0\,:\, k \in \mathcal{S} \textup{ and } \stopk{k} = \stopoob \Big\}.
  \]
  Then, with probability at least $1-\delta$, we have $k_r \leq k_0 + \bKrs + \bKru$ with $\bKrs$ and $\bKru$ as in~\eqref{eq:def_krs_kru}.
  In that event, the objective value decreases by some amount from iteration $k_0$ to $k_r + 1$:
  \[
      f(x_{k_0})-f(x_{k_r+1})   \geq \frac{\rho' \mu^2\Delta_{k_r}^2}{256 L_G}.
  \]
\end{theorem}
\begin{proof}
Proposition~\ref{prop:escape} provides the bound $k_r \leq k_0 + \bKru + \bKrs$ with probability at least $1-\delta$.
Going forward, we condition on that event.
In particular, $k_r$ is finite.
To estimate the progress in objective value, we first use Lemma~\ref{lemma:obj_decrease}.
To this end, note that the function value only changes with successful iterations, and observe that $\|\xi_k\| = \min(\sigma, \Delta_k/4) = \sigma$ for all $k$ by assumption, thus
\begin{align*}
  f(x_{k_0}) - f(x_{k_r+1})
    & = \sum_{\substack{k\in \mathcal{S}, \\ k_0\leq k \leq k_r}} f(x_{k}) - f(x_{k+1}) \\
    & \geq \sum_{\substack{k\in \mathcal{S} \\ k_0\leq k \leq k_r}}
      \left( \rho'\left[ m_{x_k}(\xi_k) - m_{x_k}(u_k) \right]
              -\|g_k\|\sigma
              -\frac{1}{2}L_G \sigma^2
      \right).
\end{align*}
From the deterministic part of Proposition~\ref{prop:escape}, for $k_0 \leq k\leq k_r$, we have $\|x_k - \xbar\| \leq R$ and hence $\|g_k\|\leq L_G \| x_k-\xbar \| \leq L_G R$.
Also, $m_{x_k}(\xi_k) - m_{x_k}(u_k) \geq 0$ for all $k$ by monotonicity of $\tCGbg$ (Lemma~\ref{lemma:cauchy_decrease}), so that
\begin{align*}
  f(x_{k_0}) - f(x_{k_r+1})
    & \geq \rho' \left[ m_{x_{k_r}}(\xi_{k_r}) - m_{x_{k_r}}(u_{k_r}) \right]
            - \sum_{\substack{k\in \mathcal{S} \\ k_0\leq k \leq k_r}}\left(
                R L_G \sigma + \frac{1}{2} L_G \sigma^2
              \right) \\
    & \geq \rho' \left[ m_{x_{k_r}}(\xi_{k_r}) - m_{x_{k_r}}(u_{k_r}) \right]
            - (K_{r,s}+1) \left(
                R + \frac{\sigma}{2}
              \right)  L_G \sigma,
\end{align*}
where $K_{r,s}$ is now the number of successful iterations between $k_0$ and $k_r$.
To bound the model progress at iteration $k_r$, we apply Proposition~\ref{prop:decrease_tcg_nonconvex} with $\nu = \mu / 2$---the assumptions are satisfied because the singular values of $H_{k_r}$ belong to $[\frac{\mu}{2},L_G]$ owing to Lemma~\ref{lemma:growth_mu} with $c = 1/2$,
%$\|\xi_{k_r}\| = \sigma \leq \Deltainf / 4 \leq \Delta_{k_r} /4$ by assumption, ---- commented out because this is now automatic
and $\stopk{k_r} = \stopoob$ by definition of $k_r$.
This leads to
\begin{equation} \label{eq:dec_cond_12}
  f(x_{k_0}) - f(x_{k_r+1})
    \geq  \frac{\rho' \mu^2\Delta_{k_r}^2}{128 L_G}
           - (K_{r,s}+1) \left( R + \frac{\sigma}{2} \right)  L_G \sigma.
\end{equation}
To proceed, we use the estimate $K_{r,s} \leq \bKrs$ from the probabilistic part of Proposition~\ref{prop:escape}, as well as conditions~\eqref{eq:cond_sigma1} and~\eqref{eq:cond_sigma2} on $\sigma$:
\[
\begin{split}
  \left(
  R + \frac{\sigma}{2}
  \right)  L_G \sigma (K_{r,s}+1) &\leq  2R L_G \sigma ( \bKrs + 1 )\\
  &\leq \;2RL_G \sigma \left(  1+  \log_2 \log_2\!\left(2+\frac{\sqrt{d}}{ \omega_2 \mu \delta\sigma} \right)\right)
   \leq \frac{\rho' \mu^2\Deltainf^2}{256 L_G}.
\end{split}
\]
Combining with \eqref{eq:dec_cond_12}, and remembering that $\Delta_{k_r} \geq \Deltainf$, yields the result.
\end{proof}

\begin{remark}[Explicit condition on $\sigma$]\label{remark:sigma_bar1} In Theorem~\ref{thm:saddle_escape}, condition~\eqref{eq:cond_sigma2} on $\sigma$ requires
  \[
    \left( 1  +  \log_2 \log_2\!\left(2 + \frac{a}{\sigma} \right) \right) \sigma \leq b
  \]
  for some $a, b > 0$ which do not depend on $\sigma$.
  Invoking Lemma~\ref{lemma:sigma_loglog}, this holds in particular if
  \[
    0 < \sigma \leq \frac{b}{2\log_2(2 + \frac{2a}{b})}.
  \]
  Thus, there exists $\bar\sigma_{\rm saddle} > 0$ such that conditions~\eqref{eq:cond_sigma1} and~\eqref{eq:cond_sigma2} are satisfied for all $0 < \sigma \leq \bar\sigma_{\rm saddle}$. Note that, as announced in Section \ref{sss:informal_thm}, $\sigma_{\rm saddle}$ has only a logarithmic dependence in the dimension $d$ and the failure probability $\delta$.
\end{remark}

To exploit the objective decrease (probabilistically) guaranteed by Theorem~\ref{thm:saddle_escape}, one needs a lower bound on the trust-region radii.
We provide one in the next section (Lemma~\ref{lemma:lowbound_deltak}).

\section{Global convergence to a local minimizer} \label{s:global_comp}

Throughout this section, we think of $f$ as satisfying Assumptions~\ref{ass:f} and~\ref{ass:mumorse_and_stronggrad}.
Under these conditions, we establish convergence of the iterates $x_k$ of Algorithm~\ref{algo:pTR} to a local minimizer of $f$.

The assumptions provide constants $(\flow, L_G, L_H)$ as well as $(\mu, R_s, \gamma_s)$.
The algorithm itself also involves parameters $(\rho', \rho'', \bar\Delta, \omega_1, \omega_2)$ and inputs $(x_0, \Delta_0, \sigma)$.
Using some of these, we further define several constants used throughout the analysis.
First,
\begin{align}
  \Deltacrit & = \frac{(1-\rho') \mu^2 }{10L_HL_G}, &
  \bar R & = \min\!\left(
  \frac{(1-\rho')\gamma_s}{256 L_G},\frac{\mu}{8\omega_2 L_G^2},
  \frac{\Deltacrit}{32}, \frac{\Delta_0}{8}
  \right),
  \label{eq:def_constants_DeltacritRbar}
\end{align}
then
\begin{align}
  \bar G & = \frac{\mu}{2} \bar R, &
  \Flg & = \frac{\rho' }{5L_G} \bar{G}^2, &
  \bar \sigma & =
  \frac{\rho'}{5L_G} \bar G \leq \bar R.
  \label{eq:def_constants}
\end{align}
Notice that they do not depend on the initial point $x_0$ or the noise scale $\sigma$.
Later, we will require $\sigma \leq \bar \sigma$.
Also, notice that $32 \bar{R} \leq \Deltacrit \leq \mu^2 / (10 L_H L_G) \leq \mu / (10 L_H)$ since $\mu \leq L_G$.

Consider the iterates $\{(x_k,\Delta_k)\}_{k \geq 0}$ generated by Algorithm~\ref{algo:pTR}.
We partition the iteration indices $\mathbb{N} = \{0, 1, 2, \ldots\}$ into three sets relative to the threshold radius $\bar R$:
\begin{itemize}
  \item \textbf{Saddle point iterations:}
  \begin{align}
    \mathcal{N} = \left\{ k  \geq 0 \,:\, \mbox{there exists a strict saddle point } \xbar \mbox{ of } f \mbox{ such that } \|x_k - \xbar\| \leq \bar R \right\}.
    \label{eq:saddleiterations}
  \end{align}
  \item \textbf{Local minimizer iterations:}
  \begin{align}
    \mathcal{M} = \left\{ k  \geq 0 \,:\, \mbox{there exists a local minimizer } x^* \mbox{ of } f \mbox{ such that } \|x_k - x^*\| \leq \bar R \right\}.
    \label{eq:localminiterations}
  \end{align}
  \item \textbf{Large gradient iterations:} these are all the other iterations, namely,
  \begin{align}
    \mathcal{G} = \mathbb{N} \setminus \left( \mathcal{N} \cup \mathcal{M}\right).
    \label{eq:largegraditerations}
  \end{align}
\end{itemize}
Since $\bar R \leq \mu / (2L_H)$, we can infer from Lemma~\ref{lemma:growth_mu} (points~\ref{item:growth_mu2} and~\ref{item:growth_mu3}) that $\mathcal{N}$ and $\mathcal{M}$ are disjoint, and therefore that $(\mathcal{N},\mathcal{M},\mathcal{G})$ indeed forms a partition of $\mathbb{N}$.

Let us first justify the term ``large gradient'' attributed to the set $\mathcal{G}$: this is the first consequence of the $(\gamma_s, R_s)$-strong gradient property (part of Assumption~\ref{ass:mumorse_and_stronggrad}).

\begin{lemma}\label{lemma:lowbound_G}
  Under Assumptions~\ref{ass:f} and~\ref{ass:mumorse_and_stronggrad},
  for all $k \in \mathcal{G}$ we have $\|g_k\| \geq \bar G$, where $g_k = \nabla f(x_k)$.
\end{lemma}
\begin{proof}
  Let $k \in \mathcal{G}$.
  If $\|x_k -\xbar\| > R_s$ for every critical point $\xbar$ of $f$, then $\|g_k\| \geq \gamma_s$ by the strong gradient property (Assumption~\ref{ass:mumorse_and_stronggrad}).
  Otherwise, let $\xbar$ be a critical point such that $\|x_k - \xbar \| \leq R_s$.
  By definition of $\mathcal{G}$, we have $\|x_k - \xbar\| > \bar R$.
  Since $R_s\leq \mu / (2L_H)$, Lemma~\ref{lemma:growth_mu}\ref{item:growth_mu1}
  %with Assumption~\ref{ass:f}
  yields
  \begin{equation*}
    \|g_k\| \geq \frac{\mu}{2} \|x_k - \xbar \| \geq \frac{\mu}{2} \bar R.
  \end{equation*}
  Thus, in all cases, $\|g_k\| \geq \min(\gamma_s, \mu \bar R / 2)$.
  Owing to the definition of $\bar R$~\eqref{eq:def_constants_DeltacritRbar}, we also have
  \[
     \frac{\mu \bar R}{2} \leq \frac{\mu(1-\rho')\gamma_s}{256L_G}\leq \gamma_s
  \]
  because $\mu \leq L_G$ and $\rho' \in (0, 1)$.
  Therefore, $\min(\gamma_s, \mu\bar R/2) = \mu \bar R / 2 = \bar G$~\eqref{eq:def_constants}.
\end{proof}

\subsection{Lower bound on the trust-region radius (deterministic)} \label{ss:radius_lowbound}

The proof sketch in Section~\ref{s:proof_sketch} argues sufficient decrease in the objective function $f$ can be secured if the trust-region radius $\Delta_k$ remains bounded away from zero.
We show this now.
The first step is to prove that iteration $k$ is necessarily successful if $\Delta_k$ is small enough: this is the purpose of the next two lemmas.
This is relevant because successful iterations never trigger a reduction in the trust-region radius.

Recall $g_k = \nabla f(x_k)$ and $H_k = \nabla^2 f(x_k)$.
Also recall the partition of iteration indices into successful ($\mathcal{S}$) and unsuccessful ($\mathcal{U}$) steps~\eqref{eq:calScalU}.

\begin{lemma}[Success guarantee for large gradient iterations] \label{lemma:large_gradient_radius}
  Let $f \colon \reals^{d} \to \reals$ be $C^2$ with an $L_G$-Lipschitz continuous gradient.
  If at iteration~$k$ of Algorithm~\ref{algo:pTR} the radius $\Delta_k$ satisfies
  \begin{align}
    \Delta_k \leq \frac{(1-\rho')}{8L_G} \|g_k\|,
    \label{subeq:delta1}
  \end{align}
  then iteration $k$ is successful: $\rho_k \geq \rho'$, that is, $k \in \mathcal{S}$.
\end{lemma}
\begin{proof}
  We control $\rho_k$~\eqref{eq:rhok} using Lipschitz continuity of the gradient~\eqref{eq:lg_bound}, as follows:
  \begin{align}
    |\rho_k - 1| & = \left|\frac{f(x_k) - f(x_k + u_k) - \left[ m_{x_k}(0) - m_{x_k}(u_k) \right] }{m_{x_k}(0) - m_{x_k}(u_k) + \theta_k }\right| \nonumber\\
    & = \left| \frac{ f(x_k) + m_{x_k}(u_k) - f(x_k + u_k) }{ m_{x_k}(\xi_k) - m_{x_k}(u_k) }\right| \nonumber\\
    & \leq \frac{L_G\|u_k\|^2}{|m_{x_k}(\xi_k) - m_{x_k}(u_k)|}.
    \label{eq:bound_rhok_former_label}
  \end{align}
  To bound the denominator, we apply Lemma~\ref{lemma:cauchy_decrease}, keeping only the term~$W_1$:
  \BEQ\label{eq:recall_model_dec}
    m_{x_k}(\xi_k) - m_{x_k}(u_k)
      \geq
      \min\!\left(
        \frac{1}{2L_G}\|r^{(0)}_k\|^2,
        \frac{1}{8}\Delta_k \|r_k^{(0)}\|
      \right)
      \geq \min\!\left(
        \frac{1}{2L_G}\|g_k\|^2,
        \frac{1}{8}\Delta_k \|g_k\|
      \right),
  \EEQ
  recalling that $\|r_k^{(0)}\| = \|H_k \xi_k + g_k\| \geq \|g_k\|$ owing to $\la H_k\xi_k,g_k \ra \geq 0$ (see line~\ref{line:xik} in Algorithm~\ref{algo:pTR}).
  By design, $\|u_k\| \leq \Delta_k$ and $\rho' \in (0, 1)$.
  Then, combining the two previous bounds with~\eqref{subeq:delta1} yields
  \[
    1 - \rho_k \leq |\rho_k - 1| \leq  \max\!\left( \frac{2L_G^2 \Delta_k^2}{\|g_k\|^2},
    \frac{8L_G \Delta_k}{\|g_k\|} \right)
    \leq \max\!\left( \frac{(1-\rho')^2}{32}, 1-\rho' \right) = 1-\rho'.
  \]
  It follows that $\rho_k \geq \rho'$ and hence that iteration $k$ is successful.
\end{proof}

\begin{lemma}[Success guarantee for iterations near critical points]\label{lemma:radius}
  Let $f \colon \reals^{d} \to \reals$ satisfy Assumption~\ref{ass:f}.
  Assume $\xbar$ is a critical point of $f$ such that the singular values of $\nabla^2 f(\xbar)$ are all at least $\mu > 0$.
  If at iteration $k$ of Algorithm~\ref{algo:pTR} we have $\|x_k - \xbar\| \leq \mu^2/(4L_HL_G)$
  and the radius $\Delta_k$ satisfies
  \begin{align*}
    \Delta_k \leq \Deltacrit && \textrm{ with } && \Deltacrit = \frac{(1-\rho') \mu^2}{10L_HL_G} \textrm{ as in~\eqref{eq:def_constants_DeltacritRbar}},
  \end{align*}
  then iteration $k$ is successful: $\rho_k \geq \rho'$, that is, $k \in \mathcal{S}$.
\end{lemma}
\begin{proof}
  First note that because $\|x_k-\xbar\|$ is small enough, Lemma~\ref{lemma:growth_mu}\ref{item:growth_mu1} with $c = 1/4$ states that $\|H_k u\| \geq \frac{3\mu}{4} \|u\|$ for all $u \in \reals^d$.

  The argument reuses certain elements of the proof of Lemma~\ref{lemma:large_gradient_radius}.
  We control the ratio $\rho_k$~\eqref{eq:rhok} using Lipschitz continuity of the Hessian~\eqref{eq:lh_bound}, as follows:
  \BEQ\label{eq:bound_rhok}
    |\rho_k - 1| = \left| \frac{ f(x_k) + m_{x_k}(u_k) - f(x_k + u_k) }{ m_{x_k}(\xi_k) - m_{x_k}(u_k) }\right| \leq \frac{L_H \|u_k\|^3 }{6| m_{x_k}(\xi_k) - m_{x_k}(u_k)|}.
  \EEQ
  We distinguish two cases, depending on the reason $\tCGbg$ terminated.

  If $\stopk{k} = \stopres$, then all iterates of $\tCGbg$ remained inside the ball of radius $\Delta_k / 2$.
  Thus, reusing the fact that $\|r_k^{(0)}\| \geq \|g_k\|$, Lemma~\ref{lemma:cauchy_decrease} implies
  \[
    m_{x_k}(\xi_k)- m_{x_k}(u_k) \geq \frac{1}{2L_G} \|r_k^{(0)}\|^2 \geq \frac{1}{2L_G}  \|g_k\|^2.
  \]
  Using also $\|u_k\| \leq \Delta_k / 2$, it follows that
  \begin{align}
    |\rho_k - 1| & \leq \frac{L_G L_H\|u_k\|^3}{3\|g_k\|^2} \leq \frac{L_G L_H\|u_k\|^2}{3\|g_k\|^2} \cdot \frac{\Delta_k}{2}.
    \label{eq:bound_rhok_res}
  \end{align}
  Since $\tCGbg$ terminated with $\stopk{k}=\stopres$, we further know that
  \[\|H_k u_k + g_k\| \leq \min(\omega_1 \|g_k\|,\omega_2\|g_k\|^2) \leq \omega_1 \|g_k\| \leq \|g_k\|. \]
  As a result, using $\|H_k u\| \geq \frac{3\mu}{4} \|u\|$ with $u = u_k$, we get
  \begin{align}
    \|u_k\| \leq \frac{4}{3\mu} \|H_ku_k\| \leq \frac{4}{3\mu} \big( \|g_k\| + \|H_ku_k + g_k\| \big) \leq \frac{8}{3\mu}\|g_k\|.
    \label{eq:uk_res_bound}
  \end{align}
  Using this bound in \eqref{eq:bound_rhok_res} together with $\Delta_k \leq \Deltacrit$ yields
  \begin{align*}
    1 - \rho_k \leq |\rho_k - 1| \leq \frac{L_HL_G}{3\|g_k\|^2}\cdot \frac{64\|g_k\|^2}{9\mu^2} \cdot\frac{\Deltacrit}{2} = \frac{64L_HL_G}{54\mu^2} \Deltacrit \leq 1-\rho',
  \end{align*}
  which implies that $\rho_k \geq \rho'$ and hence $k \in \mathcal{S}$.

  If $\stopk{k} = \stopoob$, we apply Proposition~\ref{prop:decrease_tcg_nonconvex}.
  Its assumptions hold with $\nu = 3\mu / 4$, so that
  \begin{align}
    m_{x_k}(\xi_k) - m_{x_k}(u_k) \geq \left(\frac{3}{4}\right)^2 \frac{\mu^2\Delta_k^2}{32 L_G}.
    \label{eq:moddecreaseoob}
  \end{align}
  Then, inequality~\eqref{eq:bound_rhok} implies, using the fact that $\|u_k\| \leq \Delta_k$ and $\Delta_k\leq \Deltacrit$,
  \[
    1 - \rho_k \leq |\rho_k - 1| \leq \frac{L_H\Delta_k^3}{6}\cdot \left(\frac{4}{3}\right)^2 \cdot \frac{32L_G}{\mu^2 \Delta_k^2} \leq \frac{256L_HL_G}{27\mu^2} \Deltacrit \leq 1-\rho',
  \]
  and therefore step $k$ is successful.
\end{proof}

Combining these two results, we deduce that the trust-region radius $\Delta_k$ is lower bounded.
This is the second consequence of the strong gradient property (part of Assumption~\ref{ass:mumorse_and_stronggrad}).
Notice that it is deterministic and that it holds regardless of the noise scale.

\begin{lemma}[radius lower bound] \label{lemma:lowbound_deltak}
  Under Assumptions~\ref{ass:f} and~\ref{ass:mumorse_and_stronggrad},
  let $\bar R$ be as in~\eqref{eq:def_constants_DeltacritRbar}.
  Then
  \begin{align}
    \Delta_k \geq 8 \bar R && \textrm{ for all } k = 0, 1, 2 \ldots
    \label{eq:radius_lowbound}
  \end{align}
\end{lemma}
\begin{proof}
  We first prove the following statement for every $k$:
  \begin{align}
    \textrm{If } \Delta_k \leq 32 \bar R,  \textrm{ then } \textrm{ iteration $k$ is successful.}
    \label{eq:rec_deltak}
  \end{align}
  Let $k\geq 0$. There are two cases to consider.

  First, assume that there exists a critical point $\bar x$ of $f$ such that ${\|x_k-\bar x\| \leq \mu^2 / (4L_HL_G)}$.
  Under the assumption of~\eqref{eq:rec_deltak}, Lemma~\ref{lemma:radius} applies because $\Delta_k \leq 32 \bar R \leq \Deltacrit$ by definition~\eqref{eq:def_constants_DeltacritRbar}.
  Thus, iteration $k$ is indeed successful.

  Second, assume $x_k$ is at least at a distance $\mu^2 / (4L_HL_G)$ from every critical point of $f$.
  Then, we have $\|g_k\| \geq \gamma_s$ by the strong gradient property (Assumption~\ref{ass:mumorse_and_stronggrad}).
  Under the assumptions of~\eqref{eq:rec_deltak}, Lemma~\ref{lemma:large_gradient_radius} applies because
  \[
    \Delta_k \leq 32 \bar R \leq \frac{(1-\rho') \gamma_s}{8L_G} \leq \frac{(1-\rho') \|g_k\|}{8L_G}
  \]
  by definition~\eqref{eq:def_constants_DeltacritRbar}.
  Thus, iteration $k$ is indeed successful.

  The above reasoning confirms that~\eqref{eq:rec_deltak} holds for all $k$.
  We now prove the target inequality~\eqref{eq:radius_lowbound} by induction.
  For $k = 0$, it holds because $\bar R \leq \Delta_0 / 8$ by definition~\eqref{eq:def_constants_DeltacritRbar}.
  Now assume $\Delta_k \geq 8 \bar R$ holds for some $k \geq 0$.
  Then, either $\Delta_k \leq 32 \bar R$, in which case~\eqref{eq:rec_deltak} provides that iteration $k$ is successful and hence $\Delta_{k+1} \geq \Delta_k \geq 8 \bar R$; or $\Delta_k > 32 \bar R$, in which case $\Delta_{k+1} \geq \Delta_k / 4 \geq 8 \bar R$, concluding the proof.
\end{proof}

Let us note already here that a corollary of having a lower bound on $\Delta_k$ is that most outer iterations are successful.
This is a standard, mechanical consequence of the trust-region radius update rule~\citep{Cartis2012}. % NB: I looked in the ConnGouldToint book, but didn't find this natural argument there.
\begin{lemma} \label{lem:mostlysuccessful}
  Running Algorithm~\ref{algo:pTR}, let $|\mathcal{S}|_k := |\{ 0 \leq \ell < k : \rho_\ell \geq \rho' \}|$ be the number of successful iterations completed before iteration $k$.
  Then,
  \begin{align*}
    k \leq \frac{3}{2} |\mathcal{S}|_k + \frac{1}{2} \log_2\!\left( \frac{\Delta_0}{\Delta_k} \right).
  \end{align*}
\end{lemma}
\begin{proof}
  % Similar to $|\mathcal{S}|_k$,
  Let $|\mathcal{U}|_k := |\{ 0 \leq \ell < k : \rho_\ell < \rho' \}|$, so that $k = |\mathcal{S}|_k + |\mathcal{U}|_k$.
  Observe that
  \[
    \Delta_k \leq \frac{2^{|\mathcal{S}|_k} \Delta_0}{4^{|\mathcal{U}|_k}}
  \]
  because each successful step at most doubles the trust-region radius, while each unsuccessful step divides the radius by four.
  Reorganize as
  $
    \log_2\!\left( \frac{\Delta_0}{\Delta_k} \right) \geq 2 |\mathcal{U}|_k - |\mathcal{S}|_k
  $
  and substitute $|\mathcal{U}|_k = k - |\mathcal{S}|_k$.
\end{proof}

\subsection{Objective decrease for large gradient steps (deterministic)}

From the beginning of this section, recall the partition of the iterations of Algorithm~\ref{algo:pTR} into large gradient steps ($\mathcal{G}$), local minimizer steps ($\mathcal{M}$) and saddle point steps ($\mathcal{N}$).
The (probabilistic) analysis of saddle point iterations was carried out in Section~\ref{s:saddle}.
We now turn to analyzing the large gradient steps: this is deterministic.
The local minimizer steps occupy us in the next subsection.

\begin{proposition}[Large gradient steps] \label{prop:bound_large_gradient}
  Under Assumptions~\ref{ass:f} and~\ref{ass:mumorse_and_stronggrad},
  let $\bar \sigma, \Flg$ be as in~\eqref{eq:def_constants}.
  Assume $\sigma \leq \bar \sigma$.
  If iteration~$k$ of Algorithm~\ref{algo:pTR} is a successful ``large gradient'' iteration, that is, $k \in \mathcal{G} \cap \mathcal{S}$, then
  \begin{equation} \label{eq:dec_lg_1}
    f(x_k) - f(x_{k+1}) \geq \Flg.
  \end{equation}
\end{proposition}
\begin{proof}
  Recall that $\|r^{(0)}_k\| = \|g_k + H_k \xi_k\| \geq \|g_k\|$ by design of $\xi_k$.
  Invoke Lemma~\ref{lemma:cauchy_decrease} (keeping only the term $W_1$) to obtain:
  \[
    m_{x_k}(\xi_k) - m_{x_k}(u_k) \geq \min\!\left( \frac{1}{2L_G}\|g_k\|, \frac{1}{8}\Delta_k \right) \|g_k\|.
  \]
  Note that $\|g_k\| \geq \bar G$ since $k \in \mathcal{G}$ (Lemma~\ref{lemma:lowbound_G}).
  Moreover, $\frac{1}{8} \Delta_k \geq \bar{R}$ by Lemma~\ref{lemma:lowbound_deltak}.
  Therefore,
  \begin{align*}
    m_{x_k}(\xi_k) - m_{x_k}(u_k)
      & \geq \min\!\left( \frac{1}{2L_G}\bar G, \bar R \right) \|g_k\| \\
      & = \min\!\left( \frac{\mu}{4L_G} \bar R, \bar R \right) \|g_k\|
        = \frac{\mu}{4L_G}\bar R \|g_k\|
        = \frac{1}{2L_G} \bar G \|g_k\|,
  \end{align*}
  where we used the constants in~\eqref{eq:def_constants_DeltacritRbar} and~\eqref{eq:def_constants}, and the fact that $\mu \leq L_G$.
  To relate the model decrease to the progress in objective value, we use Lemma~\ref{lemma:obj_decrease} noting $\|\xi_k\| \leq \sigma$:
  \begin{align*}
    f(x_{k}) - f(x_{k+1}) &\geq \rho'\left[ m_{x_k}(\xi_k) - m_{x_k}(u_k) \right] -\|g_k\|\sigma -\frac{L_G \sigma^2}{2}
    \geq \left( \frac{\rho'}{2L_G} \bar{G} - \sigma \right)\|g_k\| -\frac{L_G \sigma^2}{2}.
  \end{align*}
  Finally, we use the fact that $\sigma \leq \bar \sigma \leq \rho' \bar{G} / (4L_G)$ as well as $\|g_k\| \geq \bar G$ again, to find
  \begin{align*}
    f(x_{k}) - f(x_{k+1}) & \geq \frac{\rho'}{4L_G} \bar{G}^2 - \frac{L_G}{2} \left( \frac{\rho' \bar{G}}{4L_G} \right)^2
    = \left( \frac{1}{4} - \frac{\rho'}{32} \right) \frac{\rho'}{L_G}\bar{G}^2
    \geq \frac{\rho'}{5L_G}\bar{G}^2
    = \Flg,
  \end{align*}
  as announced.
  (That last step is the only one that used a bound on $\sigma$.)
\end{proof}

\subsection{Capture and quadratic convergence near local minimizers (deterministic)}

We now show that, if an iterate of Algorithm~\ref{algo:pTR} reaches a neighborhood of a local minimizer, the ensuing iterates remain in that neighborhood (they are ``captured'').
Moreover, these subsequent steps are all accepted and the sequence converges at least quadratically to the local minimizer.
Let us emphasize that this is a \emph{deterministic} statement, in spite of the randomization.

% Note: the next proposition does use the strong gradient property, because you do need Lemma~\ref{lemma:lowbound_deltak} to make sure $\Delta_k$ won't be too small by the time you get to the neighborhood of $x^*$.

\begin{proposition}[Local minimizer steps] \label{prop:min_capture}
  Under Assumptions~\ref{ass:f} and~\ref{ass:mumorse_and_stronggrad},
  let $\bar R$ be as in~\eqref{eq:def_constants_DeltacritRbar}.
  Assume $\sigma \leq \bar R$.
  If iteration~$k$ of Algorithm~\ref{algo:pTR} is a ``local minimizer'' iteration, that is, $k \in \mathcal{M}$, let $x^*$ denote the local minimizer of $f$ which satisfies $\|x_{k} - x^*\| \leq \bar R$.
  Then,
  \begin{enumerate}[label=(\roman*)]
    \item \label{item:stopres} $\stopk{k} = \stopres$,
    \item \label{item:successlocalmin} the iteration is successful ($k\in \mathcal{S}$) with $\Delta_{k+1} = \Delta_k$, and
    \item \label{item:quad_conv} $\|x_{k+1} - x^*\| \leq \|x_{k} - x^*\|^2 / (2 \bar R) \leq \bar R / 2$.
  \end{enumerate}
  The last claim ensures all subsequent iterations satisfy the same assumptions as $k$.
  In particular, $x_k, x_{k+1}, x_{k+2}, \ldots$ converges at least quadratically to $x^*$.
\end{proposition}
\begin{proof}
  First note that by Lemma~\ref{lemma:growth_mu}\ref{item:growth_mu3} with $c = 1/2$, since $\bar R \leq \Deltacrit\leq \mu / (2L_H)$, the Hessian $H_k$ is positive definite and we have $H_k \succeq \frac{\mu}{2}I_d$.
  Also, $\Delta_k \geq 8\bar{R}$ by Lemma~\ref{lemma:lowbound_deltak} (this is the only use of the strong gradient part of Assumption~\ref{ass:mumorse_and_stronggrad}).

  To prove item~\ref{item:stopres}, we use Lemma~\ref{lemma:tcg_boundedness} which provides a bound on the inner iterates of $\tCGbg$ for positive definite quadratics.
  Let $T_k$ be the number of iterations of $\tCGbg$.
  Then, for all $0 \leq t \leq T_k-1$ the tentative iterates $v^{(t)}_{k,+}$ generated inside that algorithm satisfy
  \[
    \|v_{k,+}^{(t)}\| \leq 2 \|\xi_k\| + \|H_k^{-1} g_k\|.
  \]
  We show that this bound cannot exceed $\Delta_k/2$.
  To see this, first use Lipschitz continuity of the Hessian to observe that $0 = \nabla f(x^*) = \nabla f(x_k) + \nabla^2 f(x_k)[x^* - x_k] + w_k$ with $\|w_k\| \leq \frac{L_H}{2} \|x_k - x^*\|^2$~\eqref{eq:LHnablabound}.
  Then, we can check for every $0\leq t \leq T_k-1$ that
  \begin{align}
    \|v^{(t)}_{k,+}\|
    & \leq 2 \|\xi_k\| + \|x_k-x^* - H_k^{-1} w_k\| \nonumber\\
    & \leq 2 \|\xi_k\| + \|x_k-x^*\|+ \frac{2}{\mu} \|w_k\| \nonumber\\
    & \leq 2 \|\xi_k\| + \|x_k-x^*\|+ \frac{L_H}{\mu} \|x_k-x^*\| \cdot \|x_k-x^*\| \nonumber\\
    & \leq 2 \|\xi_k\| + \left( 1 + \frac{1}{10} \right) \|x_k - x^*\|
    \label{eq:bound_uk_min} \\
    %& \leq \frac{1}{4} \Delta_k + \frac{1}{8} \Delta_k + \frac{L_H \bar{R}}{\mu} \cdot \frac{1}{8} \Delta_k \\
    & < \frac{1}{2} \Delta_k, \nonumber
  \end{align}
  where we used $\|x_k - x^*\| \leq \bar{R}$ and $\bar{R} \leq \Deltacrit \leq \frac{\mu}{10 L_H}$,
  then also $\|\xi_k\| \leq \sigma \leq \bar{R} \leq \frac{1}{8} \Delta_k$.

  In addition, we always have $\la H_k p_k^{(t-1)}, p_k^{(t-1)} \ra > 0$ since $H_k$ is positive definite.
  Combined with the argument above, this shows that the ``If'' condition that would lead to the $\stopoob$ stopping criterion in Algorithm~\ref{algo:tCG} (line~\ref{line:iftoboundaryintcgbg}) never triggers: we necessarily have $\stopk{k} = \stopres$.
  For the same reason, $v_k^{(t)} = v_{k,+}^{(t-1)}$ for all $ 1\leq t \leq T_k$.

  To prove item~\ref{item:successlocalmin}, we control the improvement ratio $\rho_k$~\eqref{eq:rhok} as in~\eqref{eq:bound_rhok}:
  \begin{align}
    |\rho_k - 1| &\leq \frac{L_H \|u_k\|^3 }{6\left( m_{x_k}(\xi_k) - m_{x_k}(u_k)\right)}.
  \end{align}
  The numerator can be bounded using~\eqref{eq:bound_uk_min} since $u_k = v^{(T_k)}_k = v_{k,+}^{(T_k-1)}$, so that
  \[
    \|u_k\|^3 \leq \big(2\|\xi_k\| + 2\|x_k-x^*\|\big)^3 = 8 \big(\|\xi_k\| + \|x_k-x^*\|\big)^3.
  \]
  As for the denominator, we use the model decrease guarantee of Lemma~\ref{lemma:cauchy_decrease}:
  \[
    m_{x_k}(\xi_k) - m_{x_k}(u_k) \geq \frac{1}{2L_G}\|r^{(0)}_k\|^2 = \frac{1}{2L_G}\|H_k \xi_k + g_k\|^2 \geq \frac{1}{2L_G} \left( \|H_k \xi_k\|^2 + \| g_k\|^2 \right)
  \]
  since $\la H_k \xi_k, g_k \ra \geq 0$ by design.
  Using $H_k \succeq \frac{\mu}{2}I_d$ and $\|g_k\| \geq \frac{\mu}{2}\|x_k-x^*\|$ (Lemma~\ref{lemma:growth_mu}) leads to
  \[
    m_{x_k}(\xi_k) - m_{x_k}(u_k) \geq \frac{\mu^2}{8L_G}\left( \|\xi_k\|^2 + \|x_k - x^*\|^2 \right)
     \geq \frac{\mu^2}{16L_G}\left( \|\xi_k\| + \|x_k - x^*\| \right)^2.
  \]
  Therefore, the bound on $\rho_k$ yields
  \[
    \begin{split}
      1 - \rho_k \leq |\rho_k-1|
        & \leq \frac{8\cdot 16 L_G \cdot L_H}{6\mu^2}\cdot \frac{\left(\|\xi_k\| + \|x_k-x^*\|\right)^3 }{\left( \|\xi_k\| + \|x_k - x^*\| \right)^2} \\
        & = \frac{64L_HL_G}{3\mu^2}\cdot \left( \|\xi_k\| + \|x_k-x^*\| \right)\\
        & \leq \frac{64L_HL_G}{3\mu^2}\cdot 2\bar R \\
        & \leq 1-\rho',
    \end{split}
  \]
  using again $\|\xi_k\| \leq \sigma \leq \bar{R}$ and $\|x_k - x^*\| \leq \bar{R}$, as well as $2 \bar{R} \leq \frac{\Deltacrit}{16} = (1-\rho') \frac{\mu^2}{160 L_H L_G}$~\eqref{eq:def_constants_DeltacritRbar}.
  This proves that $\rho_k \geq \rho'$ and hence that iteration $k$ is successful.
  Since also $\stopk{k} = \stopres$, the mechanism of Algorithm~\ref{algo:pTR} ensures $\Delta_{k+1} = \Delta_k$.

  Finally, item~\ref{item:quad_conv} follows  from Lemma~\ref{lemma:quadratic_conv_crit}.
\end{proof}

\subsection{Global convergence proof}

We now have all the tools we need to prove convergence to a local minimizer with high probability.
The argument combines the deterministic control on large gradient steps (Proposition~\ref{prop:bound_large_gradient}), the deterministic control on steps near local minimizers (Proposition~\ref{prop:min_capture}), and the probabilistic control on steps near saddle points (Theorem~\ref{thm:saddle_escape}).

\begin{theorem}[Convergence to a local minimizer]\label{thm:outer_complexity}
  Let $f \colon \Rd \to \reals$ satisfy Assumptions~\ref{ass:f} and~\ref{ass:mumorse_and_stronggrad}.
  Run Algorithm~\ref{algo:pTR} with $x_0 \in \reals^d$, an initial trust-region radius $\Delta_0$ and a noise scale $\sigma$ satisfying $0 < 4\sigma \leq \Delta_0 \leq \bar \Delta$.

  Set the constants $(\Deltacrit, \bar R, \bar G, \Flg, \bar \sigma)$ as defined in~\eqref{eq:def_constants_DeltacritRbar} and~\eqref{eq:def_constants}.
  Let $\delta \in (0,1)$ be a confidence parameter.
  Define $\delta'$ and the saddle escape time $K_{\rm esc}(\delta',\sigma)$ as
  \begin{align}
    \delta' = \left(\frac{\Flg}{f(x_0) - \flow+\Flg} \right)\delta \;\;\in (0,1) && \textrm{ and } &&
    K_{\rm esc}(\delta', \sigma) = 2
          + \log_2 \log_2\!\left(2+\frac{\sqrt{d}}{ \omega_2 \mu \delta'\sigma}\right).
  \end{align}
  Assume that the noise scale $\sigma$ is small enough so that
  \begin{align}
    \sigma \leq \bar \sigma && \textrm{ and } && \sigma\cdot K_{\rm esc}(\delta',\sigma) \leq \frac{\rho' \mu^2 {\bar R}}{8  L_G^2}.
    \label{eq:cond_sigma_global}
  \end{align}
  Then, with probability at least $1 - \delta$, there is an iteration $K_{\rm cap}$ (for ``capture'') which is the first such that $\|x_{K_{\rm cap}} - x^*\| \leq \bar{R}$ for some local minimizer $x^*$ of $f$ and also
  \begin{align}
    K_{\rm cap} \leq \bar K_{G,N} := \frac{3}{2} \left( \frac{f(x_0)-\flow}{\Flg} + 1 \right) K_{\rm esc}(\delta',\sigma) + \frac{1}{2} \log_2\!\left( \frac{\Delta_0}{8\bar R} \right).
    \label{eq:KGN}
  \end{align}
  Conditioned on that event, for all $\epsilon > 0$, let $K$ be the first iteration such that $\|x_K - x^*\| \leq \epsilon$.
  Then,
  \begin{align}
    K \leq K_{\rm cap} + \bar K_{M,\epsilon} && \textrm{ with } && \bar K_{M,\epsilon} := \log_2\!\left( 1 + \left[\log_2 \frac{\bar R}{\epsilon} \right]_+ \right),
    \label{eq:KM}
  \end{align}
  that is, starting from $K_{\rm cap}$, the iterates converge at least quadratically to $x^*$.

  All iterations before $K_{\rm cap}$ are of the ``large gradient'' and ``saddle point'' type.
  All iterations starting from $K_{\rm cap}$ are of the ``local minimizer'' type.
\end{theorem}

\begin{remark}\label{remark:sigma_global} On a technical note regarding conditions~\eqref{eq:cond_sigma_global} above, we can prove similarly to Remark~\ref{remark:sigma_bar1}, using Lemma~\ref{lemma:sigma_loglog}, that there exists $\bar \sigma_{\rm global} > 0$ such that the conditions are satisfied for all $0 < \sigma \leq \bar \sigma_{\rm global}$. Moreover, $\sigma_{\rm global}$ only depends logarithmically on $d$ and $\delta$.
\end{remark}

\begin{proof}
  Recall the partition $(\mathcal{G}, \mathcal{N}, \mathcal{M})$ of the iteration indices $\mathbb{N} = \{ 0, 1, 2, \ldots \}$ as defined in the beginning of Section~\ref{s:global_comp}.
  Proposition~\ref{prop:min_capture} guarantees that local minimizers ``capture'' nearby iterates, in the sense that if $k$ is in $\mathcal{M}$ then (in particular) $k+1, k+2, \ldots$ are all in $\mathcal{M}$.
  Accordingly, let $K_{\rm cap} = \inf \mathcal{M}$ be the (random) time of capture.
  There are only two possibilities:
  either $K_{\rm cap}$ is finite and $\mathcal{M} = [K_{\rm cap}, +\infty)$, or $\mathcal{M} = \emptyset$.

  For a subset $\mathcal{C}$ of $\naturals$ and $k \in \mathbb{N}$, we use notation $|\mathcal{C}|_k = |\mathcal{C}\cap \{0,1,\dots k-1\}|$.

  \paragraph{Local phase.}
  Whether or not $K_{\rm cap}$ is finite is a random event: we will show next that it is the case with high probability.
  Assuming it is finite, there is a local minimizer $x^*$ such that $\|x_{K_{\rm cap}} - x^*\| \leq \bar{R}$ and we can apply Proposition~\ref{prop:min_capture}\ref{item:quad_conv} (which is deterministic) to claim that
  \begin{align*}
    \|x_{K_{\rm cap} + j} - x^*\| \leq \frac{\bar R}{2^{2^j-1}} && \textrm{ for all } j \geq 0.
  \end{align*}
  Therefore, the number of ``local minimizer'' iterations needed after $K_{\rm cap}$ to find a point $x_K$ satisfying $\|x_K-x^*\| \leq \epsilon$ is at most $\log_2\!\left( 1 + \left[\log_2 \frac{\bar R}{\epsilon} \right]_+ \right)$, that is, $\bar K_{M,\epsilon}$~\eqref{eq:KM}.

  \paragraph{Global phase.}
  It remains to show that $K_{\rm cap}$ is finite with high probability, and then to control how large it can be.
  The idea is that iterations in $\mathcal{G}$ (large gradient) and $\mathcal{N}$ (saddle point) decrease the objective at least by a constant value, and this cannot happen an infinite number of times as $f$ is bounded below.
  However, there is a subtlety: while the saddle escape theorem (Theorem~\ref{thm:saddle_escape}) guarantees that the objective value decreases \emph{overall} after at most some constant number of iterations (with high probability), this decrease may not be monotone.

  To handle this issue, let us first construct a sequence of indices $k_0, k_1, k_2, \ldots$ in $\mathbb{N}\cup\{+\infty\}$ as follows.
  Set $k_0 = 0$.
  Then, assuming $k_0, \ldots, k_i$ have been chosen, set $k_{i+1}$ according to the rules below, inductively (recall $\mathcal{S}$ indexes successful iterations):
  \begin{itemize}
    \item If $k_i = +\infty$, set $k_{i+1} = +\infty$.
    \item If $k_i \in \mathcal{G}$ (large gradient), set $k_{i+1} = 1 + \inf \{ k \geq k_i \,:\, k \in \mathcal{S}\}$.
    \item If $k_i \in \mathcal{N}$ (saddle point), set
    \[k_{i+1} = 1+\inf \{ k \geq k_i \,:\, k \in \mathcal{S}\textrm{ and } \stopk{k} = \stopoob \}.\]
    \item If $k_i \in \mathcal{M}$ (local minimizer), set $k_{i+1} = k_i+1$.
  \end{itemize}
  For each index $i \geq 0$, define a random event called \eqref{eq:rec_ki}, as follows:
  \begin{align}
    k_i < \infty, &&
    |\mathcal{S}|_{k_i} \leq i K_{\rm esc}(\delta',\sigma), &&
    \textrm{ and } &&
    \textrm{ if }\,  k_{i} \notin \mathcal{M}, \textrm{ then } f(x_0) - f(x_{k_i}) \geq i\Flg.
    \tag{$P_i$}
    \label{eq:rec_ki}
  \end{align}
  (Recall that $|\mathcal{S}|_{k_i}$ denotes the number of successful iterations before $k_i$.)

  Let us prove the following assertion by induction:
  \begin{align}
    \textrm{For every $i \geq 0$ such that $i \delta' < 1$, event \eqref{eq:rec_ki} holds with probability at least $1-i\delta'$.}
    \tag{Q}
    \label{Piproba}
  \end{align}
  This assertion is true for $i = 0$ because ($P_0$) always holds.
  Now, assume the assertion is true for some $i \geq 0$ such that $(i+1)\delta' < 1$.
  We must show that ($P_{i+1}$) holds with probability at least $1-(i+1)\delta'$.
  If we can show that ($P_{i+1}$) holds with probability at least $1-\delta'$ conditioned on ($P_i$), then we are done, as we would have
  $$
    \Prob(P_{i+1}) \geq \Prob(P_{i+1}\cap P_i) = \Prob(P_{i+1} | P_i) \Prob(P_i) \geq (1-\delta')(1-i\delta') \geq 1-(i+1)\delta'.
  $$
  Accordingly, let us assume ($P_i$) holds.
  Then, what is the probability that ($P_{i+1}$) holds?

  There are three possibilities to consider, depending on $k_i$.
  \begin{itemize}
    \item \textbf{``Large gradient'' iterations: $k_i \in \mathcal{G}$.} By Lemma~\ref{lemma:lowbound_deltak} which guarantees $\Delta_k \geq 8\bar R$, the number of consecutive unsuccessful steps is bounded and therefore $k_{i+1}$ is finite.

    By definition of $k_{i+1}$, we have $k_{i+1}-1\in \mathcal{S}$ while $k_i, \dots, k_{i+1}-2$ are unsuccessful. Therefore, we have $x_{k_i}=x_{k_i+1}= \cdots =x_{k_{i+1}-1}$ and thus $k_{i+1}-1 \in \mathcal{G}$. This means that $k_{i+1}-1 \in \mathcal{G} \cap \mathcal{S}$, and we can apply the deterministic decrease result of Proposition~\ref{prop:bound_large_gradient}:
      \[f(x_{k_{i+1}-1})-f(x_{k_{i+1}})\geq \Flg.\]
    We deduce using~\eqref{eq:rec_ki} that
    \begin{align*}
      f(x_0)-f(x_{k_{i+1}}) &=f(x_0)-f(x_{k_{i}}) +  f(x_{k_{i+1}-1})-f(x_{k_{i+1}}) \geq (i+1)\Flg.
    \end{align*}
    Noting that there is only one successful iteration between $k_i$ and $k_{i+1}-1$, we also have
    \[
      |\mathcal{S}|_{k_{i+1}} = |\mathcal{S}|_{k_{i}} + 1 \leq i K_{\rm esc}(\delta',\sigma) + 1 \leq (i+1) K_{\rm esc}(\delta',\sigma).
    \]
    Thus, if ($P_i$) holds and $k_i \in \mathcal{G}$, then ($P_{i+1}$) holds.

    \item \textbf{``Saddle point'' iterations: $k_i  \in \mathcal{N}$.}
    We invoke the saddle escape result of Theorem~\ref{thm:saddle_escape} with $R = \bar R$, $\Deltainf = 8\bar{R}$ and confidence parameter $\delta'$.
    The assumptions are satisfied: indeed, Lemma~\ref{lemma:lowbound_deltak} guarantees $\Delta_k \geq \Deltainf$ for all $k$,
    and conditions~\eqref{eq:cond_sigma1} and~\eqref{eq:cond_sigma2} follow from our conditions~\eqref{eq:cond_sigma_global} on $\sigma$ upon recalling that $\bar \sigma \leq \bar R$ (see~\eqref{eq:def_constants_DeltacritRbar}, \eqref{eq:def_constants}).

    Then, Theorem~\ref{thm:saddle_escape} implies that $k_{i+1}$ is finite with probability at least $1-\delta'$.
    Moreover, in that event,
    \[
      f(x_{k_{i}}) - f(x_{k_{i+1}}) \geq \frac{\rho' \mu^2\Delta_{k_i}^2}{256 L_G} \geq \frac{\rho' \mu^2{\bar R}^2}{4 L_G} = 5\Flg \geq \Flg,
    \]
    where we used $\Delta_{k_i} \geq 8 \bar R$ again, and the definition of $\Flg$~\eqref{eq:def_constants}.
    Using~\eqref{eq:rec_ki} then yields
    \[f(x_{0}) - f(x_{k_{i+1}}) \geq (i+1) \Flg.\]
    Still in that same event, Theorem~\ref{thm:saddle_escape} asserts that the number of successful iterations between $k_i$ and $k_{i+1}$ is at most $K_{\rm esc}(\delta',\sigma)$, and therefore
    \[
      |\mathcal{S}|_{k_{i+1}} \leq |\mathcal{S}|_{k_{i}} + K_{\rm esc}(\delta',\sigma) \leq (i+1) K_{\rm esc}(\delta',\sigma).
    \]
    Thus, if ($P_i$) holds and $k_i \in \mathcal{N}$, then ($P_{i+1}$) holds with probability at least $1 - \delta'$.

    \item \textbf{``Local minimizer'' iterations: $k_i \in \mathcal{M}$.}
    By the capture result of Proposition~\ref{prop:min_capture}, $k_{i+1} = k_i+1 \in \mathcal{M}$ and $|\mathcal{S}|_{k_i+1} = |\mathcal{S}|_{k_i}+1$.
    Thus, if ($P_i$) holds and $k_i \in \mathcal{M}$, then ($P_{i+1}$) holds.
  \end{itemize}

  In all three cases, we showed that the probability that $(P_{i+1})$ holds given $(P_i)$ is at least $1-\delta'$.
  Thus, the proof by induction is complete for assertion~\eqref{Piproba}.

  From here, we use~\eqref{Piproba} to claim that~\eqref{eq:rec_ki} with a carefully chosen index $i$ holds with good probability.
  Specifically, we choose $i = \bar I := \big\lfloor \tfrac{f(x_0)-\flow}{\Flg} \big\rfloor+1$.
  By definition of $\delta'$ (see theorem statement), we have
  \begin{equation}
    \bar I \delta'
      = \frac{ \big\lfloor \tfrac{f(x_0)-\flow}{\Flg} \big\rfloor + 1}{ \left(\frac{f(x_0)-\flow}{\Flg}\right) + 1 } \cdot \delta
      \leq \delta
      < 1.
  \end{equation}
  Therefore, $(P_{\bar I})$ holds with probability at least $1-\bar I \delta' \geq 1 - \delta$.

  Going forward, assume $(P_{\bar I})$ holds.
  Let us show that $K_{\rm cap}$ is bounded.
  We know $k_{\bar I}$ is finite.
  Also, for contradiction, if $k_{\bar I}$ is not in $\mathcal{M}$, then $(P_{\bar I})$ and the lower bound $\flow$ on $f$ from Assumption~\ref{ass:f} further provide
  \begin{align}
    f(x_0) - \flow \geq f(x_0) - f(x_{k_{\bar I}}) \geq  \bar I \Flg = \left(\big\lfloor \tfrac{f(x_0)-\flow}{\Flg} \big\rfloor + 1\right)\Flg  > f(x_0) - \flow.
  \end{align}
  This is not possible; thus, $k_{\bar I}$ is in $\mathcal{M}$ and it follows that $K_{\rm cap} \leq k_{\bar I}$.

  Property $(P_{\bar I})$ additionally offers a bound on the number of successful iterations up to $k_{\bar I}$, namely, $|\mathcal{S}|_{k_{\bar I}} \leq \bar I K_{\rm esc}(\delta', \sigma)$.
  It also follows from the mechanics of the trust-region radius update (Lemma~\ref{lem:mostlysuccessful}) that
  \begin{align*}
    K_{\rm cap} \leq k_{\bar I} \leq \frac{3}{2} |\mathcal{S}|_{k_{\bar I}} + \frac{1}{2} \log_2\!\left( \frac{\Delta_0}{\Delta_{k_{\bar I}}} \right).
  \end{align*}
  Since the trust-region radius is lower bounded as $\Delta_k \geq 8\bar R$ for all $k$ (Lemma~\ref{lemma:lowbound_deltak}), we find
  \begin{align*}
    K_{\rm cap}
      & \leq \frac{3}{2} \bar I K_{\rm esc}(\delta', \sigma) + \frac{1}{2} \log_2\!\left( \frac{\Delta_0}{8\bar{R}} \right) \\
      & \leq \frac{3}{2} \left( \frac{f(x_0)-\flow}{\Flg} +1\right) K_{\rm esc}(\delta',\sigma) + \frac{1}{2} \log_2\!\left( \frac{\Delta_0}{8\bar R} \right),
  \end{align*}
  where we used the expression of $\bar I$.
  The final right-hand side is $\bar K_{G,N}$, as announced in~\eqref{eq:KGN}.
\end{proof}

\subsection{Aside: algorithm's behavior under weaker assumptions}

In this section, we prove that, in the absence of the $\mu$-Morse and strong gradient property (Assumption \ref{ass:mumorse_and_stronggrad}), the method still enjoys the standard first-order guarantees of TR-tCG.

The strong gradient property (part of Assumption~\ref{ass:mumorse_and_stronggrad}) provides control on $f$ away from its critical points.
We use that assumption (only) to prove Lemma~\ref{lemma:lowbound_G} (lower bound on the gradient norm) and Lemma~\ref{lemma:lowbound_deltak} (lower bound on the trust-region radius).
In turn, those lemmas are used several times throughout our analysis.

Even without this assumption, Algorithm~\ref{algo:pTR} is still capable of finding approximate critical points of $f$, provided that the noise scale is small enough. The result below states that the method finds a point with gradient norm $\epsilon$ in at most $\tildeO(1/\epsilon^2)$ iterations, deterministically.
The argument tracks the standard trust-region proof while accounting for the perturbations $\xi_k$.
\begin{proposition}\label{prop:epsilon_complexity}
  Let $f$ satisfy Assumption~\ref{ass:f}.
  Let $0<\epsilon<32L_G\Delta_0$.
  Run Algorithm~\ref{algo:pTR} with noise scale $\sigma>0$ satisfying
  \[
    \sigma \leq \frac{\rho'(1-\rho')}{513 L_G} \epsilon. % Note: yes, 513, not 512. This gives just enough leeway to absorb the $\sigma^2$ term.
  \]
  Let $K$ be the first iteration index such that $\|\nabla f(x_K)\| \leq \epsilon$. Then,
  \[
      K \leq \frac{768 L_G(f(x_0) - \flow)}{\rho'(1-\rho')\epsilon^2} + \frac{1}{2} \log_2\!\left( \frac{32 L_G \Delta_0}{(1-\rho')\epsilon} \right).
  \]
\end{proposition}

\begin{proof}
  Note that none of the arguments below rely on Assumption~\ref{ass:mumorse_and_stronggrad}.
  Assume that for all $k < K$, we have $\|g_k\| > \epsilon$.
  We first notice that for all $k \leq K$, we have
  \begin{equation}\label{eq:radius_lowbound_weak}
    \Delta_k \geq \frac{(1-\rho')\epsilon}{32 L_G}.
  \end{equation}
  Indeed, the inequality is true for $k =0 $ by assumption.
  For $k\geq 1$, it follows from the fact that every iteration such that $\Delta_k$ is smaller than $(1-\rho')\|g_k\| / (8 L_G)$ is successful (Lemma \ref{lemma:large_gradient_radius}).

  Consider a successful iteration $k \in \mathcal{S}$.
  From Lemma~\ref{lemma:cauchy_decrease} (model decrease), using $\|r^{(0)}_k\| \geq \|g_k\| > \epsilon$ and Eq.~\eqref{eq:radius_lowbound_weak}:
  $$
  m_{x_k}(\xi_k) - m_{x_k}(u_k) \geq \min\!\left( \frac{\|g_k\|^2}{2L_G}, \frac{\Delta_k \|g_k\|}{8} \right) > \min\!\left( \frac{\epsilon^2}{2L_G}, \frac{(1-\rho')\epsilon^2}{256 L_G} \right) = \frac{(1-\rho')}{256 L_G} \epsilon^2.
  $$
  Now, applying Lemma~\ref{lemma:obj_decrease} (objective decrease) with $\|\xi_k\| \leq \sigma$:
  \begin{equation}\label{eq:suff_dec}
      f(x_k) - f(x_{k+1}) \geq \rho' \left( m_{x_k}(\xi_k) - m_{x_k}(u_k) \right) - \epsilon \sigma - \frac{L_G}{2} \sigma^2 > \frac{\rho'(1-\rho')}{512 L_G} \epsilon^2.
  \end{equation}
  The usual telescoping sum argument then reveals
  \[
      f(x_0) - \flow \geq |\mathcal{S}|_K \frac{\rho'(1-\rho')}{512 L_G} \epsilon^2.
  \]
  This implies an upper bound on the number of successful steps.
  Finally, we relate the total iterations $K$ to successful iterations using Lemma~\ref{lem:mostlysuccessful}, namely,
  \[
      K \leq \frac{3}{2} |\mathcal{S}|_K + \frac{1}{2} \log_2\!\left( \frac{\Delta_0}{\Delta_K} \right),
  \]
  and use the lower bound on $\Delta_K$ from Eq.~\eqref{eq:radius_lowbound_weak} to conclude.
\end{proof}

\section{Complexity estimate in terms of oracle calls}\label{s:complexity}

The previous section bounds the number of outer iterations Algorithm~\ref{algo:pTR} may use to reach a point within distance $\epsilon$ of a local minimizer.
With proper care to avoid redundant computations, assuming $f(x_0)$ and $\nabla f(x_0)$ have been computed, each outer iteration requires:
\begin{itemize}
  \item at most one evaluation of $f$ (at $x_k + u_k$),
  \item at most one evaluation of $\nabla f$ (at $x_k$), and
  \item at most $T_k + 2$ Hessian-vector products (using $\nabla^2 f(x_k)$), where $T_k$ is the number of inner iterations of the call to $\tCGbg$ at outer iteration $k$. % Note: yes, it is T_k + 2. Think of the case where $T_k = 1$ and we did a boundary step. You had to compute Hxi (1), then also Hp0 (2), and also (in the boundary gradient step) you need the gradient of the model at v1 (that's free with proper reuse), called -r, but then also Hr: that's (3). And the method returns T_k = 1.
\end{itemize}
In this section, we estimate the total number of inner iterations, summed across all outer iterations.
To do so, we bound the number of iterations performed by $\tCGbg$ for each run, by distinguishing three scenarios: ``near a local minimizer'', ``near a saddle'' and ``large gradient''.
These require some understanding of the behavior of $\tCGbg$ in each setting, handled separately in the next three sections.

\subsection{Number of inner iterations at points near local minimizers (deterministic)}

First, we examine the behavior of $\tCGbg$ (Algorithm~\ref{algo:tCG}) when it is called for a point $x_k$ near a local minimizer.
In this case, the quadratic model is strongly convex and (classical) CG enjoys fast linear convergence:
we capitalize on this to show that the ``small residual'' criterion ($\stopres$) triggers rapidly (unless another stopping criterion triggers even sooner).

\begin{lemma}\label{lemma:cg_rate_convex} % [Complexity of $\tCGbg$ for strongly convex quadratics]
  Let $H \in \symm_d$ satisfy $\opnorm{H} \leq L_G$.
  Fix vectors $g, \xi \in \reals^d$ and a radius $\Delta > 0$ such that $\|\xi\| \leq \Delta / 4$ and $g \neq 0$.
  Assume $H \succeq \nu I_d$ with $\nu > 0$.
  Run $\tCGbg(H, g, \Delta, \xi)$, with output $(u, \stopcrit, T)$.
  Then, $T$ (the number of inner iterations performed) satisfies
  \begin{align*}
    T \leq 1 + \sqrt{\frac{L_G}{\nu}} \log\!\left( 2 \sqrt{\frac{L_G}{\nu}} \cdot \frac{\|g + H\xi\|}{\min(\omega_1\|g\|, \omega_2\|g\|^2)} \right).
  \end{align*}
  (We also have $T \leq d$, which remains valid if $g = 0$ by Lemma~\ref{lemma:well_known_props}.)
\end{lemma}
\begin{proof}
  Let $m(v) = \la g,v \ra + \frac{1}{2} \la Hv,v \ra$.
  The minimizer of $m$ is $v^* = -H^{-1}g$.
  Consider $t \leq T-1$ for the entire proof.
  For such $t$, regardless of what brings $\tCGbg$ to terminate, we know $v^{(t)}$ is equal to the $t$th iterate of the standard CG method (see~\eqref{eq:def_tin}).
  Thus, $v^{(t)}$ satisfies the convergence bound
  \[
    m(v^{(t)}) - m(v^*) = \frac{1}{2}\langle H(v^{(t)} - v^*),v^{(t)} - v^* \rangle \leq 4
    \exp\!\left(
      -2t \sqrt{\frac{\nu}{L_G}}
    \right)
    \left( m(v^{(0)}) - m(v^*) \right),
  \]
  see, e.g., \cite[Thm~3.1.1]{Greenbaum1997iterative}.
  As $r^{(t)} = -(Hv^{(t)}+g) = H(v^*-v^{(t)})$ (Lemma~\ref{lemma:well_known_props}), we also have
  \begin{equation*}
    m(v^{(t)}) -m(v^*)= \frac{1}{2}\la H^{-1}r^{(t)},r^{(t)} \ra.
  \end{equation*}
  Since $\nu I_d \preceq H \preceq L_GI_d$, we deduce
  \begin{align*}
      \frac{1}{2L_G}\|r^{(t)}\|^2 &\leq   m(v^{(t)}) - m(v^*) \leq \frac{1}{2\nu} \|r^{(t)}\|^2.
  \end{align*}
  It follows that
  \begin{equation*}
    \|r^{(t)}\| \leq 2\exp\!\left( -t \sqrt{\frac{\nu}{L_G}} \right) \sqrt{\frac{L_G}{\nu}} \|r^{(0)}\|.
  \end{equation*}
  To conclude, write $r^{(0)} = -(g + Hv^{(0)}) = -(g + H\xi)$, and use the fact that $\tCGbg$ terminates as soon as $\|r^{(t)}\| \leq \min(\omega_1 \|g\|,\omega_2\|g\|^2)$ (if not earlier).
\end{proof}

\subsection{Number of inner iterations at points near saddle points}

Next, we turn to the behavior of $\tCGbg$ when it is called at a point $x_k$ near a saddle point, so that the Hessian has a sufficiently negative eigenvalue.
We first recall a fact about the eigenvalues of matrices restricted to Krylov spaces.
It can be traced back to \citet{LanczosBound1992}, and we quote the form presented in~\citep[Lem.~3]{Carmon2018a}.

\begin{lemma}\label{lemma:finding_eig}
  Let $M\in \symm_d$.
  Assume $M \succeq 0$ and $\la q, Mq \ra = 0$ for some unit vector $q \in \reals^d$.
  Let $r\in \reals^d$, $r \neq 0$.
  Then, for all $t \geq 1$ there exists $z^{(t)} \in \mathcal{K}_t(M, r) = \mathrm{span}(r, Mr, \ldots, M^{t-1}r)$ such that
  \begin{align*}
    \|z^{(t)}\| = 1  && \textrm{ and } &&
    \la z^{(t)}, Mz^{(t)} \ra \leq \frac{\opnorm{M}}{16(t - \frac{1}{2})^2} \left(\log\!\left(4\frac{\|r\|^2}{\la q, r\ra^2} - 2\right) \right)^2.
  \end{align*}
\end{lemma}

We use this lemma to show that the restriction of $H$ to a sufficiently large Krylov space~\eqref{eq:KrylovspaceKtHb} has a negative eigenvalue, causing $\tCGbg$ to terminate with $\stopcrit = \stopoob$ (if not sooner).

\begin{lemma}\label{lemma:tcg_nc_innerits} % [Complexity of $\tCGbg$ for non-convex quadratics]
  Let $H \in \symm_d$ satisfy $\opnorm{H} \leq L_G$.
  Fix vectors $g, \xi \in \reals^d$ and a radius $\Delta > 0$ such that $\|\xi\| \leq \Delta / 4$.
  Assume the smallest eigenvalue of $H$, denoted by $\lambdamin(H)$, is negative, and let $\qmin$ be a corresponding unit eigenvector.
  Run $\tCGbg(H, g, \Delta, \xi)$, with output $(u, \stopcrit, T)$.
  Then, $T$ (the number of inner iterations performed) satisfies either $T = 0$ or
  \[
    T \leq \frac{3}{2} + \sqrt{
      \frac{L_G}{8|\lambdamin(H)|}
    }
      \log\!\left(4\frac{\|g+H\xi\|^2}{\la \qmin, g + H\xi\ra^2} - 2\right).
  \]
\end{lemma}
\begin{proof}
  If $r^{(0)} = -(g+H\xi) = 0$, then $T = 0$.
  Assume not.
  By Lemma~\ref{lemma:well_known_props}\ref{item:psd_kt}, for $1 \leq t \leq T-1$ the restriction of $H$ to the $t$th Krylov space~\eqref{eq:KrylovspaceKtHb} is positive definite:
  \begin{equation}
    \la z, H z \ra > 0, \qquad \forall z \in \mathcal{K}_t(H,r^{(0)}), z \neq 0.
    \label{eq:hpsdkt}
  \end{equation}
  We show that this condition cannot hold for $t$ larger than a certain value, thereby proving that the algorithm must terminate earlier.
  Define $M = H - \lambdamin(H) I_d$. Then $M \succeq 0$ and $\la \qmin,M \qmin \ra = 0$.
  By Lemma~\ref{lemma:finding_eig} applied to $M$, for all $t \geq 1$ there exists $z^{(t)} \in \mathcal{K}_t(M,r^{(0)})$ such that $\|z^{(t)}\| = 1$ and
  \[
    \la z^{(t)},Hz^{(t)} \ra - \lambdamin(H) \leq \frac{\opnorm{H - \lambdamin(H)I_d}}{16(t - \frac{1}{2})^2} \left(\log\!\left(4\frac{\|r^{(0)}\|^2}{\la \qmin, r^{(0)}\ra^2} - 2\right) \right)^2.
  \]
  We use the bound $\opnorm{H-\lambdamin(H)I_d} \leq L_G+|\lambda_{\min}(H)| \leq 2L_G $.
  Note that $\mathcal{K}_t(M,r^{(0)}) = \mathcal{K}_t(H,r^{(0)})$, therefore by \eqref{eq:hpsdkt} we have $\la z^{(t)}, Hz^{(t)} \ra > 0$. It follows that
  \[
  \begin{split}
    |\lambdamin(H)| & \leq \frac{L_G}{8(t-\frac{1}{2})^2}
      \left(\log\!\left(4\frac{\|r^{(0)}\|^2}{\la \qmin, r^{(0)}\ra^2} - 2\right) \right)^2.
  \end{split}
  \]
  Substitute $t = T-1$ to derive the announced bound on $T$.
\end{proof}

\subsection{Number of inner iterations at large gradient points (deterministic)}

Finally, we deal with the ``large gradient'' scenario.
In this case, the Hessian can be indefinite or ill-conditioned, making standard complexity results for CG inapplicable.
Nevertheless, we can provide an estimate when the initial residual is large enough.
The idea is that, as long as the $\stopres$ criterion does not trigger, the residual $\|r^{(t)}\|$ remains bounded away from zero, thus ensuring that the model decreases at least by a constant amount at each iteration.

\begin{lemma}\label{lemma:cg_complexity_delta} % [Complexity of $\tCGbg$ for general models]
  Let $H \in \symm_d$ satisfy $\opnorm{H} \leq L_G$.
  Fix vectors $g, \xi \in \reals^d$ and a radius $\Delta > 0$ such that $\|\xi\| \leq \Delta / 4$ and $g \neq 0$.
  Run $\tCGbg(H, g, \Delta, \xi)$, with output $(u, \stopcrit, T)$.
  Then, $T$ (the number of inner iterations performed) satisfies
  \[
    T \leq 1 + \frac{2L_G\Delta \cdot \|g + H\xi\|}{\big(\min(\omega_1 \|g\|, \omega_2 \|g\|^2)\big)^2}.
  \]
\end{lemma}
\begin{proof}
  If $r^{(0)} = -(g+H\xi) = 0$, then $T = 0$ and the lemma holds. Assume not.
  Let $m(v) = \la g,v \ra + \frac{1}{2} \la Hv,v \ra$.
  By Lemma~\ref{lemma:well_known_props}\ref{item:grad_decrease}, for $\Tin$ defined as in~\eqref{eq:def_tin} we have
  \[
    m(v^{(0)}) - m(v^{(\Tin)}) \geq \frac{1}{2L_G} \sum_{t=0}^{\Tin-1} \|r^{(t)}\|^2.
  \]
  As long as $\tCGbg$ did not terminate, the inequality that would trigger the $\stopres$ criterion does not hold and therefore $\|r^{(t)}\| > \min\!\left( \omega_1 \|g\|,\omega_2 \|g\|^2 \right)$ for $0 \leq t \leq \Tin - 1$.
  Thus,
  \begin{equation} \label{eq:bound_vt_omegag}
    m(v^{(0)}) - m(v^{(\Tin)}) \geq \frac{\Tin}{2L_G} \left(\min\!\left( \omega_1 \|g\|,\omega_2 \|g\|^2 \right)\right)^2.
  \end{equation}
  Moreover, by Lemma~\ref{lemma:well_known_props}\ref{item:psd_kt}, the restriction of $H$ to the Krylov space $\mathcal{K}_{\Tin}(H,r^{(0)})$ is positive definite.
  Since $v^{(0)}-v^{(\Tin)}$ belongs to this space (Lemma~\ref{lemma:well_known_props}\ref{item:kt}), the restriction of the quadratic $m$ to the segment $[v^{(0)},v^{(\Tin)}]$ is convex, in the sense that $t \mapsto m((1-t) v^{(0)} + t v^{(\Tin)})$ is convex.
  Therefore,
  \[
  \begin{split}
    m(v^{(0)}) - m(v^{(\Tin)}) &\leq -\la \nabla m(v^{(0)}), v^{(\Tin)} - v^{(0)} \ra \\
     & \leq \|r^{(0)}\| \cdot \|v^{(\Tin)} - v^{(0)}\| \\
     & \leq \|g + H\xi\| \left( \|v^{(\Tin)}\| + \|v^{(0)}\| \right)
      \leq \|g + H\xi\| \cdot \Delta, %\\
     %& \leq \left(\|g\| + L_G \|\xi\| \right) \Delta,\\
     %& \leq \left(\|g\| + \frac{1}{4} L_G \Delta \right) \Delta,
  \end{split}
  \]
  where we used the fact that $\|v^{(t)}\| \leq \Delta / 2$ for $t \leq \Tin$ by definition.
  Combining with \eqref{eq:bound_vt_omegag} and recalling that the total number of iterations $T$ is at most $\Tin + 1$ allows to conclude.
\end{proof}

\subsection{Combined oracle complexity estimate}

We are now ready to establish the complexity estimate for the randomized trust-region method, in terms of the total number of calls to $f$ and $\nabla f$ (corresponding to the number of outer iterations, already bounded in Theorem~\ref{thm:outer_complexity}) and $\nabla^2 f$ (Hessian-vector products, corresponding to the total number of inner iterations, bounded below).

\begin{theorem}[Inner iterations complexity] \label{thm:global_inner_its}
  Continuing from Theorem~\ref{thm:outer_complexity},
  condition on the stated event of probability at least $1-\delta$.
  Thus, there exists an iteration $K_{\rm cap} \leq \bar K_{G,N}$~\eqref{eq:KGN} which is the first such that $\|x_{K_{\rm cap}} - x^*\| \leq \bar{R}$ for some local minimizer $x^*$ of $f$.

  Let $T_{\rm lg}$ denote the total number of inner iterations of $\tCGbg$ performed in the ``large gradient'' region.
  It satisfies
  \begin{equation}
    T_{\rm lg} \leq
     \bar K_{G,N}
     \left(
      1 + \frac{ ( L_G \bar \Delta)^2}{\left(\min(\omega_1 \bar G, \omega_2 \bar G^2) \right)^2}
     \right).
  \end{equation}
  Likewise, let $T_{\rm saddle}$ denote the total number of inner iterations of $\tCGbg$ performed in the ``near saddle points'' region.
  With probability at least $1-\delta$ (for a combined probability of at least $1-2\delta$), it satisfies
  \begin{equation}
    T_{\rm saddle} \leq
      \bar K_{G,N}
      \left(2+\sqrt{
      \frac{L_G}{4\mu}
      }
      \log\!\left( \frac{512 L_G^2 {\bar{R}}^2 \bar K_{G,N}^2 d  }{\pi \sigma^2 \delta^2 \mu^2}\right) \right).
  \end{equation}
  For all $\epsilon > 0$, let $K$ be the first iteration such that $\|x_K - x^*\| \leq \epsilon$.
  Let $T_{\rm min}$ denote the total number of inner iterations of $\tCGbg$ performed in the ``near local minimizers'' region before reaching $x_K$.
  Then,
  \begin{equation}
    T_{\rm min,\epsilon} \leq
      \bar K_{M,\epsilon}
      \left(
        1
        +  \sqrt{ \frac{2L_G}{\mu}}
        \log\!\left(
        \sqrt{\frac{2L_G}{\mu}} \cdot
        \frac{
          16L_G\bar R
        }{
          \min(\omega_1 \mu \epsilon , \omega_2 \mu^2 \epsilon^2)
        }
        \right)
      \right),
  \end{equation}
  where $\bar K_{M,\epsilon}$ is as defined in~\eqref{eq:KM}.

  Overall, with probability at least $1-2\delta$, such a point $x_K$ is first reached after a total of $T_{\rm lg} + T_{\rm saddle} + T_{\rm min}$ inner iterations of $\tCGbg$, bounded as above.
\end{theorem}
\begin{proof}%[Proof of Theorem~\ref{thm:global_inner_its}]
  Recall the partition $(\mathcal{G}, \mathcal{N}, \mathcal{M})$ of the iteration indices $\mathbb{N} = \{ 0, 1, 2, \ldots \}$ as defined in the beginning of Section~\ref{s:global_comp}.
  By definition, $|\mathcal{G}| + |\mathcal{N}| = K_{\rm cap}$.

  Further recall that $T_k$ denotes the total number of inner iterations performed by $\tCGbg$ during the $k$th outer iteration of the trust-region method.
  Below, we use the complexity estimates established earlier in this section to bound the sum of $T_k$ in each region.

  \paragraph{``Large gradient'' iterations.}
  Let $k \in \mathcal{G}$.
  We have $\|g_k\| \geq \bar G$ by Lemma~\ref{lemma:lowbound_G} and $\|\xi_k\| \leq \Delta_k / 4$ by design, so that Lemma~\ref{lemma:cg_complexity_delta} provides the first inequality in
  \begin{equation*}
    T_k
    \leq 1 + \frac{2L_G \Delta_k \cdot (\|g_k\| + \frac{1}{4}L_G \Delta_k)}{\left(\min(\omega_1 \|g_k\|, \omega_2 \|g_k\|^2)\right)^2}
    \leq 1 + \frac{2L_G \bar \Delta \cdot (\bar G + \frac{1}{4}L_G \bar \Delta)}{\left(\min(\omega_1 \bar G, \omega_2 {\bar G} ^ 2)\right)^2}
    \leq 1 + \frac{ (L_G \bar \Delta)^2 }{\left(\min(\omega_1 \bar G, \omega_2 {\bar G} ^ 2)\right)^2},
  \end{equation*}
  whereas the second inequality holds because the first bound is decreasing\footnote{Indeed, the function $\psi:t\mapsto \frac{t+a}{\min(bt,ct^2)^2}$ is decreasing, as $\psi(t)= \frac{1}{\min(b^2t,c^2t^3)}+\frac{a}{\min(b^2t^2,c^2t^4)}$.} as a function of $\|g_k\|$, and then also because $\Delta_k \leq \bar\Delta$ by design (see Algorithm~\ref{algo:pTR}).
  The third inequality comes from $\bar G = \mu \bar R / 2 \leq \mu \Delta_0 / 16 \leq L_G \bar \Delta  / 16$ by~\eqref{eq:def_constants} and the fact that $\mu \leq L_G$.
  This bound, along with $|\mathcal{G}| \leq K_{\rm cap} \leq \bar K_{G,N}$ from Theorem~\ref{thm:outer_complexity}, leads to the upper estimate of $T_{\rm lg}$.

  \paragraph{``Saddle point'' iterations.}
  Let $k \in \mathcal{N}$.
  The smallest eigenvalue of $H_k = \nabla^2 f(x_k)$ is $\lambdamin(H_k)$, which is negative.
  Let $q_k$ be a corresponding unit-norm eigenvector of $H_k$.
  By Lemma~\ref{lemma:tcg_nc_innerits}, the number of inner iterations performed by $\tCGbg$ at iteration $k$ is bounded as
  \[
  T_k \leq 2 + \sqrt{
    \frac{L_G}{8|\lambdamin(H_k)|}
  }
    \log\!\left(4\frac{\|g_k+H_k\xi_k\|^2}{\la q_k, g_k + H_k\xi_k\ra^2} \right).
  \]
  Since $k \in \mathcal{N}$, we have $\|x_k-\xbar\| \leq \bar R$ for some saddle point $\xbar$, and hence $|\lambdamin(H_k)| \geq \mu / 2$ by Lemma~\ref{lemma:growth_mu}.
  As usual, $\|\xi_k\| = \sigma \leq \bar \sigma \leq \bar R \leq \frac{1}{8} \Delta_k$ (Lemma~\ref{lemma:lowbound_deltak}).
  This allows us to bound the numerator inside the logarithm by
  \begin{equation} \label{eq:boundgkplusHkxiklate}
    \|g_k+H_k\xi_k\| \leq \|g_k\| + \|H_k \xi_k\| \leq L_G \big(\|x_k-\bar x\| + \sigma \big) \leq 2 L_G \bar R.
  \end{equation}
  For the denominator, we use Lemma~\ref{lemma:concentration_xi}: for any $\delta'' \in (0, 1)$ we have with probability at least $1-\delta''$ that
  \begin{equation*}
    \la q_k, g_k + H_k\xi_k\ra^2 = \big( \la H_k q_k, \xi_k \ra + \la q_k, g_k \ra \big)^2 \geq \frac{ \pi \sigma^2(\delta'')^2}{8d} \|H_k q_k\|^2 \geq \frac{ \pi \sigma^2(\delta'')^2\mu^2}{32d}.
  \end{equation*}
  Therefore, for each $k \in \mathcal{N}$, with probability at least $1-\delta''$ we have
  \[
  T_k \leq 2+\sqrt{
   \frac{L_G}{4\mu}
   }
   \log\!\left( \frac{512 L_G^2 {\bar{R}}^2 d}{\pi \sigma^2 (\delta'')^2 \mu^2}\right).
  \]
  Choose $\delta'' = \delta / \bar K_{G,N}$.
  Then, that bound along with $|\mathcal{N}| \leq K_{\rm cap} \leq \bar K_{G,N}$ from Theorem~\ref{thm:outer_complexity} and a union bound over all $k \in \mathcal{N}$ leads to the claimed upper bound on $T_{\rm saddle}$, valid with probability at least $1 - \bar K_{G,N} \delta'' = 1-\delta$.
  Recall that we had already conditioned on the event of probability at least $1-\delta$ in Theorem~\ref{thm:outer_complexity}: an additional union bound with that event yields the overall probability of at least $1-2\delta$.

  \paragraph{``Local minimizer'' iterations.}
  % Note: it's important to exclude work done at $x_K$ itself, because for all we know $x_K = x^*$, in which case $g_K = 0$.
  By assumption, $K_{\rm cap}$ is finite and $f$ has a local minimizer $x^*$ such that $\|x_k - x^*\| \leq \bar R$ for all $k \geq K_{\rm cap}$.
  In particular, $H_k \succeq \frac{\mu}{2}I_d$ for all such $k$.
  By the complexity estimate of $\tCGbg$ for strongly convex quadratics (Lemma~\ref{lemma:cg_rate_convex}) we have
  \begin{equation*}
    T_k \leq 1 + \sqrt{ \frac{2L_G}{\mu}} \log\!\left( 2 \sqrt{\frac{2L_G}{\mu}} \cdot \frac{\|g_k + H_k\xi_k\|}{\min(\omega_1\|g_k\|, \omega_2\|g_k\|^2)} \right).
  \end{equation*}
  The numerator in the log is upper bounded as in~\eqref{eq:boundgkplusHkxiklate}. % $\|g_k\| + L_G\|\xi_k\| \leq L_G \big(\|x_k-x^*\| + \sigma\big) \leq 2L_G \bar{R}$.
  To bound the denominator, we restrict our attention to $k$ in $K_{\rm cap}, \ldots, K-1$.
  Then, we also have $\|x_k - x^*\| > \epsilon$ so that $\|g_k\| \geq \frac{\mu}{2} \|x_k-x^*\| \geq \frac{\mu}{2} \epsilon$ (Lemma~\ref{lemma:growth_mu}).
  This leads to a bound on $T_k$ for the stated range of $k$.
  Since also $K - K_{\rm cap} \leq \bar K_{M,\epsilon}$ as per Theorem~\ref{thm:outer_complexity}, the proof is complete.
\end{proof}

To connect Theorem~\ref{thm:global_inner_its} with the simplified formulation in Section~\ref{sss:informal_thm}, refer to Appendix~\ref{app:informal_justif}.

\section{Numerical experiments}\label{s:numerical_experiments}

We implemented the proposed randomized trust-region method with $\tCGbg$ in Matlab, integrated into the Riemannian optimization toolbox Manopt~\citep{manopt}.
In doing so, we apply the usual recipes to generalize our algorithm from the Euclidean setting to the more general setting of Riemannian manifolds~\citep{genrtr,Absil2008,boumal2023intromanifolds}.
This extension makes it possible to test the algorithm on a wider range of problems.
Another advantage is that for some problems with symmetry, the Morse property may fail in the base formulation, yet hold when passing to the quotient manifold.

Code for the numerical experiments below and for the algorithm is available at:\\
\indent\href{https://github.com/NicolasBoumal/randomized-trustregions-code}{github.com/NicolasBoumal/randomized-trustregions-code} and\\
\indent\href{https://github.com/NicolasBoumal/manopt/tree/master/manopt/solvers/trustregions}{github.com/NicolasBoumal/manopt/tree/master/manopt/solvers/trustregions}.

\paragraph{Competing algorithms.}
We compare several versions of the trust-region method. % (since the purpose is to assess the role of randomization in that algorithmic framework).
They all rely on the same core Manopt code, so that we can reliably change specific parameters while keeping all else fixed.
In particular, they all use the same default parameter values ($\omega_1, \omega_2$, $\bar\Delta$, $\Delta_0$, $\rho'$, etc.) and the same  trust-region update mechanism (cosmetically different from Algorithm~\ref{algo:pTR}, see code).

The versions we compare differ along two axes:
\begin{enumerate}
  \item The inner problem solver can be either the classical tCG, or the new randomized tCG with boundary gradient step ($\tCGbg$). For the latter, we must set a noise level $\sigma$: we try both $\sigma = 10^{-6}$ and $\sigma = 10^{-3}$.

  \item The algorithm may use the true Hessian (``no Hessian shift''), or a shifted Hessian $\nabla^2 f(x_k) + \sqrt{\epsilon_{\mathrm{M}}} I_d$ where $\epsilon_{\mathrm{M}}$ is machine precision (``Hessian shift''). This shift is known to help with numerical stability when the Hessian is nearly singular at minimizers, which may happen if the Morse property does not hold.
\end{enumerate}
Notice that we make no effort to tune $\sigma$ to the specific problems: we simply fix a somewhat small value a priori.
Our implementation of TR-tCG-bg departs from theory as follows: in line~\ref{line:xik} of Algorithm~\ref{algo:pTR}, in principle we set $\xi_k$ to have norm $\min(\sigma, \Delta_k/4)$.
In the actual code, we set the norm to be $\min(\max(\sigma, \sqrt{\epsilon_{\mathrm{M}}}), \Delta_k/100)$; this effectively replaces the 4 by 100 and seemed slightly better.

\paragraph{Recorded metrics.}
For each experiment, we generate a random problem instance and a random initial guess $x_0$.
We then run each algorithm on the same problem instance from that same initial guess.
As each algorithm proceeds, we record the following metrics at each outer iteration:
\begin{itemize}
  \item The objective function value $f(x_k)$.
  \item The gradient norm $\|\nabla f(x_k)\|$.
  \item The cumulative number of inner iterations performed so far. Notice that the extra boundary gradient steps do not count as inner iterations here.
  \item The cumulative number of Hessian-vector products performed so far: this \emph{mostly} matches the inner iteration counts, up to two subtleties:
  \begin{enumerate}
    \item When $\tCGbg$ uses an extra boundary gradient step, the associated Hessian-vector product \emph{is} accounted for here.
    \item In Manopt's implementation of TR-tCG, if step $k$ is rejected, then the Hessian-vector products computed by tCG at iteration $k$ are reused at iteration $k+1$. This refinement could be brought to $\tCGbg$ as well (with mild tweaks to the theory), but we did not implement it. As a result, $\tCGbg$ may use more Hessian-vector products than required.
  \end{enumerate}
\end{itemize}

\paragraph{Observations.}

As stated in the introduction, classical TR-tCG is already quite effective at escaping saddle points.
It is its worst-case behavior that is lacking in theory.
Thus, the baseline is merely to verify that our proposed randomized variant is comparable to the deterministic one, that is, the practical price of the theoretical benefits should not be too high.
We find that this is mostly the case.
We might further hope that the randomized variant is actually better at escaping saddles.
That is not particularly borne out: even when initializing close to a saddle point, randomization is not the decisive factor.
(Of course, if we initialize \emph{exactly} at a saddle point, then \emph{only} the randomized variants succeed; but this is unlikely to happen in practice.)

% Notes to selves: About trying a simple saddle to see what happens:
% \begin{enumerate}
%   \item With $f(x) = \frac{1}{2}\|x\|^2 - x_1^2$, if you initialize at $x_0 = 0$ then of course deterministic methods are stuck but randomized methods escape in 2 Hessian-vector products. This is highly unusual though: if $(x_0)_1$ is even slightly nonzero, classical TR-tCG escapes just as well as our randomized method. This goes to the narrative that standard methods actually behave quite well: it is their worst-case behavior that hinders good theory, but their practical behavior is quite good, and this is why practical randomized methods should not detract them too much.
%   \item With $f(x) = \frac{1}{2}\|x\|_4^4 - x_1^4$, initialize very close to zero, both TR-tCG and the randomized method escape equally well (5--20 Hess-vecs) if they have NO Hessian shift. Otherwise, stuck.
%   \item Emphasize that deterministic methods fail (in that they are truly stuck) if they are initialized exactly (both mathematically and numerically exactly) on a saddle point. Randomized methods can escape even in that situation. But this is an adversarial setup: numerical ``noise'' is enough to allow deterministic methods to escape too.
% \end{enumerate}

\paragraph{Problem 1: sine saddle.} % sinesaddle.m
Fix $d = 10^5$.
Let $w_1 = -10^{-2}$ and sample $w_2, \ldots, w_d > 0$ i.i.d.\ uniformly at random in the interval $[1, 2]$.
The (artificial) cost function is\footnote{For a vector $x \in \Rd$, we write $x_i$ to denote the $i$th entry of $x$; not to be confused with $x_k$ which usually denotes the $k$th iterate of an algorithm.}
\begin{align}
  f(x) = -w_1 + \sum_{i=1}^d w_i \sin(x_i)^2.
  \label{eq:sinesaddle}
\end{align}
Its minimal value is zero, attained when $\sin(x_1)^2 = 1$ and $\sin(x_i)^2 = 0$ for $i = 2, \ldots, d$.
The origin ($x = 0$) is a saddle point with value $f(0) = -w_1 = 10^{-2}$ and $\nabla^2 f(0)$ has eigenvalues $2w_1, \ldots, 2w_d$: exactly one of these is negative.
We initialize all algorithms at a common, random initial point close to the origin.
Figure~\ref{fig:sinesaddle} shows that all methods are able to escape this saddle point efficiently, to then converge to a global minimizer.

For additional context, the gradient is $\nabla f(x) = w \odot \sin(2x)$ (where $\odot$ denotes entry-wise product) and the Hessian is $\nabla^2 f(x) = 2 \diag(w \odot \cos(2x))$, where $\diag$ forms a diagonal matrix from a vector in the usual way.
A point $x$ is critical exactly if $\sin(2x_i) = 0$ for all $i$.
At such points, $\cos(2x_i) = \pm 1$ so that $f$ satisfies the $\mu$-Morse property with $\mu = 2\min_i |w_i| = 2 \cdot 10^{-2}$.
Also, $\nabla f$ is $L_G$-Lipschitz continuous with $L_G = \max_x \opnorm{\nabla^2 f(x)} = 2\|w\|_\infty \leq 4$, while $\nabla^2 f$ is $L_H$-Lipschitz continuous with $L_H = 4\|w\|_\infty \leq 8$. % because $\opnorm{\nabla^2 f(x) - \nabla^2 f(y)} = 2\| w \odot (\cos(2x) - \cos(2y)) \|_\infty \leq 2 \|w\|_\infty \|\cos(2x) - \sin(2x)\|_\infty \leq 4 \|w\|_\infty \|x - y\|_\infty$ then bound the last $\|x - y\|_\infty \leq \|x - y\|$.

% Note: Here we see a different picture depending on whether you plot inner iterations or Hess-vecs, because initial step is rejected, and deterministic method can recycle. Might also change $\Delta_0$ I suppose. At $n=10^5$ get clear picture; at $n=10^6$, it takes a bit longer, and we see more prominently the role of recycling: the initial $\Delta_0$ is too large, and it takes 4 rejections before things get started. That's expensive for us, but not for the deterministic method. It's not a relevant difference though, because we could recycle for the randomized method too; it's just a matter of putting in a lot of engineering efforts in the code, and making the proofs even longer: not difficult, but let's not go there.

\paragraph{Problem 2: nonlinear synchronization.} % nonlinearsynchro.m
Fix $d = 3$, $n = 1000$ and $\beta = 6$.
Following the setup in~\citep{criscitiello2026synchcircles}, we minimize the function
\begin{align}
  f(X) = - \frac{1}{2\beta n^2} \sum_{i = 1}^n \sum_{j = 1}^n e^{\beta x_i^\top x_j}
  \label{eq:nonlinearsynchro}
\end{align}
where $X \in \reals^{d \times n}$ has unit-norm columns denoted by $x_1, \ldots, x_n$.
This is a potential for $n$ particles on a sphere which attract each other with nonlinear forces.
The minimizers are the synchronized states, namely, any configuration such that $x_1 = \cdots = x_n$.
In particular, the minimizers form a smooth manifold and hence we do \emph{not} have the $\mu$-Morse property, but there is still a lot of structure.

We first run gradient descent for a while, initialized at a random initial point: this reliably brings us near a saddle point.
Then, we run each trust-region variant from that point.
Here too, Figure~\ref{fig:nonlinearsynchro} shows that all methods are able to escape this saddle point efficiently, to then converge to a global minimizer.

\paragraph{Problem 3: rectangular matrix approximation.} % matrixapprox_rect.m
Fix $m = 1000, n = 3000$ and generate a matrix $A \in \reals^{m \times n}$ with approximately 1\% of entries chosen at random with i.i.d.\ values uniformly in $[0, 1]$, and all other entries set to zero.
We minimize the cost function
\begin{align}
  f(L, R) = \frac{1}{2} \|LR^\top - A\|_\mathrm{F}^2 + \frac{\lambda}{2} \left( \|L\|_{\mathrm{F}}^2 + \|R\|_{\mathrm{F}}^2 \right),
  \label{eq:matrixapprox_rect}
\end{align}
with $L \in \reals^{m \times r}, R \in \reals^{n \times r}$.
This corresponds to seeking a rank-$r$ approximation of $A$ with a nuclear norm regularization controlled by $\lambda$.
We initialize $L, R$ close to zero (a saddle point) and try three different choices of $(r, \lambda)$, reported in Figure~\ref{fig:matrixapprox_rect}:
\begin{enumerate}
  \item $(r = 1, \lambda = 0)$: The solution forms a manifold since $f(\alpha L, \frac{1}{\alpha} R) = f(L, R)$ for all $\alpha \neq 0$. Thus, the $\mu$-Morse property does not hold, and we can see that only the methods with Hessian shift behave well.
  \item $(r = 1, \lambda = 0.01)$: The added regularization ensures uniqueness of the solution, and all methods behave well.
  \item $(r = 2, \lambda = 0.01)$: The increased rank leads to non-uniqueness of the solution again (despite regularization) since $f(LQ, RQ) = f(L, R)$ for all orthogonal $Q \in \reals^{r\times r}$. Here too, the Hessian shift is useful.
\end{enumerate}
In all cases, we see that the algorithms manage to escape the saddle, and they reach a minimizer.

\paragraph{Problem 4: positive semidefinite matrix approximation.} % matrixapprox_psd.m
Fix $n = 3000$ and generate a symmetric matrix $A \in \reals^{n \times n}$ with approximately 1\% of entries chosen at random with values uniformly in $[0, 1]$ (i.i.d.\ up to the fact that $A = A^\top$), and all other entries set to zero.
We minimize
\begin{align}
  f(X) = \frac{1}{4}\|XX^\top - A\|_{\mathrm{F}}^2
  \label{eq:matrixapprox_psd}
\end{align}
with $X \in \reals^{n \times r}$.
This corresponds to seeking a rank-$r$ positive semidefinite approximation of $A$.
We initialize $X$ close to zero (a saddle point) and try three different versions, reported in Figure~\ref{fig:matrixapprox_psd}:
\begin{enumerate}
  \item $r = 1$: This is Example~\ref{ex:mb} (with possibly indefinite $A$, see Lemma~\ref{lemma:mu_morse_example}) which fits our assumptions. Algorithms behave well.
  \item $r = 2$: The increased rank breaks the $\mu$-Morse property since $f(XQ) = f(X)$ for all orthogonal $Q \in \reals^{r \times r}$, and we see that the Hessian shift starts to play a role.
  \item $r = 2$, quotiented: One can quotient out the invariance of $f$ under the aforementioned action of the orthogonal group, reducing the domain of $f$ to the quotient $\reals^{n \times r} / \sim$ where $X \sim Y$ if there exists an orthogonal $Q$ such that $X = YQ$. In doing so, the uniqueness of the solution is reinstated, and with it the $\mu$-Morse property as well. As a result, all algorithms work well. See~\citep{massart2020quotientpsd} for details.
\end{enumerate}
Here too, the algorithms escape the saddle and reach a minimizer.

\begin{figure}[htbp]
  \centering
  \includegraphics[width=1.0\linewidth]{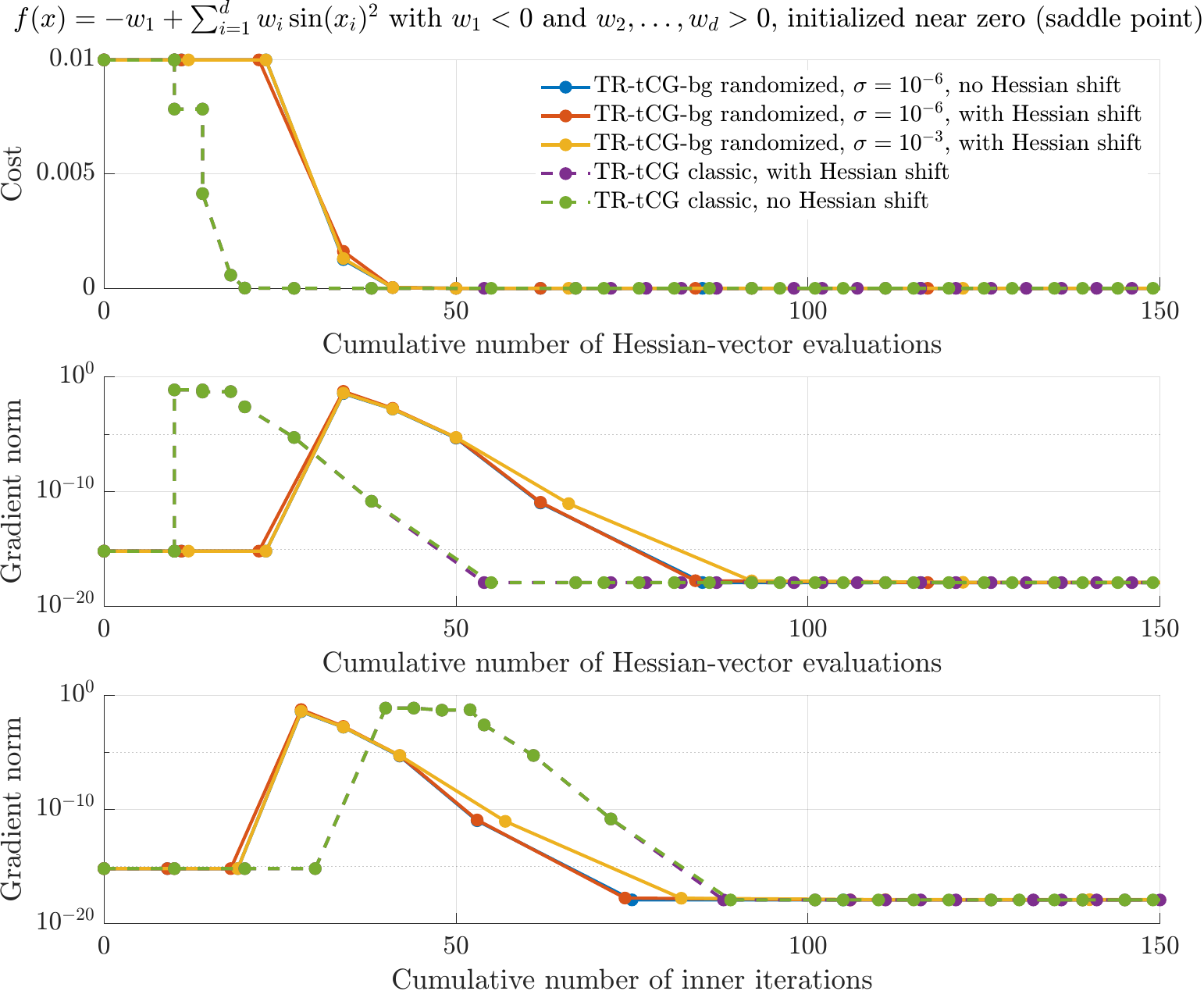}
  \caption{Experimental run for Problem 1 (sine saddle, see~\eqref{eq:sinesaddle}). All methods escape the saddle point and converge to a global minimizer. The dashed lines (TR-tCG with and without Hessian shift) overlap almost perfectly, as do the curves for TR-tCG-bg with $\sigma = 10^{-6}$.}
  \label{fig:sinesaddle}
\end{figure}

\begin{figure}[htbp]
  \centering
  \includegraphics[width=1.0\linewidth]{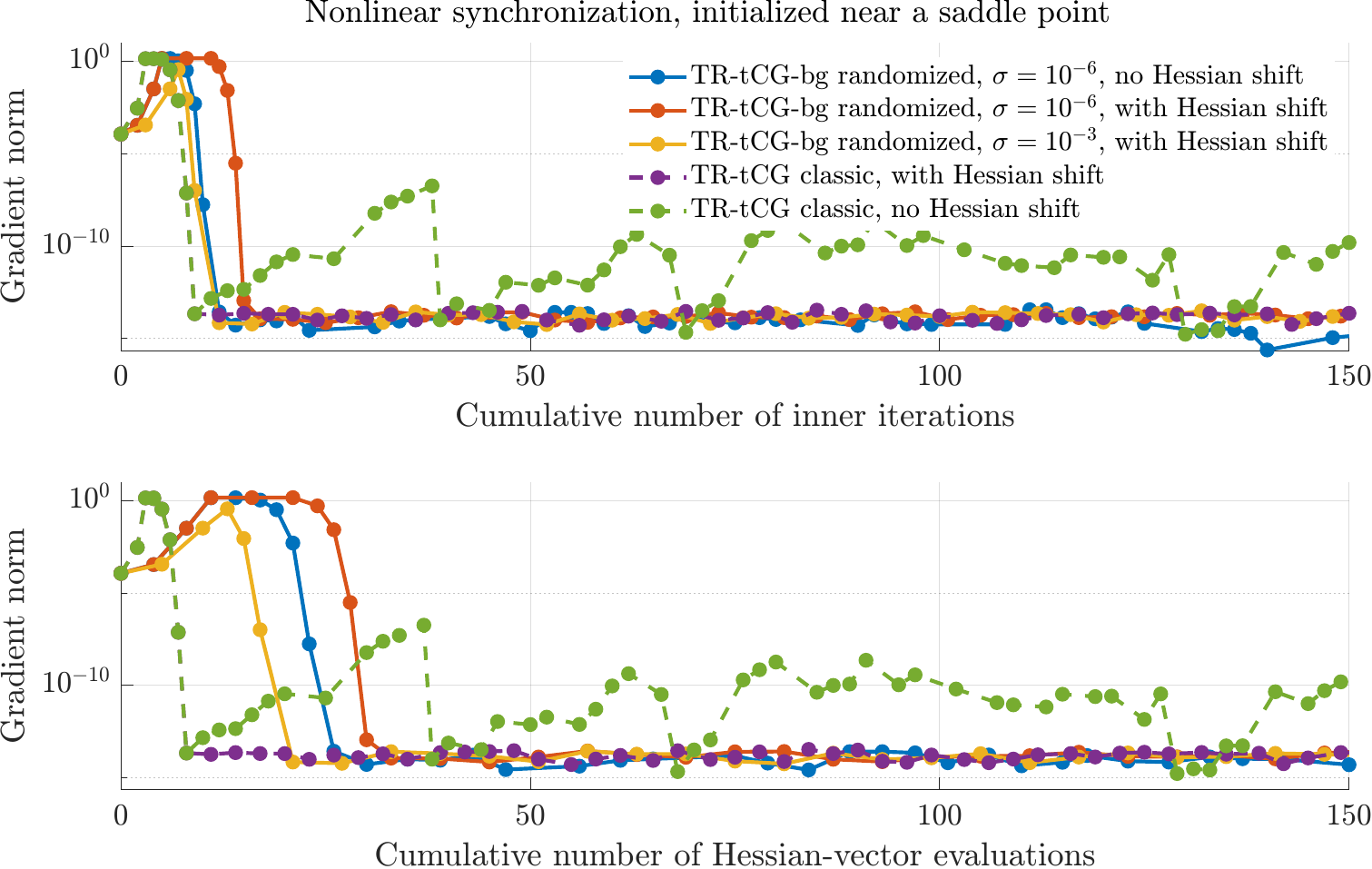}
  \caption{Experimental run for Problem 2 (nonlinear synchronization, see~\eqref{eq:nonlinearsynchro}).}
  \label{fig:nonlinearsynchro}
\end{figure}

\begin{figure}[htbp]
  \centering
  \includegraphics[width=1.0\linewidth]{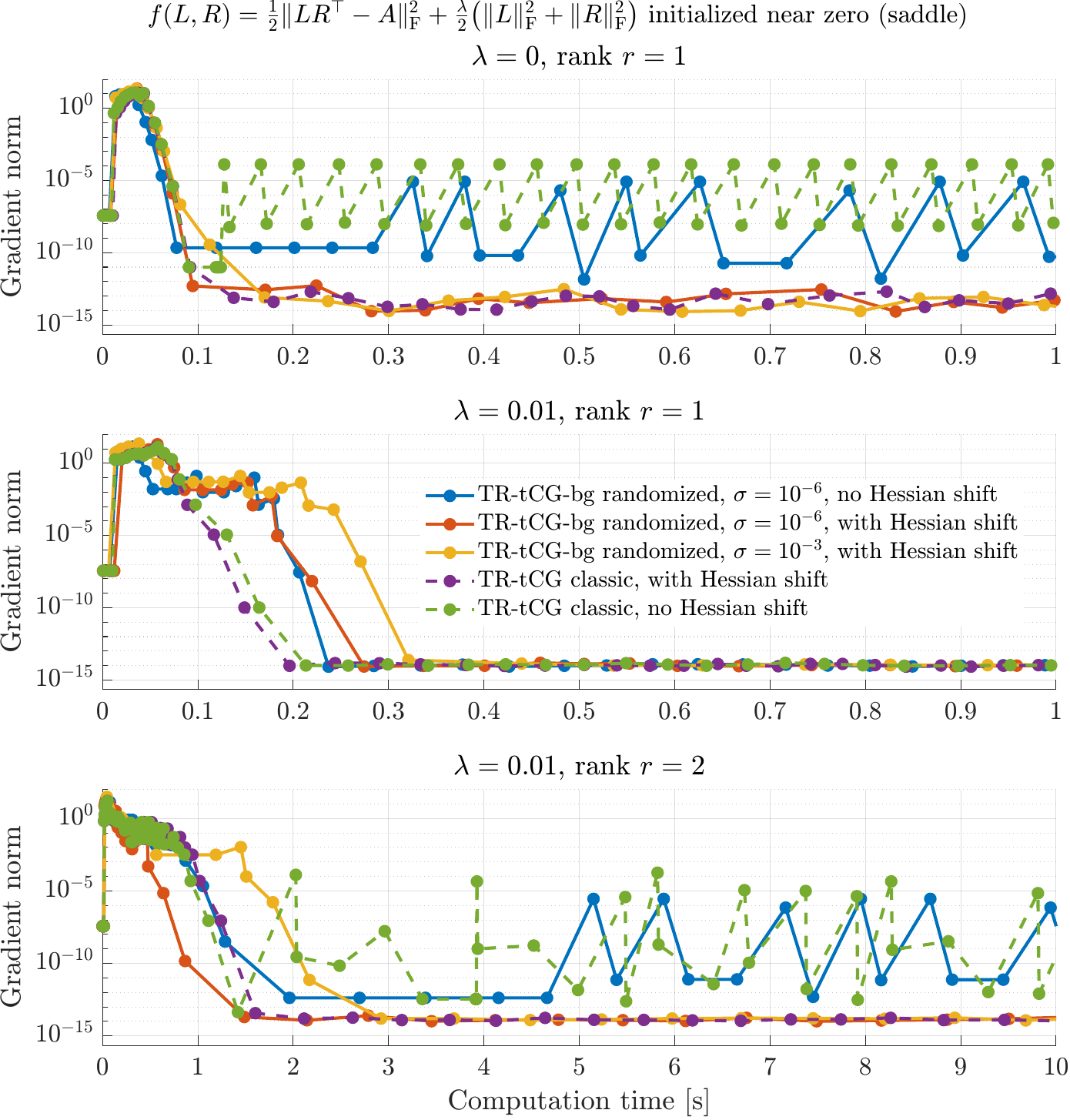}
  \caption{Experimental run for Problem 3 (rectangular matrix approximation, see~\eqref{eq:matrixapprox_rect}) with three different choices of rank $r$ and regularization parameter $\lambda$.}
  \label{fig:matrixapprox_rect}
\end{figure}

\begin{figure}[htbp]
  \centering
  \includegraphics[width=1.0\linewidth]{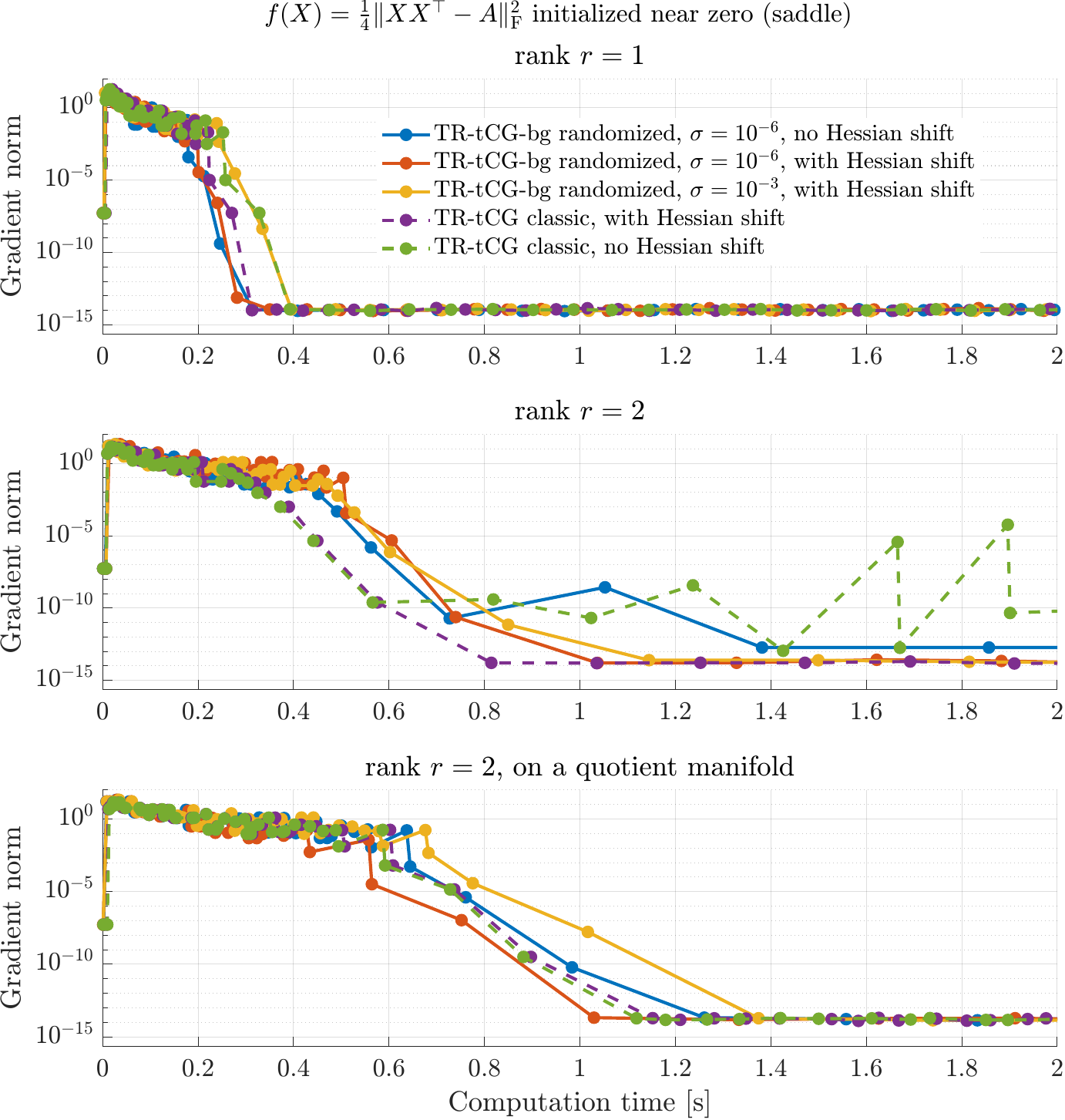}
  \caption{Experimental run for Problem 4 (positive semidefinite matrix approximation, see~\eqref{eq:matrixapprox_psd}) with two different choices of rank $r$, and one experiment on a quotient manifold.}
  \label{fig:matrixapprox_psd}
\end{figure}

\section{Further discussion of the theoretical results} \label{s:discussion}

In this section, we explain why the maximal radius $\bar \Delta$ appears in the complexity bounds, and we prove that randomization is essential to break the $\Omega(d)$ complexity of deterministic methods.

\subsection{Inner iteration complexity: dependency on $\bar \Delta$}

In Theorem~\ref{thm:global_inner_its}, the estimate for the number of $\tCGbg$ steps for \emph{large gradient} iterations depends on the maximal trust-region radius $\bar \Delta$.
This parameter is typically chosen large, as a safeguard.

The dependency is unavoidable: for \emph{general non-convex} quadratics, no better worst-case estimate is possible.
Efficient rates can only be proved for strongly convex models (Lemma~\ref{lemma:cg_rate_convex}) or when the Hessian has a sufficiently small negative eigenvalue (Lemma~\ref{lemma:tcg_nc_innerits}).
In the general case, when the iterates are not near a local minimizer or a saddle point, the spectrum of the Hessian can be arbitrary.
The only way to establish a bound is to use the fact that the total model decrease is bounded, as its iterates are confined in the trust region with radius $\Delta_k$.
In the worst case, this radius can be as high as $\bar\Delta$.

This limitation stems from worst-case analysis.
A solution could be to set an artificial \emph{cap} on the $\tCGbg$ number of iterations, as done by \citet{Royer2020}.
We chose not to do so as this introduces additional parameters and can degrade practical performance.

\subsection{Why randomization is necessary}
\label{ss:justification_randomization}

The fact that randomization is necessary to achieve complexities that are sublinear in $d$ is folklore: see for example~\citep[Ex.~2.1]{simchowitz2017gapstrictsaddlestrueconvexity}, where it is stated in passing for a different though related context.
That paper (and its follow-up in~\citep{simchowitz2018queryLBforPCA}) further shows that adaptive randomization brings the lower-bound down to $\Omega(\log d)$, that is, a polylogarithmic dependence on $d$ is inevitable.

Let us justify explicitly that randomization is necessary in our context.
We do so with an information-theoretic lower bound based on the following ``worst-cast function'':
\begin{equation}
  f(x) = \cos(x[d]) - 1 + \frac{1}{2} \sum_{i=1}^{d-1} x[i]^2,
  \label{eq:worstcasefunction}
\end{equation}
where $x[i]$ denotes the $i$th coordinate of $x \in \Rd$.
It is easily verified that $f$ satisfies Assumption~\ref{ass:f} with $L_G = 1$ and $L_H = 1$.
It also satisfies the $\mu$-Morse property with $\mu = 1$.
The minimum of $f$ is $\flow = -2$, and the origin is a saddle point with value $f(0) = 0$.
Moreover, since the critical points of $f$ are of the form $(0, \ldots, 0, k\pi)$ with $k \in \mathbb{Z}$, we can reason that the distance from any $x$ to the closest critical point is at most $\frac{\pi}{2} \|\nabla f(x)\|$.
% Note: to see this, start from:
%  dist(x, critical set)^2 = min_{k \in Z} x_1^2 + ... + x_{d-1}^2 + (x_d - k\pi)^2
%                          = \|\nabla f(x)\|^2 - \sin(x_d)^2 + min_{k \in Z} (x_d - k\pi)^2
% Then check: [min_{k \in Z} |t - k\pi|] = arcsin(|sin(t)|) \leq \frac{\pi}{2} |\sin(t)|.
%
%   t = linspace(-10, 10, 501);
%   plot(t, min(mod(t, pi), pi-mod(t, pi)), t, (pi/2)*abs(sin(t)));
%
% Now use $|\sin(x_d)| \leq \|\nabla f(x)\|$.
Therefore, $f$ satisfies the $(R_s, \gamma_s)$-strong gradient property with $R_s = \frac{1}{4}$ and $\gamma_s = \frac{1}{2\pi}$, as required by Assumption~\ref{ass:mumorse_and_stronggrad}.

To prove the lower bound, we use the classical technique of \emph{adversarial rotations} \citep{Nemirovski1983,Guzman2015}.
At a high level: as an algorithm queries $f$, $\nabla f$ and $\nabla^2 f$ (through Hessian-vector products), the ``resisting oracle'' rotates $f$ so as to reveal as little information as possible, while remaining consistent with past queries.
This produces a rotated function $f$ such that this (deterministic) algorithm cannot reach a value lower than $f(0)$ in fewer than $d/2 - 1$ queries.
Stated differently: bypassing the saddle at the origin requires $\Omega(d)$ queries.

Our class of algorithms is as follows.
Many (including standard TR-tCG) belong to this class.

\begin{definition}\label{def:algo}
  A \emph{deterministic second-order algorithm} is a procedure which generates a sequence of pairs $(x_1, u_1), (x_2, u_2), \ldots$ in $\Rd \times \Rd$ in the following manner: for $k\geq 1$,
  \begin{itemize}
    \item Choose $(x_k, u_k)$ as a deterministic function of $\left\{ \left( f(x_i), \nabla f(x_i), \nabla^2 f(x_i)[u_i] \right) \right\}_{1 \leq i \leq k-1}$.
    \item Query the oracle at $(x_k, u_k)$ to obtain $\left( f(x_k), \nabla f(x_k), \nabla^2 f(x_k)[u_k] \right)$.
  \end{itemize}
\end{definition}

\begin{proposition}\label{prop:oracle_complexity}
  Let $K \geq 1$.
  For every deterministic second-order algorithm $\mathcal{A}$, there exists a smooth function $f \colon \reals^{2K+1} \to \reals$ satisfying Assumptions~\ref{ass:f} and~\ref{ass:mumorse_and_stronggrad} such that, upon running $\mathcal{A}$ on $f$, the first $K$ query points $x_1, \ldots, x_K$ satisfy $f(x_k) \geq f(0) = 0$ for all $k = 1, \ldots, K$, whereas the minimal value $\min_{x} f(x) = -2$ is attained at a point $x^*$ with norm $\|x^*\| = \pi$.
\end{proposition}
\begin{proof}
  The resisting oracle proceeds as follows: when the algorithm queries $(x, u)$ to obtain the values $\left(f(x), \nabla f(x), \nabla^2 f(x)[u] \right)$, it returns $(\frac{1}{2}\|x\|^2, x,u)$.
  Let us explain why this works.

  Assume that the algorithm has performed at most $K$ oracle calls.
  Then, all queried points $x_1, \ldots x_{K}$ and directions $u_1,\ldots u_{K} $ belong to a subspace $\mathcal{V}_{K}$ of dimension at most $2K$ in $\reals^{2K+1}$.
  Let $q$ be a unit-norm vector orthogonal to $\mathcal{V}_{K}$, and let $V_{K} \in \reals^{(2K+1)\times (2K+1)}$ be the orthogonal projection matrix on $\mathcal{V}_{K}$.
  Consider the function $f$ given by
  \[
     f(x) = \frac{1}{2}\| V_K x\|^2 + \cos\!\big( q^\top x \big) -1.
  \]
  (Notice this is simply~\eqref{eq:worstcasefunction} in a rotated coordinate system.)
  Its gradient and Hessian are given by
  \begin{align*}
    \nabla f(x) = V_K x - \sin\!\big( q^\top x \big) q, &&
    \nabla^2 f(x)[u] = V_K u - \cos\!\big( q^\top x \big) (q^\top\!u) q .
  \end{align*}
  For all $x, u \in \mathcal{V}_{K}$, we have $\left( f(x),\nabla f(x),\nabla^2 f(x)u\right) = (\frac{1}{2}\|x\|^2, x,u)$.
  Thus, the oracle's outputs are consistent with this $f$ so far.
  Since the algorithm is \textbf{deterministic} and uses only the information given by the oracle, the points $x_1, \ldots, x_{K}$ are exactly those generated by algorithm $\mathcal{A}$ applied to $f$, and $f$ indeed has the announced properties.
\end{proof}

In other words, Proposition~\ref{prop:oracle_complexity} shows that no deterministic second-order method can escape saddles points (in function value) in less than $(d-1)/2$ oracle calls, where $d = 2K$ is the dimension of the optimization space.
In contrast, when applied to $f$~\eqref{eq:worstcasefunction} or any of its rotations, our method escapes with high probability to find a point $\hat{x}$ satisfying $f(\hat{x}) - f(0) \leq -\Omega(1)$ in at most $\mathcal{\tilde{O}}\!\left(\sqrt{\frac{L_H}{\mu}}\right) = \mathcal{\tilde{O}}(1)$ steps (Theorem~\ref{thm:global_inner_its}), with a polylogarithmic dependence on $d$.

\section{Perspectives}

Our work opens several directions for further research.

On the theory side, we believe that some parts of the complexity bounds could be sharpened.
In particular, the dependence of the outer-iteration estimate on the condition number $L_G / \mu$ appears to be suboptimal.
Improving this would require a finer analysis of the algorithm's behavior around saddle points across multiple iterations.
It would also be interesting to reduce the dependence on the maximal radius $\bar \Delta$.
While this dependence appears unavoidable for general non-convex models, it could be overcome with new assumptions.
It would also be interesting to extend the analysis to Morse--Bott functions (allowing for non-isolated critical points), for example using tools developed by~\citet{rebjock2025thesis}.

On the practical side, an important issue is finding a reliable \emph{stopping criterion}.
At the moment, our analysis guarantees that the method finds an $\epsilon$-local minimizer point in a certain number of iterations, with high probability.
However, Algorithm~\ref{algo:pTR} provides no mechanism for certifying that a given iterate is indeed an $\epsilon$-minimizer.
The difficulty is that, when the gradient is small, the algorithm cannot determine whether the current iterate is near a saddle or a local minimizer (it could if it knew the problem constants $\mu$ etc.).
Another question is whether we can adjust the noise scale $\sigma$ adaptively, thus eliminating the need for even that single additional parameter.
The task is to bring these improvements while preserving the good behavior of the algorithm.

% THIS LINE EXCLUDES THE APPENDICES FROM THE TOC
\addtocontents{toc}{\protect\setcounter{tocdepth}{0}}

\appendix

\section{A tool and an example to check problem assumptions}\label{app:intro_proofs}

This appendix holds additional details for the discussion in Section~\ref{ss:assumptions}, related to Definitions~\ref{def:mu_morse} ($\mu$-Morse) and~\ref{def:strong_saddle} (strong gradient).

\begin{proposition}[Sufficient condition for strong gradient property]\label{prop:strong_gradient}
   Let $f\colon\reals^d \rightarrow \reals$ be $C^2$.
   Assume there exist $B,C> 0$ such that $\|\nabla f(x)\| \geq B$ for every $x$ such that $\|x\| \geq C$.
   Then, for every radius $R > 0$, there exists $\gamma > 0$ such that $f$ satisfies the $(R,\gamma)$-strong gradient property.
\end{proposition}
\begin{proof}
   Let $\mathrm{Crit} f$ be the set of critical points of $f$.
   For some $R>0$, define the set
     \[
     \mathcal{D} = \{ x \in \reals^d \,:\, \|x\| \leq C \textrm{ and }  \;\forall \,\xbar \in \mathrm{Crit} f, \;\|x-\xbar\| \geq R  \}.
     \]
     This is an intersection of closed sets and it is bounded: it is therefore compact. Define
     \[
     \gamma_{\mathcal{D}} = \inf_{x\in \mathcal{D}} \|\nabla f(x)\|.
     \]
     Since $\mathcal{D}$ is compact and $f$ is $C^2$, there exists $y \in \mathcal{D}$ such that $\|\nabla f(y)\| = \gamma_{\mathcal{D}}$. Because $\mathcal{D}$ and $\mathrm{Crit} f$ are disjoint, $y$ is not a critical point and hence $\gamma_{\mathcal{D}} > 0$.

     Then, $f$ satisfies the $(R,\gamma)$-strong gradient property with $\gamma = \min(B,\gamma_{\mathcal{D}})$.
     Indeed, for each $x \in \reals^d$, at least one of the following is true: (i) there exists $\xbar \in \mathrm{Crit} f$ such that $\|x-\xbar\| \leq R$, or (ii) $x \in \mathcal{D}$ and $\|\nabla f(x)\| \geq \gamma_{\mathcal{D}}$, or (iii) $\|x\| \geq C$ and $\|\nabla f(x)\| \geq B$.
\end{proof}

\begin{lemma}[Example of $\mu$-Morse function.]\label{lemma:mu_morse_example}
  Let $M \in \symm_d$ be a symmetric matrix.
  Assume its eigenvalues are distinct and denote them $\lambda_1 > \cdots > \lambda_{d'} > 0 > \lambda_{d'+1} > \cdots > \lambda_d$ with corresponding unit eigenvectors $u_1, \ldots, u_d$.
  Consider the function $f$ defined for $x\in \reals^d$ by
  \[
    f(x) = \frac{1}{4}\|xx^\top - M\|_{\mathrm{F}}^2.
  \]
  Then $f$ satisfies the $\mu$-Morse property with constant
  \[
    \mu = \min\!\big( \lambda_1 - \lambda_2, \ldots, \lambda_{d'-1} - \lambda_{d'}, \lambda_{d'} - 0, 0 - \lambda_{d'+1}, \lambda_{d'+1} - \lambda_{d'+2}, \ldots, \lambda_{d-1} - \lambda_d \big) > 0.
  \]
\end{lemma}
\begin{proof}
The gradient is $\nabla f(x) = (xx^\top - M)x$ and therefore the set of critical points of $f$ is
\[
  \mathrm{Crit} f = \{0\}\cup \{ \pm \sqrt{\lambda_i} u_i \, : \, i = 1, \ldots, d'\}.
\]
The Hessian is $\nabla^2 f(x)= \|x\|^2 I_d + 2 xx^\top - M$.
At the origin, we have $\nabla^2 f(0) = -M$: all eigenvalues have absolute value at least $\min(\lambda_{d'}, -\lambda_{d'+1})$.
Consider now the pair of critical points $\pm\sqrt{\lambda_i}u_i$ for some $i \in \{1, \ldots, d'\}$.
The Hessian is then $H_i = \lambda_i I_d + 2\lambda_i u_iu_i^\top - M$.
Letting $\lambda$ be an eigenvalue of $H_i$ and $v$ the corresponding eigenvector, we have
\[
  \lambda_i v + 2\lambda_i  u_i u_i^\top v - Mv  = \lambda v.
\]
Write $v = \alpha u_i + u_\perp$ with $u_i^\top u_\perp = 0$.
It follows from the previous equation that
\begin{align*}
  (2\lambda_i - \lambda) \alpha = 0 && \textrm{ and } &&
  Mu_\perp = (\lambda_i - \lambda)u_\perp.
\end{align*}
If $\alpha \neq 0$, then $\lambda = 2\lambda_i$, which is larger than $\lambda_{d'}$.
If $\alpha = 0$, then $u_\perp$ is nonzero and the second equation implies that $\lambda = \lambda_i - \lambda_j$ for some $j \neq i$.
Thus, at each critical point, each eigenvalue of the Hessian has absolute value at least $\mu$.
\end{proof}

\section{Additional proofs for Section~\ref{s:preliminaries}}\label{app:additional}

This appendix holds the deferred proofs for preliminary lemmas stated in Section~\ref{s:preliminaries}.

\begin{proof}[Proof of Lemma~\ref{lemma:growth_mu}]
  Let us prove item~\ref{item:growth_mu1}.
  Since the singular values of $\nabla^2 f(\xbar)$ are at least $\mu$, using the Lipschitz Hessian property and $\nabla f(\bar x) = 0$ we find
  \begin{align*}
    \| \nabla f(x) \|
      & = \left\|\int_{0}^{1} \nabla^2 f\left(\xbar + t(x-\xbar)\right)(x-\xbar) \dt \right\| \\
      & \geq \| \nabla^2 f(\xbar)(x-\xbar)\| - \left\|\int_{0}^{1}  \big[\nabla^2 f\left(\xbar + t(x-\xbar)\right) - \nabla^2 f(\xbar)\big](x-\xbar) \dt \right\| \\
      & \geq \mu\| x-\xbar\| - \frac{L_H}{2}\|x-\xbar\|^2 \\
      & \geq \mu(1 - \frac{c}{2})\|x-\xbar\|.
  \end{align*}
  To prove the second inequality, we write
  \begin{align*}
    \|\nabla^2 f(x) u\|
      & \geq \|\nabla^2 f(\xbar) u\| - \opnorm{\nabla^2 f(x)-\nabla^2 f(\xbar)} \|u\| \\
      & \geq \mu \|u\| - L_H\|x-\xbar\|  \|u\| \\
      & \geq \mu (1-c)\|u\|.
  \end{align*}
  To prove item~\ref{item:growth_mu2}, assume that $\xbar$ is a saddle point of $f$.
  Then, $\nabla^2 f(\xbar)$ has a negative eigenvalue, and it must be less than $-\mu$.
  Therefore, there exists a vector $\bar u$ such that $\la \nabla ^2 f(\xbar)\bar u, \bar u \ra \leq -\mu \|\bar u\|^2$, so that
  \[
    \la \nabla^2 f(x) \bar u, \bar u\ra \leq \la \nabla^2 f(\bar x) \bar u, \bar u \ra + L_H\|x-\xbar\| \|\bar u\|^2 \leq (-\mu  + \mu c) \|\bar u\|^2
    % \leq -(\mu+c) \|\bar u\|^2,
  \]
  implies that $\nabla^2 f(x)$ has an eigenvalue smaller than $(-1+c)\mu$.

  If on the other hand $\xbar$ is a local minimizer of $f$, then $\nabla^2 f(\xbar) \succeq \mu I_d$ and so
  \[
    \la \nabla^2 f(x)u,u \ra \geq \la \nabla^2 f(\xbar)u,u \ra - L_H \|x-\xbar\| \|u\|^2 \geq \mu \|u\|^2 - \mu c \|u\|^2
  \]
  for all $u \in \reals^{d}$, which proves item~\ref{item:growth_mu3}.
\end{proof}

\begin{lemma}\label{lemma:cg_equiv}
  For $0\leq t \leq \Tin$, the iterates $(v^{(t)},p^{(t)},r^{(t)})$ generated by Algorithm~\ref{algo:tCG} initialized at $v^{(0)}=\xi$ and applied to the quadratic $v \mapsto m(v)$ are equal to $(\xi + \tilde{v}^{(t)},\tilde{p}^{(t)},\tilde{r}^{(t)})$, where $( \tilde{v}^{(t)},\tilde{p}^{(t)},\tilde{r}^{(t)})$ are the iterates of CG initialized at $\tilde{v}^{(0)}=0$ and applied to the translated quadratic
  $
    \tilde v \mapsto m(\xi + \tilde{v}). % = m(\xi) + \la g + H\xi,\tilde{v} \ra +\frac{1}{2} \la H\tilde{v},\tilde{v} \ra .
  $
\end{lemma}
\begin{proof}[Proof of Lemma~\ref{lemma:cg_equiv}]
  The proof is by induction. By construction, we have $v^{(0)}=\xi$ and $\tilde v^{(0)}=0$ thus $v^{(0)}=\xi+\tilde v^{(0)}$. Define the translation $\tilde v := v-\xi$ and the translated quadratic
  \[
  \tilde m(\tilde v) := m(\xi+\tilde v)
  = m(\xi) + \langle g+H\xi,\tilde v\rangle + \tfrac12 \langle H\tilde v,\tilde v\rangle .
  \]
  Then we see that $\tilde r^{(0)} = -(g + H\xi) = r^{(0)}$ and therefore $\tilde p^{(0)} = p^{(0)}$. Assume for some $t\ge 0$ (with $t+1\le \Tin$~\eqref{eq:def_tin}) that
  \[
  v^{(t)}=\xi+\tilde v^{(t)},\qquad r^{(t)}=\tilde r^{(t)},\qquad p^{(t)}=\tilde p^{(t)}.
  \]
  Then by definition $\alpha^{(t+1)} = \tilde\alpha^{(t+1)}$, thus
  \[
  v^{(t+1)}=v^{(t)} + \alpha^{(t+1)}p^{(t)} = \xi+\tilde v^{(t)} + \tilde \alpha^{(t+1)}\tilde p^{(t)} = \xi + \tilde v^{(t+1)} \qquad \text{and} \qquad r^{(t+1)}=\tilde r^{(t+1)}.
  \]
  This leads to $\beta^{(t+1)} = \tilde\beta^{(t+1)}$, which in turn implies that $p^{(t+1)}=\tilde p^{(t+1)}$.
\end{proof}

\begin{proof}[Proof of Lemma~\ref{lemma:well_known_props}]
  Recall that the iterates $v^{(0)}, \ldots, v^{(\Tin)}$ are the same as those generated by the classical CG method~\citep{Hestenes1952CG} applied to $m$ and initialized at $v^{(0)} = \xi$. % Indeed, the difference between conjugate gradients (CG) and Algorithm~\ref{algo:pTR} only appears for iterates that lie outside the ball $B(0, \Delta / 2)$.
  Items~\ref{item:pt_basis}-\ref{item:kt} are known fundamental properties of CG \citep{Hestenes1952CG,Conn2000}.
  To prove item~\ref{item:grad_decrease}, let $t$ satisfy $1 \leq t \leq \Tin-1$.
  From item~\ref{item:kt} we have
  \[
    m(v^{(t+1)}) \leq m(v),\quad \forall v \in v^{(0)} + \mathcal{K}_{t+1}(H, r^{(0)}).
  \]
  It is true in particular if we choose $v$ to be $v^{(t+1)}_{\rm GD} = v^{(t)} + \frac{1}{L_G} r^{(t)}$ (which is indeed in the right space).
  Since $r^{(t)} = -\nabla m(v^{(t)})$, the point $v^{(t+1)}_{\rm GD}$ corresponds to a gradient descent step with step size $1/L_G$.
  As the function $m$ has Lipschitz continuous gradients with constant $L_G$, a classical result from smooth optimization (see, e.g., \cite[\S1.2.3]{Nesterov20018}) states that
  \[
    m(v^{(t+1)}_{\rm GD}) \leq m(v^{(t)}) - \frac{1}{2L_G}\|r^{(t)}\|^2.
  \]
  Combining the two previous inequalities yields
  \[
    m(v^{(t+1)}) - m(v^{(t)}) \leq m(v^{(t+1)}_{\rm GD}) - m(v^{(t)}) \leq-\frac{1}{2L_G} \|r^{(t)}\|^2.
  \]

  To prove item~\ref{item:psd_kt}, fix $1 \leq t \leq \Tin$ and a nonzero $u \in \mathcal{K}_t(H,r^{(0)})$.
  We know from \cite[\S5.1]{Conn2000} that the vectors $p^{(0)}, \ldots, p^{(t-1)}$ form a basis of this space and that they are $H$-conjugate, meaning that $\la Hp^{(s)},p^{(s')} \ra = 0$ for $s\neq s'$.
  We can write $u = \sum_{s = 0}^{t-1} z_s p^{(s)}$ with at least one $z_s$ which is nonzero.
  Then,
  \[
    \la Hu,u \ra = \sum_{s=0}^{t-1} z_s^2 \la Hp^{(s)},p^{(s)} \ra >0,
  \]
  where the first equality follows from $H$-conjugacy and the inequality holds because the algorithm did not stop before iteration $t$ (recall that Algorithm~\ref{algo:tCG} terminates as soon as $\la  Hp^{(t)},p^{(t)}\ra \leq 0$).
  This proves that $H$ is positive definite on the $t$th Krylov subspace.

  To establish~\ref{item:steplength}, we rely on~\citep[Lem.~7]{Royer2020}, which states for $1 \leq t \leq \Tin$ that
  \[
    \la r^{(t-1)}, H r^{(t-1)}\ra \geq \la  p^{(t-1)}, H p^{(t-1)} \ra.
  \]
  From there, we find
  \[
    \alpha^{(t)} = \frac{\|r^{(t-1)}\|^2}{\la H p^{(t-1)},p^{(t-1)} \ra} \geq \frac{\|r^{(t-1)}\|^2}{\la H r^{(t-1)},r^{(t-1)} \ra} \geq \frac{1}{L_G}.
  \]
  (Note that \citet{Royer2020} assume that the method is initialized at 0, however the result also holds in our case thanks to Lemma~\ref{lemma:cg_equiv}.)

  Finally, we prove $T \leq d$.
  For contradiction, assume $\tCGbg$ runs $d$ (or more) iterations.
  In particular, it produces iterates $v^{(0)}, \ldots, v^{(d)}$.
  If $\|v^{(d)}\| = \Delta/2$, then the algorithm terminates with $\stopcrit = \stopoob$ and $T = d$.
  Otherwise, we know $\|v^{(d)}\| < \Delta/2$ so that $\Tin \geq d$.
  Then, item~\ref{item:pt_basis} provides $\mathcal{K}_d(H, r^{(0)}) = \reals^d$, while item~\ref{item:psd_kt} reveals that $H$ is positive definite on $\Rd$, and hence item~\ref{item:kt} provides that $v^{(d)}$ is the critical point of $m$.
  It follows that $r^{(d)} = 0$ (by item~\ref{item:res}) and so the algorithm terminates with $\stopcrit = \stopres$ and $T = d$.
\end{proof}

\section{Additional lemmas} \label{app:technical}

The first lemma of this appendix is a helper for the lemma thereafter.
%In this appendix, $S^{d-1} = \{ x \in \Rd : \|x\| = 1 \}$ is the unit sphere in $\Rd$, and $\mathrm{Unif}(S^{d-1})$ denotes the uniform distribution on that sphere.

\begin{lemma}\label{lemma:distribution_s}
  Fix $d \geq 2$.
  Let $q \in \Rd$ have unit norm, and let $\bar\xi$ be sampled uniformly at random from the unit sphere in $\Rd$.
  %$\bar \xi \sim \mathrm{Unif}(S^{d-1})$ and $q \in S^{d-1}$.
  Then, the random variable $|\la q, \bar \xi\ra|$ has density
  $$
    \rho(t) = 2a_d (1-t^2)^{(d-3)/2} \qquad \mbox{for } t \in [0, 1)
  $$
  and $\rho(t) = 0$ for $t \notin [0, 1)$, where the normalizing constant $a_d$ satisfies
  $$
    a_d = \frac{\Gamma(d/2)}{\Gamma(1/2) \Gamma((d-1)/2)} \leq \sqrt{\frac{d}{2\pi}}.
  $$
  % NB checked numerically: see density_sphere_inner_product.m
\end{lemma}
\begin{proof}
  The variable $\la q,\bar \xi \ra^2$ is beta-distributed with parameters $\left(\frac{1}{2}, \frac{d-1}{2}\right)$ \citep[Sec. 3]{dixon1983}.
  Its density is given by $t \mapsto a_d t^{-1/2} (1-t)^{(d-3)/2}$ on $(0,1)$.
  The expression of the density of $|\la q,\bar \xi \ra|$ follows from a change of variable.
\end{proof}

The next lemma is a probabilistic statement used in the proofs of Lemma~\ref{lemma:reslowerbound} and Theorem~\ref{thm:global_inner_its}.

\begin{lemma}\label{lemma:concentration_xi}
  Let $z\in \reals^d$, $c \in \reals$ and $\delta \in (0,1)$.
  Assume $d \geq 3$.
  With probability at least $1-\delta$, the vector $\xi_k$ generated at iteration $k$ of Algorithm~\ref{algo:pTR} (line~\ref{line:xik}) satisfies
  \[
    \big| \la z, \xi_k \ra + c \, \big| \geq \delta \|z\| \|\xi_k\| \sqrt{\frac{\pi}{8d}}.
  \]
\end{lemma}
\begin{proof}
  If $z = 0$, the claim is clear.
  If $z \neq 0$, then $|\inner{z}{\xi_k} + c\,| = \big|\inner{z/\|z\|}{\xi_k} + c/\|z\| \big| \cdot \|z\|$.
  Thus, assume $\|z\| = 1$ without loss of generality.

  Recall that $\xi_k = \|\xi_k\| \cdot e_k \bar \xi_k$ where $e_k \in \{\pm 1\}$ and $\bar \xi_k$ is sampled uniformly at random from the unit sphere in $\Rd$.
  For $\epsilon > 0$ we have
  \begin{align*}
    \Prob\!\left( |\la z,\xi_k\ra + c| \leq \epsilon \right)
    &\leq \Prob\!\left( \big||\la z,\xi_k\ra| - |c|\big| \leq \epsilon \right),
  \end{align*}
  because (for any two reals $a, b$) it holds that $|a + b| \geq |a| - |b|$ and $|a + b| \geq |b| - |a|$, so that $|a + b| \geq \big| |a| - |b| \big|$.
  On the right-hand side, replace $\xi_k$ by $\|\xi_k\| \cdot e_k \bar \xi_k$ and divide by $\|\xi_k\|$ to deduce:
  \begin{align*}
    \Prob\!\left( |\la z,\xi_k\ra + c| \leq \epsilon \right)
      & \leq \Prob\!\left( \left| \big|\la z, \bar \xi_k\ra\big| - \frac{|c|}{\|\xi_k\|}\right| \leq \frac{\epsilon}{\|\xi_k\|} \right)
        = \int_{\frac{|c|-\epsilon}{\|\xi_k\|}}^{\frac{|c|+\epsilon}{\|\xi_k\|}} \rho(t)dt,
  \end{align*}
  where $\rho$ is the density defined in Lemma~\ref{lemma:distribution_s}.
  For $d\geq 3$, $\rho$ is maximal at $0$ so that
  \[
    \Prob\!\left( |\la z,\xi_k\ra + c| \leq \epsilon \right)
    \leq \frac{2\epsilon}{\|\xi_k\|} \rho(0) \leq \frac{4\epsilon}{\|\xi_k\|} \sqrt{\frac{d}{2\pi}}.
  \]
  The right-hand side equals $\delta$ upon setting $\epsilon = \frac{1}{4}\delta \|\xi_k\| \sqrt{\frac{2\pi}{d}}$.
\end{proof}

The final lemma of this appendix is referenced in Remark~\ref{remark:sigma_bar1}.

\begin{lemma}\label{lemma:sigma_loglog}
  Let $a,b > 0$.
  Then for any $\sigma > 0$ we have
  \[
    \sigma \leq \frac{b}{\log_2\!\left( 2 + \frac{a}{b}\right)} \qquad \implies \qquad \sigma \log_2\log_2\!\left(2 + \frac{a}{\sigma} \right) \leq b.
  \]
\end{lemma}
\begin{proof}
  We first prove the auxiliary inequality
  \begin{equation} \label{eq:aux_ineq}
    \forall t \geq 0, \qquad \log_2\!\big( 2 + t \log_2(2+t) \big) \leq 2+t.
  \end{equation}
  % % Matlab check:
  % t = linspace(0, 20, 1001);
  % lhs = log2( 2 + t .* log2(2 + t) );
  % rhs = 2 + t;
  % plot(t, lhs, t, rhs);
  %
  Equivalently, we should prove that $\phi(t) = 2^{2+t} - 2 - t \log_2(2+t)$ is nonnegative for $t\geq 0$.
  We have
  \begin{align*}
    \phi'(t) &= (\log 2) 2^{2+t} - \log_2(2+t) - \frac{t}{(\log 2)(2+t)}, \\
    \phi''(t) &= (\log 2)^2 2^{2+t} - \frac{2}{(\log 2)(2+t)} + \frac{t}{(\log 2) (2+t)^2}.
  \end{align*}
  We claim that $\phi$ is convex: indeed we have $\phi''(t) \geq (\log 2)^2 2^2 - \frac{2}{(\log 2) 2} \approx 0.48$ for $t\geq 0$.
  Then, since $\phi(0) = 2 > 0$ and $\phi'(0) = 4\log 2 - 1 > 0$, we deduce that $\phi(t) > 0 $ for $t \geq 0$.
  This proves inequality~\eqref{eq:aux_ineq}.

  We now prove the lemma.
  Assume $\sigma$ satisfies the prerequisites.
  Then we can write
  \begin{align*}
    \sigma = \frac{b}{\log_2\!\left(2 + z \right)} && \textrm{ with } && z \geq \frac{a}{b}.
  \end{align*}
  We therefore have
  \[
  \begin{split}
    \sigma \log_2 \log_2\!\left( 2 + \frac{a}{\sigma} \right) &= \frac{b}{\log_2(2+z)} \cdot  \log_2 \log_2\!\left( 2 + \frac{a}{b} \log_2(2+z) \right)\\
    &\leq \frac{b}{\log_2(2+z)} \cdot  \log_2 \log_2\!\left( 2 + z \log_2(2+z) \right)\\
    &\leq \frac{b}{\log_2(2+z)} \cdot  \log_2(2+z)\\
    &= b,
  \end{split}
  \]
  owing to inequality~\eqref{eq:aux_ineq}.
\end{proof}

\section{Connection between formal and informal statements of the theorem}\label{app:informal_justif}

In this section, we justify that the informal presentation of the main theorems in Section \ref{sss:informal_thm} is indeed a simplification of Theorems~\ref{thm:outer_complexity} and~\ref{thm:global_inner_its}.
First, Theorem~\ref{thm:outer_complexity} states that with probability at least $1-\delta$, the method finds a point $x_K$ that is $\epsilon$-close to a local minimizer $x^*$ in at most $K$ outer iterations with
\begin{equation}\label{eq:boundK}
  K \leq \bar K_{G,N} + \bar K_{M,\epsilon} = \frac{3}{2} \left( \frac{f(x_0)-\flow}{\Flg} + 1 \right) K_{\rm esc}(\delta',\sigma) + \frac{1}{2} \log_2\!\left( \frac{\Delta_0}{8\bar R} \right) + \log_2\!\left(1 + \left[\log_2 \frac{\bar R}{\epsilon} \right]_+\right).
\end{equation}
Here $\Flg = \frac{\rho'}{5 L_G} \gradlb^2$ is defined as in \eqref{eq:def_constants} and
\begin{align}
    \delta' = \left(\frac{\Flg}{f(x_0) - \flow+\Flg} \right)\delta \;\;\in (0,1) && \textrm{ and } &&
    K_{\rm esc}(\delta', \sigma) = 2
          + \log_2 \log_2\!\left(2+\frac{\sqrt{d}}{ \omega_2 \mu \delta'\sigma}\right).
\end{align}
So $\delta'$ can be rewritten as
$$
    \delta' = \dfrac{\delta}{\kappa + 1} \quad \textrm{ with } \quad \kappa = \dfrac{f(x_0) - \flow}{\Flg}.
$$
The second term in \eqref{eq:boundK} is hidden in the $\tildeO$ notation in \eqref{eq:bound_outer_informal} and the third term matches exactly the second term in \eqref{eq:bound_outer_informal} since $\epsilon \leq \bar R$. For the first term, we have
\begin{align*}
  \left( \kappa + 1 \right) K_{\rm esc}(\delta',\sigma) &= \left( \kappa + 1 \right) \left( 2 + \log_2 \log_2\!\left(2+\frac{\sqrt{d}(\kappa + 1)}{ \omega_2 \mu \delta\sigma}\right) \right) \\
  & = \mathcal{O}\left(\kappa \log \log\!\left(\frac{\sqrt{d}(\kappa + 1)}{ \omega_2 \mu \delta\sigma}\right)\right) = \tildeO\left(\kappa \log \log\!\left(\frac{d}{ \delta\sigma}\right)\right),
\end{align*}
where we absorbed the $\log\log$ dependence on $\kappa$, $\omega_2$, $\mu$ in the $\tildeO$ notation.

For the inner iterations, Theorem~\ref{thm:global_inner_its} implies that for each outer iteration $k$, the number of inner iterations $T_k$ during that outer iteration satisfies
\begin{equation}\label{eq:bound_Tk}
  \begin{aligned}
    T_k &\leq 4 + \frac{ ( L_G \bar \Delta)^2}{\left(\min(\omega_1 \gradlb, \omega_2 \gradlb^2) \right)^2}
  + \sqrt{
      \frac{L_G}{4\mu}
      }
      \log\!\left( \frac{512 L_G^2 {\bar{R}}^2 \bar K_{G,N}^2 d  }{\pi \sigma^2 \delta^2 \mu^2}\right) \\
      &+ \sqrt{ \frac{2L_G}{\mu}}
        \log\!\left(
        \sqrt{\frac{2L_G}{\mu}} \cdot
        \frac{
          16L_G\bar R
        }{
          \min(\omega_1 \mu \epsilon , \omega_2 \mu^2 \epsilon^2)
        }
        \right).
  \end{aligned}
\end{equation}
The first term matches exactly the second term in \eqref{eq:bound_inner_informal}. For the second and third terms, we note that $\bar K_{G,N} = \tildeO(\kappa \log\log(d/(\delta \sigma)))$ from above. After hiding logarithmic dependence on $\omega_1, \omega_2, \mu, L_G, \bar R, \kappa$ in the $\tildeO$ notation, we obtain that $T_k$ is upper bounded by the second term in \eqref{eq:bound_inner_informal}. Finally, summing the bound on $T_k$ for $k \leq K$ gives the desired result.

\section*{Acknowledgments}

This work was supported by the Swiss State Secretariat for Education, Research and Innovation (SERI) under contract number MB22.00027.
RD is a chair holder of the Hi! Paris interdisciplinary research center.

{\small \bibliographystyle{plainnat_modif}
\bibsep 1ex
\bibliography{library,library_supp}}

\end{document}